\documentclass{amsart}

\usepackage{mathtools,amsthm}
\usepackage[letterpaper, margin=1.3in]{geometry}
\usepackage[utf8]{inputenc}
\usepackage[T2A,T1]{fontenc}
\usepackage[pagebackref]{hyperref}
\usepackage{enumitem}
\usepackage{mathrsfs}
\usepackage{tikz-cd}
\usepackage{amssymb}
\usepackage{xspace}
\usepackage{nicefrac}
\usepackage{multicol}
\usepackage{enumitem}
\usepackage{xstring}
\usepackage{etoolbox}
\usepackage{aliascnt}
\usepackage[nameinlink]{cleveref}
\usepackage{setspace}
\usepackage{stmaryrd}
\SetSymbolFont{stmry}{bold}{U}{stmry}{m}{n}
\usepackage{version}
\usepackage{environ}
\usepackage[titletoc]{appendix}
\usepackage{xparse}
\usepackage{pstricks}
\usepackage{amsfonts}
\usepackage{mathdots}

\hypersetup{colorlinks={true},linkcolor={blue},urlcolor=blue}
\usepackage{pdflscape}

\newtheorem{theorem}{Theorem}[subsection]
\newtheorem*{theorem*}{Theorem}
\newtheorem{lemma}[theorem]{Lemma}

\newtheorem{problem*}{Problem}
\newtheorem{lemma*}{Lemma}
\newtheorem{corollary}[theorem]{Corollary}
\newtheorem{proposition}[theorem]{Proposition}

\newtheorem{conjecture}[theorem]{Conjecture}
\newtheorem*{conjecture*}{Conjecture}

\newtheorem{innercustomgeneric}{\customgenericname}
\providecommand{\customgenericname}{}
\newcommand{\newcustomtheorem}[2]{%
  \newenvironment{#1}[1]
  {%
   \renewcommand\customgenericname{#2}%
   \renewcommand\theinnercustomgeneric{\kern-0.3em ##1}%
   \innercustomgeneric
  }
  {\endinnercustomgeneric}
}

\newcustomtheorem{specialtheorem}{}
\crefname{specialtheorem}{}{}

\theoremstyle{definition}
\newtheorem{definition}[theorem]{Definition}
\newtheorem{notation}[theorem]{Notation}
\newtheorem{example}[theorem]{Example}

\theoremstyle{remark}
\newtheorem{remark}[theorem]{Remark}
\newtheorem*{remark*}{Remark}

\crefname{theorem}{Theorem}{Theorems}
\crefname{lemma}{Lemma}{Lemmas}
\crefname{corollary}{Corollary}{Corrolaries}
\crefname{proposition}{Proposition}{Propositions}
\crefname{definition}{Definition}{Definitions}
\crefname{remark}{Remark}{Remarks}

\crefname{appsec}{Appendix}{Appendices}

\usepackage{tikz}
\usetikzlibrary{chains,shapes.symbols}

\tikzset{commutative diagrams/.cd,
mysymbol/.style={start anchor=center,end anchor=center,draw=none}
}
\tikzset{zshift/.style={xshift={-0.3*#1},yshift={-1.2*#1}}}
\newcommand\Symbol[2]{\arrow[mysymbol]{#1}[description]{#2}}

\DeclareMathOperator{\Spec}{Spec}
\DeclareMathOperator{\Spf}{Spf}
\DeclareMathOperator{\Lie}{Lie}
\let\Im\relax
\DeclareMathOperator{\Im}{Im}
\newcommand{\pDiv}[2]{\mathrm{pDiv}_{#1}^{#2}}
\newcommand{\X}{\mathbb{X}}
\newcommand{\N}[2]{\mathcal{N}_{#1}^{#2}}
\newcommand{\E}{\mathbb{E}}
\newcommand{\ZZ}{\mathcal{Z}}
\newcommand{\FJ}{\mathrm{FJ}}
\newcommand{\Int}{\mathrm{Int}}
\newcommand{\xdashrightarrow}[1]{\overset{#1}{\dashrightarrow}}
\renewcommand{\red}{\mathrm{red}}
\newcommand{\cplx}{\mathbf{c}}
\newcommand{\rflx}{\mathrm{rflx}}
\newcommand{\Sch}[1]{\mathrm{Sch}_{/#1}}

\newcommand{\sch}[1]{\mathrm{Sch}'_{/#1}}
\newcommand{\Ab}[1]{\mathrm{Ab}_{\O_F}^{#1}}
\newcommand{\UAb}[1]{\mathrm{Ab}_{\O_F/\O_{F^+}}^{#1}}
\newcommand{\op}{\mathrm{op}}
\newcommand{\IAb}[1]{{}^0\mathrm{Ab}_{\O_F}^{#1}}
\newcommand{\IUAb}[1]{{}^0\mathrm{Ab}_{\O_F/\O_{F^+}}^{#1}}
\newcommand{\Sh}{\mathrm{Sh}}
\newcommand{\Set}{\mathrm{Set}}
\let\Mod\relax
\DeclareMathOperator{\Mod}{\mathrm{Mod}}
\newcommand{\RR}{\mathrm{R}}
\newcommand{\unr}{\mathrm{unr}}
\newcommand{\sing}{\mathrm{sing}}
\newcommand{\lr}{\mathrm{lr}}

\let\bigstar\relax
\DeclareMathOperator*{\bigstar}{\raisebox{-5.5pt}{\scalebox{3.45}{$\star$}}}
\DeclareMathOperator{\im}{im}

\DeclareMathOperator{\coker}{coker}

\newcommand{\V}{\mathscr{V}}
\newcommand{\Z}{\mathbb{Z}}
\newcommand{\Q}{\mathbb{Q}}
\newcommand{\R}{\mathbb{R}}
\renewcommand{\C}{\mathbb{C}}
\newcommand{\F}{\mathbb{F}}
\newcommand{\p}{\mathfrak{p}}
\newcommand{\q}{\mathfrak{q}}
\newcommand{\m}{\mathfrak{m}}
\renewcommand{\O}{\mathcal{O}}
\newcommand{\n}{\mathfrak{n}}
\renewcommand{\L}{\mathbb{L}}
\newcommand{\Hom}{\mathrm{Hom}}
\newcommand{\rightiso}{\xrightarrow{\sim}}
\newcommand{\iso}{\simeq}
\newcommand{\xrightiso}[1]{\xrightarrow[\raisebox{1pt}{$\sim$}]{#1}}
\newcommand{\Typ}{\mathrm{Typ}}
\newcommand{\typ}{\mathrm{typ}}
\newcommand{\defeq}{\vcentcolon=}
\renewcommand{\d}[1]{\mathop{d#1}}
\newcommand{\et}{\text{\'et}}
\newcommand{\qbinom}[3]{\genfrac{[}{]}{0pt}{}{#1}{#2}_{#3}}
\newcommand{\blank}{{-}}

\newcommand{\PSch}{\mathrm{PSch}'}
\newcommand{\A}{\mathbb{A}}

\usepackage{bm}

\usepackage[bbgreekl]{mathbbol}
\DeclareMathSymbol\bbDelta \mathord{bbold}{"01}
\newcommand{\bbnabla}{\raisebox{\depth}{\scalebox{1}[-1]{$\bbDelta$}}}
\DeclareSymbolFontAlphabet{\mathbbm}{bbold}
\DeclareSymbolFontAlphabet{\mathbb}{AMSb}%
\DeclareFontFamily{OT1}{pzc}{}
\DeclareFontShape{OT1}{pzc}{m}{it}{<-> s * [1.10] pzcmi7t}{}
\DeclareMathAlphabet{\mathpzc}{OT1}{pzc}{m}{it}

\DeclareFontFamily{U}{matha}{\hyphenchar\font45}
\DeclareFontShape{U}{matha}{m}{n}{
      <5> <6> <7> <8> <9> <10> gen * matha
      <10.95> matha10 <12> <14.4> <17.28> <20.74> <24.88> matha12
      }{}
\DeclareSymbolFont{matha}{U}{matha}{m}{n}
\DeclareMathSymbol{\obot}{2}{matha}{"6B}

\begin{document}

\title{First explicit reciprocity law for unitary Friedberg--Jacquet periods}
\author{Murilo Corato-Zanarella}

\begin{abstract}
Consider a unitary group $G(\mathbb{A}_{F^+})=U_{2r}(\mathbb{A}_{F^+})$ over a CM extension $F/F^+$ with $G(\mathbb{A}_\infty)$ compact. In this article, we study the Beilinson--Bloch--Kato conjecture for motives associated to irreducible cuspidal automorphic representations $\pi$ of $G(\mathbb{A}_{F^+}).$ We prove that if $\pi$ is distinguished by the unitary Friedberg--Jacquet period, then the Bloch--Kato Selmer group (with coefficients in a favorable field) of the motive of $\Pi=\mathrm{BC}(\pi)$ vanishes.
\end{abstract}
\maketitle
\setcounter{tocdepth}{2}
\let\oldtocsection=\tocsection
\let\oldtocsubsection=\tocsubsection
\let\oldtocsubsubsection=\tocsubsubsection
\renewcommand{\tocsection}[2]{\hspace{0em}\oldtocsection{#1}{#2}}
\renewcommand{\tocsubsection}[2]{\hspace{1em}\oldtocsubsection{#1}{#2}}
\renewcommand{\tocsubsubsection}[2]{\hspace{2em}\oldtocsubsubsection{#1}{#2}}
\tableofcontents
\section{Introduction}
\subsection{Main result}
Let $F/F^+$ be a CM extension of number fields and $r\ge1.$ We study the Beilinson--Bloch--Kato conjecture for motives associated to irreducible cuspidal automorphic representations $\pi$ of $G(\mathbb{A}_{F^+})=U_{2r}(\mathbb{A}_{F^+})$ which are distinguished by a subgroup of the form $H(\mathbb{A}_{F^+})=U_r(\mathbb{A}_{F^+})\times U_r(\mathbb{A}_{F^+}).$ Here, all unitary groups are compact at the real places.

The automorphic periods $\mathcal{P}_H(f)=\int_{[H]}f(h)\d{h}$ for $f\in\pi$ are a variant of the linear periods studied in \cite{FJ}, and are called the \emph{unitary Friedberg--Jacquet periods}. It is expected that such periods are closely related to the central values $L(\Pi,\frac{1}{2})$ where $\Pi=\mathrm{BC}(\pi)$ is the base change of $\pi$ to $\mathrm{GL}_{2r}(\mathbb{A}_F).$ More precisely:
\begin{conjecture}[Xiao--Zhang]\label{Xiao-Zhang}
    Let $\pi$ be tempered cuspidal automorphic representation. Then $\mathcal{P}_H\rvert_\pi\neq0$ if and only if
    \begin{enumerate}
        \item $\Pi$ is symplectic, i.e. the exterior square $L$-function $L(\Pi,\bigwedge^2,s)$ has a pole at $s=1,$
        \item $\pi$ is locally distinguished by $H,$ i.e. $\Hom_{H(\A_{F^+}^\infty)}(\pi,\C)\neq0,$
        \item the standard $L$-function $L(\Pi,s)$ is nonvanishing at $s=\frac{1}{2}.$
    \end{enumerate}
\end{conjecture}

There has been a lot of recent progress on this conjecture such as in Leslie \cite{Leslie1,Leslie2}, Pollack--Wan--Zydor \cite{PWZ}, Chen--Gan \cite{Chen-Gan} and Leslie--Xiao--Zhang \cite{LXZ}.

In light of the Beilinson--Bloch--Kato conjecture, this lead us to the expectation that
\begin{equation*}
    \mathcal{P}_H\rvert_\pi\neq0\implies H^1_f(F,\rho_{\Pi,\lambda}(r))=0
\end{equation*}
where $H^1_f(F,\rho_{\Pi,\lambda}(r))$ denotes a Block--Kato Selmer group for the Galois representation $\rho_{\Pi,\lambda}$ attached to $\Pi$ (\Cref{GaloisExistence}). In this article, we verify some cases of this expectation.
\begin{theorem}[{\Cref{ThmA}}]\label{ThmAIntro}
    Suppose $F^+\neq\Q.$ Let $\pi$ be an irreducible cuspidal automorphic representation of $G(\mathbb{A}_{F^+})$ of weight $(0,\ldots,0)$ such that its base change $\Pi=\mathrm{BC}(\pi)$ is an irreducible cuspidal automorphic representation of $\mathrm{GL}_{2r}(\mathbb{A}_F).$ Let $E\subseteq\C$ be a strong coefficient field of $\Pi$ (\Cref{StrongCoefDef}). If $\mathcal{P}_H\rvert_\pi\neq0,$ then for all admissible primes $\lambda$ of $E$ with respect to $\Pi,$ the Bloch--Kato Selmer group $H^1_f(F,\rho_{\Pi,\lambda}(r))$ vanishes.
\end{theorem}
\begin{remark}\label{IntroductionRemark}
We make some remarks about the above assumptions.
\begin{enumerate}[leftmargin=*]
    \item The assumption $F^+\neq\Q$ can be lifted once \cite[Hypothesis 3.2.10]{LTXZZ} is known for $N\ge4$ and $F^+=\Q.$
    \item The notion of admissible primes is given in \Cref{admissibledef}, and consists of a long list of assumptions:
    \begin{enumerate}
        \item[{\hyperref[L1]{(L1-2)}}] Are elementary, and exclude only finitely many primes $\lambda.$
        \item[\ref{L3}] Is expected to hold for all but finitely many primes $\lambda.$
        \item[\ref{L4}] Is a big image assumption on the residual Galois representation attached to $\Pi.$ It is expected to hold for  all but finitely many primes $\lambda$ as long as $\Pi$ is not a transfer from a smaller group.
        \item[\ref{L5}] Is a technical assumption only used in an $R=T$ argument in \cite{LTXZZ2}, and it is expected to hold for all but finitely many primes.
        \item[\ref{L6}] Is a technical assumption to ensure the vanishing of certain Hecke localized cohomology of unitary Shimura varieties off middle degree.
    \end{enumerate}
    In short, as long as $\Pi$ is not a transfer from a smaller subgroup, we expect all but finitely many primes to be admissible. 
    \item See \Cref{admissibleDiscussion} for a discussion of when the above conditions are known to hold for all but finitely many primes. In particular, all but finitely many primes are admissible in the following two situations:
    \begin{enumerate}[label=(\alph*)]
        \item if $F^+\neq\Q,$ $E=\Q$ and there is an elliptic curve $A$ over $F^+$ such that $\mathrm{End}(A_{\overline{F}})=\Z$ and $\rho_{\Pi,\ell}\iso\mathrm{Sym}^{2r-1}H^1_\et(A_{\overline{F}},\Q_\ell)\rvert_{\Gamma_F}$;
        \item if $F^+\neq\Q$ and there are two nonarchimedean places $\p$ and $w$ of $F$ such that $\p$ is very special inert (\Cref{DefSpecialInert}), $\Pi_\p$ is Steinberg and $\Pi_w$ is supercuspidal.
    \end{enumerate}
\end{enumerate}
\end{remark}

Together with the cases of \Cref{Xiao-Zhang} verified by Leslie--Xiao--Zhang \cite[Theorem 1.1, Remark 1.1, Remark 1.2]{LXZ}, we obtain the following special cases of the Beilinson--Bloch--Kato conjecture.
\begin{theorem}\label{ThmB}
    Let $\pi$ be an irreducible cuspidal automorphic representation of $\mathrm{GL}_{2r,F^+}$ of symplectic type. We assume the following:
    \begin{enumerate}
        \item $F/F^+$ is everywhere unramified.
        \item $F/F^+$ splits over every finite place $v$ of $F^+$ such that $p\le \max\{e(v),2\}$ where $e(v)$ denotes the absolute ramification index of $v.$
        \item There exists a split non-archimedean place $v$ such that $\pi_{v}$ is supercuspidal and distinguished by $\mathrm{GL}_{r,F^+}\times\mathrm{GL}_{r,F^+}.$
        \item $\pi$ is unramified at each non-split non-archimedean place of $F^+.$
        \item For each archimedean place $v$ of $F^+,$ $\pi_v$ is a unitary representation with trivial infinitesimal character.
    \end{enumerate}
    Denote $\Pi=\mathrm{BC}(\pi)$ the base change of $\pi,$ which is an irreducible cuspidal automorphic representation of $\mathrm{GL}_{2r,F},$ and $E\subseteq\C$ a strong coefficient field (\Cref{StrongCoefDef}) of $\Pi.$ If $L(\Pi,\frac{1}{2})\neq0,$ then for all admissible primes $\lambda$ of $E$ we have that the Bloch--Kato Selmer group $H^1_f(F,\rho_{\Pi,\lambda}(r))$ vanishes.
\end{theorem}
\begin{remark}
    Recently, Peng \cite{Peng} has obtained results towards the Bloch--Kato conjecture for conjugate self-dual motives which include that of \Cref{ThmB}, by relying on the Gan--Gross--Prassad case studied in \cite{LTXZZ}. However, we note that our (proof of) \Cref{ThmAIntro} has significance beyond this application: this is one of two ingredients (the other being a ``second reciprocity law'') to obtain divisibilities towards an Iwasawa main conjecture in the unitary Friedberg--Jacquet setting, similarly to \cite{Bertolini-Darmon} and \cite{LTX}. We hope to study this in the future.
\end{remark}

\subsection{Strategy and new obstacles}
We attack this problem via level-raising congruences, following the methods first introduced in \cite{Bertolini-Darmon} in the case of Heegner points. Such methods have seen a resurgence in recent years, having been used to prove striking results towards the Beilinson--Bloch--Kato conjecture \cite{LiuTwisted,LiuTriple,Liu-Tian,Haining1,Haining2,LTXZZ,Sweeting}.

There are, however, some new interesting obstacles that we have to face in the unitary Friedberg--Jacquet setting. To discuss these, we first recall the basic prototype for the argument. One establishes a congruence modulo $\p$ (up to certain Hecke operators) between the period $\mathcal{P}_H$---which lives in a Shimura set $\mathrm{Sh}_K$---and (the singular part of the localization of) the cohomology class arising from a special cycle $Z_H$ in arithmetic middle dimension on a ``nearby-at-$\p$'' Shimura variety $\mathrm{Sh}_{K^{\text{nearby}}}^{\text{nearby}}$ of dimension $2r-1.$ This comparison happens via an \emph{arithmetic level-raising} map, whose computation involves analyzing the reduction modulo $\p$ of the cycle $Z_H.$

There are two (related) new obstacles to carrying this plan out in the present setting:
\begin{enumerate}[leftmargin=*]
\item For $r>1,$ the ``naive'' guess for the cycle $Z_H,$ namely the cycle induced by an embedding of Shimura data of the form
\begin{equation*}
    U(r-1,1)\times U(r,0)\hookrightarrow U(2r-1,1),
\end{equation*}
does not map to $\mathcal{P}_H$ under the arithmetic level raising. The discrepancy between $\mathcal{P}_H$ and the image of $Z_H$ happens only at $\p,$ and we are led to study the following local question: the cycle $Z_H$ determines a function in $C_c^\infty(H(F^+_\p)\backslash G(F^+_\p)/K_\p),$ and we are trying to see if this is a multiple of the basic function under the local Hecke algebra $C_c^\infty(H(F^+_\p)\backslash G(F^+_\p)/K_\p).$ This discrepancy can now be explained by the fact that $C_c^\infty(H(F^+_\p)\backslash G(F^+_\p)/K_\p),$ is a free $C_c^\infty(H(F^+_\p)\backslash G(F^+_\p)/K_\p)$-module of rank $2^{r-1}.$ In comparison, in most previous works on bipartite Euler systems, the corresponding space of spherical functions is free of rank $1$ over the Hecke algebra.\footnote{Notably, the setting of \cite{Sweeting} also lacks this local multiplicity one.} In order to bypass this problem, we need to consider a larger supply of special cycles, which will be certain Hecke translates (at $\p$) of the ``naive'' cycle.
\item It is hard to understand the reduction of $Z_H$ modulo $\p,$ since its natural moduli-theoretic integral model is not equidimensional. Furthermore, the situation becomes even worse when considering Hecke translates of $Z_H$ as above. To overcome this, we instead consider certain \emph{derived integral models} ${}^{\mathbb{L}}\mathcal{Z}$ of such special cycles, and we explain how one can still compute the arithmetic level raising map with such derived integral models.
\end{enumerate}
More precisely, for a suitable space $\mathrm{Param}$ parametrizing the derived integral models ${}^{\mathbb{L}}\mathcal{Z},$ we will establish (a version with integral coefficients of) the commutative diagram
\begin{equation*}
    \begin{tikzcd}
        &\mathrm{Param}=\bigotimes'_v\mathrm{Param}_v\arrow[dr,"\phi"]\arrow[dl]&\\
        C_c^\infty(X^{\text{nearby}}(\A_{F^+,f}))^{K^{\text{nearby}}}\arrow[d,"\Theta_{X^{\text{nearby}}}"]&&C_c^\infty(X(\A_{F^+,f}))^K\arrow[d,"\Theta_X"]\\
        \mathrm{CH}^r(\mathrm{Sh}^{\text{nearby}}_{K^{\text{nearby}}})\arrow[rr,dashed,"\text{arithmetic level raising}"]&&\C[\mathrm{Sh}_K]
    \end{tikzcd}
\end{equation*}
where: i) $X$ is the spherical variety $X=H\backslash G$ and $\Theta_X\colon C_c^\infty(X(\A_{F^+,f}))\to\mathrm{Distr}(G(\A_{F^+,f}),\C)$ is the usual theta map, ii) $X^{\text{nearby}}$ is similarly defined in terms of the nearby spaces, and $\Theta_{X^{\text{nearby}}}$ is an arithmetic theta map, iii) the bottommost arrow is a shorthand for the arithmetic level raising, which is not a map as indicated but rather a certain congruence, which we defer to the body of the article for its definition. Then the strategy is to compute the image of $\phi,$ namely to prove that it contains a suitable Hecke-multiple of the basic function.

The first two rows in the above diagram are isomorphisms away from $\p,$ and the map
\begin{equation*}
    \phi_\p\colon\mathrm{Param}_\p\to C_c^\infty(X(F_\p^+))^{K_\p}
\end{equation*}
has a surprising interpretation in terms of local harmonic analysis, which is crucial for the computation of its image. The space $C_c^\infty(X(F_\p^+))^{K_\p}$ has both a relative Cartan decomposition and a relative Satake transform
\begin{equation*}
    \C[\Lambda_X^+]\xrightarrow[\sim]{\text{Cartan}}C_c^\infty(X(F_\p^+))^{K_\p}\xrightarrow[\sim]{\text{Satake}}\C[\Lambda_X]^{W_X}
\end{equation*}
where the composition $(\text{Satake}\circ\text{Cartan})^{-1}\colon \C[\Lambda_X]^{W_X}\rightiso\C[\Lambda_X^+]$ can be read off from the geometry of the variety $X,$ via the so-called ``inverse Satake transform'' of Sakellaridis \cite{Sakellaridis}. As for $\mathrm{Param}_\p,$ it is naturally of the form $\mathrm{Param}_\p=\C[\Lambda_X^{++}]\subset \C[\Lambda_X^+]$ for a certain subset $\Lambda_X^{++}\subset\Lambda_X^+,$ in such a way that $\phi$ naturally extends to an isomorphism
\begin{equation*}
    \phi_\p\colon \C[\Lambda_X^+]\rightiso C_c^\infty(X(F_\p^+))^{K_\p}.
\end{equation*}
The crucial observation is that this isomorphism behaves like a \emph{different} ``relative Cartan decomposition'', in the sense that the composition $(\text{Satake}\circ\phi_\p)^{-1}\colon\C[\Lambda_X]^{W_X}\rightiso\C[\Lambda_X^+]$ has formally the same shape as an inverse Satake transform.

\subsection{Outline of this article}
The next $4$ sections of this article are independent of each other.
\begin{itemize}
    \item In \Cref{WeightChapter}, we extend the results of \cite[Section 2]{LiuTriple} about the computations of ``arithmetic level raising maps'' to allow for the consideration of derived integral models. We expect that this may also be useful for constructing bipartite Euler systems in settings other than the present article.
    \item In \Cref{RZChapter}, we compute certain derived intersection numbers in Rapoport--Zink spaces. These show up through the $p$-adic uniformization of the unitary Shimura varieties we consider later. Such intersection numbers are the core of the computation of the arithmetic level raising for our special cycles.
    \item In \Cref{ShimuraChapter}, we introduce the integral models of (RSZ variants) of unitary Shimura varieties, and construct new (derived) integral models of special cycles. We study the generic fiber and the $p$-adic uniformization of such special cycles.
    \item In \Cref{ComputationChapter}, we study the local harmonic analysis related to the arithmetic level raising. This analysis is what ultimately will imply that the automorphic period $\mathcal{P}_H$ is (up to an appropriate Hecke translation) the image of a special cycle under the arithmetic level raising map.
\end{itemize}
In \Cref{ReciprocityChapter}, we use all the above ingredients together with the arithmetic level raising of \cite{LTXZZ} to establish the first reciprocity law. We refer to \Cref{Figure1,Figure2} for a schematic view of the proof of the reciprocity law. Finally, in \Cref{ProofChapter} we use such reciprocity law to perform the Euler system argument and prove our main theorem.

\Cref{KAppendix} recalls some definitions about $K$-theory and proves a lemma used in \Cref{WeightChapter}.

\subsection{Conventions and notations}
\begin{itemize}[leftmargin=*]
    \item We denote $(a\mapsto a^\cplx)\in\mathrm{Aut}(\C/\Q)$ the complex conjugation.
    \item For a field $K,$ we denote $\Gamma_K\defeq\mathrm{Gal}(K^{\mathrm{sep}}/K)$ the absolute Galois group of $K.$
    \item For a number field $F,$ we denote by $\A_F\defeq\bigotimes'_vF_v$ the ad\`eles of $F.$ For a finite set of places $\Sigma$ of $F,$ we denote $\A_F^S\defeq\bigotimes'_{v\not\in S}F_v$ the ad\`eles away from $S.$ For the case $F=\Q,$ we denote $\A\defeq\A_\Q$ and $\A^S\defeq\A_\Q^S.$ We also use the shorthand $\A_f^S\defeq\A^{S\cup\{\infty\}}.$
    \item For a rational prime $\ell,$ we say a $\Z_\ell$-ring $L$ is an \emph{$\ell$-adic coefficient ring} if it is finite over $\Z_\ell$ or over $\Q_\ell.$
    \item For $\lambda$ a nonarchimedean place of a number field $E,$ an $\O_{E_\lambda}$-module $M$ and $x\in M,$ we denote
    \begin{equation*}
        \exp_\lambda(x,\ M)\defeq\min\{d\in\Z_{\ge0}\cup\{\infty\}\colon\lambda^dx=0\}.
    \end{equation*}
    \item For a henselian discrete valuation field $K$ with inertia subgroup $I_K\subseteq\Gamma_K$ and an abelian group $M$ with continuous $\Gamma_K$-action, we denote the \emph{unramified cohomology}
    \begin{equation*}
        H^1_\unr(K,M)\defeq\ker(H^1(K,M)\to H^1(I_K,M)).
    \end{equation*}
    We also denote the \emph{singular cohomology} the quotient by the unramified cohomology
    \begin{equation*}
        H^1_\sing(K,M)\defeq H^1(K,M)/H^1_\unr(K,M),
    \end{equation*}
    with quotient map $\partial\colon  H^1(K,M)\twoheadrightarrow H^1_{\sing}(K,M)$
    \item For a (formal) scheme $S,$ we denote by $\Sch{S}$ the category of $S$-schemes. When $S$ is Noetherian, we denote by $\sch{S}$ the full subcategory of locally Noetherian $S$-schemes, and we also denote $\PSch{S}$ the category of presheaves on $\sch{S}.$
    \item For a noetherian scheme $X,$ we denote $\mathrm{Z}^i(X)$ the abelian group of algebraic cycles of codimension $i$ in $X.$ If $\ell$ is a rational prime which is invertible in $X$ and $\Lambda\in\{\Z_\ell,\Q_\ell,\Z/\ell^\nu\Z\text{ for }\nu\ge0\},$ we have an absolute cycle class map\footnote{See for example \cite[Section 3.3.1]{LiuTwisted} for a definition.} $\mathrm{cl}_{X,\ell}\colon \mathrm{Z}^i(X)\to H^{2i}_\et(X,\Lambda(i)).$
    \item If $X$ is finite type over a field $k,$ we denote by $\mathrm{CH}^i(X)$ the Chow group with respect to rational equivalence. For a rational prime $\ell\neq\mathrm{char}(k),$ the cycle class map factors via a map $\mathrm{cl}_{X,\ell}\colon\mathrm{CH}^i(X)\to H^{2i}_\et(X,\Z_\ell(i)).$
    \item For a locally noetherian (formal) scheme $X,$ and $Y\subseteq X$ a closed (formal) subscheme, we denote $K_0^Y(X)$ the $K_0$-theory of $X$ supported on $Y,$ and denote $K_0(X)=K_0^X(X).$ This is equipped with a descending filtration $F^kK_0^Y(X)$ by the codimension of the support, and we denote $\mathrm{Gr}^kK_0^Y(X)$ its associated graded. When $X$ is regular of finite type over a field $k,$ this is equipped with a map $[\O_{\blank}]\colon \mathrm{CH}^k(X)\to\mathrm{Gr}^kK_0(X).$
    \item If $W\subseteq X$ is a closed subscheme of a regular locally noetherian (formal) scheme and we have a proper map $\pi\colon W\to \Spf A$ for a complete discrete valuation ring $A,$ then we consider the degree map $\chi(X,\blank)\colon K_0^W(X)\to\Z$ given by $\chi(X,\mathscr{E})=\sum_{i}\mathrm{length}_{\O_{\Spf A}}R^i\pi\mathscr{E}.$
\end{itemize}

\subsection*{Acknowledgments}
I would like to thank my PhD advisor Wei Zhang for suggesting this problem and for his encouragement and advice. I am grateful to Ashay Burungale, Kazim B\"uy\"ukboduk, Olivier Fouquet, Qiao He, Shilin Lai, Yifeng Liu, Andreas Mihatsch, Yiannis Sakellaridis, Marco Sangiovanni Vincentelli, Chris Skinner, Matteo Tamiozzo, Zhiyu Zhang, and especially Naomi Sweeting for helpful conversations. This work was supported in part by the NSF grant DMS--1901642, and in part by the NSF grant DMS--1440140, while the author was in residence at the Simons Laufer Mathematical Sciences Institute in Berkeley, California, during the Spring Semester of 2023.

\section{Weight spectral sequence and potential map}\label{WeightChapter}
Let $K$ be a henselian discrete valuation field with separable closure $\overline{K},$ ring of integers $\O_K$ and residue field $\kappa.$ We consider $\Lambda$ an $\ell$-adic coefficient ring for some rational prime $\ell\neq\mathrm{char}(\kappa).$

We recall that $\Gamma_K$ denotes $\mathrm{Gal}(\overline{K}/K),$ and we denote by $I_K\subseteq\Gamma_K$ its inertia subgroup. We consider the $\ell$-adic quotient $t_\ell\colon I_K\to\Z_\ell(1),$ defined by the property that $T(\varpi^{1/\ell^m})=(\varpi^{1/\ell^m})^{t_\ell(T)}$ for every uniformizer $\varpi$ of $K,$ $m\ge1$ and $T\in I_K.$

\subsection{Semistable schemes and weight spectral sequence}
\begin{definition}[{\cite[Definition 2.1]{LiuTriple}}]
    Let $\mathcal{X}\to\Spec\O_K$ be a scheme locally of finite presentation of pure relative dimension $n.$ We say that $\mathcal{X}$ is \emph{(strictly) semistable} if it is Zariski locally \'etale over
    \begin{equation*}
        \Spec\O_K[t_1,\ldots,t_n]/(t_1\cdots t_s-\varpi)
    \end{equation*}
    for some varying integers $0\le s\le n$ and $\varpi$ a uniformizer of $K.$
\end{definition}

When $\mathcal{X}/\O_K$ is semistable, its special fiber $\mathcal{X}_\kappa$ is a normal crossings divisor of $\mathcal{X}.$ If $\{X_1,\ldots,X_m\}$ denote its irreducible components, we consider
\begin{equation*}
    X_I\defeq \bigcap_{i\in I} X_i,\quad\text{for }I\subseteq\{1,\ldots,n\}\text{ nontrivial}
\end{equation*}
which are smooth schemes over $\kappa,$ of codimension $\# I-1.$ We denote
\begin{equation*}
    X_\kappa^{(p)}\defeq\bigsqcup_{\substack{I\subseteq\{1,\ldots,m\}\\\# I=p+1}}X_I,\quad\text{for }p\in\Z_{\ge0},
\end{equation*}
which is pure of codimension $p$ in $\mathcal{X}_\kappa,$ and we consider $a_p\colon X^{(p)}_\kappa\hookrightarrow\mathcal{X}_\kappa$ the natural inclusion.

For $I\subseteq J\subseteq\{1,\ldots,m\},$ consider the inclusion $i_{I,J}\colon X_J\hookrightarrow X_I.$ These induce pullback maps
\begin{equation*}
    \delta_p^*\colon H_\et^q(X^{(p)}_{\overline{\kappa}},\Lambda(j))\to H_\et^q(X^{(p+1)}_{\overline{\kappa}},\Lambda(j))
\end{equation*}
and Gysin maps
\begin{equation*}
    \delta_{p,*}\colon H_\et^q(X^{(p)}_{\overline{\kappa}},\Lambda(j))\to H_\et^{q+2}(X^{(p-1)}_{\overline{\kappa}},\Lambda(j+1))
\end{equation*}
defined as a certain alternating sum of the maps $i_{I,J}^*$ and $i_{I,J,*}.$ Namely, if $J=\{j_0,\ldots,j_k\}$ with $j_0<\cdots<j_k$ and $I = J\setminus\{j_l\},$ then we consider the maps $i_{I,J}^*$ and $i_{I,J,*}$ weighted by $(-1)^l.$

If we are given a commutative monoid $\mathbb{T}$ with a monoidal functor $\mathbb{T}\to\mathrm{\acute{E}tCor}(\mathcal{X})$ (see \cite[Definition 2.11]{LiuTriple} for the definition of the target), then $\mathbb{T}$ acts on the cohomology groups $H^*_\et(\mathcal{X}_{\overline{K}},\Lambda)$ and $H^*_\et(X^{(p)}_{\overline{\kappa}},\Lambda).$ In particular, given a maximal ideal $\m\subseteq\Lambda[\mathbb{T}]$ we will consider the localized cohomology groups $H^*_\et(\mathcal{X}_{\overline{K}},\Lambda)_\m$ and $H^*_\et(X^{(p)}_{\overline{\kappa}},\Lambda)_\m.$

\newcommand{\pss}{$\m$-semistable}
\begin{definition}\label{pssDef}
    Given a maximal ideal $\m\subseteq\Lambda[\mathbb{T}],$ We say $\mathcal{X}\to\Spec\O_K$ is \emph{\pss{}} if i) the following specialization maps are isomorphisms
    \begin{equation*}
        H_{\et,c}^*(\mathcal{X}_{\overline{K}},\Lambda)_\m\to H_{\et,c}^*(\mathcal{X}_{\overline{\kappa}},R\Psi\Lambda)_\m,\qquad H_\et^*(\mathcal{X}_{\overline{K}},\Lambda)_\m\to H_\et^*(\mathcal{X}_{\overline{\kappa}},R\Psi\Lambda)_\m,
    \end{equation*}
    where $R\Psi\Lambda\defeq\bar{i}^*R\bar{j}_*\Lambda$ for $\mathcal{X}_{\overline{{\kappa}}}\xhookrightarrow{\bar{i}}\mathcal{X}_{\O_K^{\mathrm{ur}}}\xhookleftarrow{\bar{j}}\mathcal{X}_{\overline{K}}$ is the nearby cycles, and ii) for each $p\in\Z_{\ge0}$ and $i\in\Z_{\ge0},$ the canonical maps $H^i_{\et,c}(X^{(p)}_{\overline{\kappa}},\Lambda)_\m\to H^i_\et(X^{(p)}_{\overline{\kappa}},\Lambda)_\m$ are isomorphisms.
\end{definition}
\begin{remark}
    This is automatic when $\mathcal{X}\to\Spec\O_K$ is proper.
\end{remark}

When $\mathcal{X}$ is \pss{}, we have the localized weight spectral sequence
\begin{equation*}
    E_{1,\m}^{p,q}=\bigoplus_{i\ge\max(0,-p)}H^{q-2i}_\et(X^{(p+2i)}_{\overline{\kappa}},\Lambda(-i))_\m\Longrightarrow H_\et^{p+q}(\mathcal{X}_{\overline{K}},\Lambda)_\m,
\end{equation*}
with differentials $d_{1,\m}^{p,q}\colon E_{1,\m}^{p,q}\to E_{1,\m}^{p+1,q}$ given by $\sum_{i\ge\max(0,-p)}(\delta_{p+2i,*}+\delta_{p+2i}^*).$

\subsection{Potential map}\label{Weight-Potential}
We recall the main results from \cite[Section 2]{LiuTriple}. Fix $\mathcal{X}$ a \pss{} scheme over $\O_K,$ for a maximal ideal $\m\subseteq\Lambda[\mathbb{T}]$ as above.
\begin{definition}\label{ABDef}
    We consider the following $\Lambda$-submodules of $H^{2j}_\et(X^{(0)}_{\overline{\kappa}},\Lambda(j))$
    \begin{equation*}
        B^j(\mathcal{X},\Lambda)\defeq\ker\left(H^{2j}_\et(X^{(0)}_{\overline{\kappa}},\Lambda(j))\xrightarrow{\delta_0^*}H^{2j}_\et(X^{(1)}_{\overline{\kappa}},\Lambda(j))\right)
    \end{equation*}
    and noticing that $B^j(\mathcal{X},\Lambda)\subseteq \ker d_1^{0,2j}(j),$
    \begin{equation*}
        B^j(\mathcal{X},\Lambda)^0\defeq\ker\left(B^j(\mathcal{X},\Lambda)'\to E_2^{0,2j}(j)\right).
    \end{equation*}
    Dually, we consider the following $\Lambda$-quotients of $H^{2(n-j)}_\et(X^{(0)}_{\overline{\kappa}},\Lambda(j))$
    \begin{equation*}
        B_j(\mathcal{X},\Lambda)\defeq\coker\left(H^{2(n-j-1)}_\et(X^{(1)}_{\overline{\kappa}},\Lambda(n-j-1))\to H^{2(n-j)}_\et(X^{(0)}_{\overline{\kappa}},\Lambda(n-j))\right)
    \end{equation*}
    and
    \begin{equation*}
        B_j(\mathcal{X},\Lambda)_0\defeq\coker\left(E_2^{0,2(n-j)}(n-j)\to B_j(\mathcal{X},\Lambda)\right).
    \end{equation*}
    We also consider $A^j(\mathcal{X},\Lambda)^0=(B^j(\mathcal{X},\Lambda)^0)^{G_\kappa}$ and $A_j(\mathcal{X},\Lambda)_0=(B_j(\mathcal{X},\Lambda)_0)_{G_\kappa}$
\end{definition}
\begin{lemma}[{\cite[Lemma 2.4, Definition 2.5]{LiuTriple}}]\label{LiuFactor}
The composite map
\begin{equation*}
    H^{2r-2}_\et(X^{(0)}_{\overline{\kappa}},\Lambda(r-1))\xrightarrow{\delta_0^*}H^{2r-2}_\et(X^{(1)}_{\overline{\kappa}},\Lambda(r-1))\xrightarrow{\delta_{1,*}}H^{2r}(X^{(0)}_{\overline{\kappa}},\Lambda(r))
\end{equation*}
factors through $B_{n+1-r}(\mathcal{X},\Lambda)_0\to B^r(\mathcal{X},\Lambda)^0.$ We define the \emph{potential map} $\Delta^r\colon A_{n+1-r}(\mathcal{X},\Lambda)_0\to A^r(\mathcal{X},\Lambda)^0$ to be such map, that is, the map that fits in the following commutative diagram.
\begin{equation*}
    \begin{tikzcd}
        H^{2r-2}_\et(X^{(0)}_{\overline{\kappa}},\Lambda(r-1))\arrow[r,"\delta_0^*"]\arrow[d,twoheadrightarrow]&H^{2r-2}_\et(X^{(1)}_{\overline{\kappa}},\Lambda(r-1))\arrow[r,"\delta_{1,*}"]&H^{2r}(X^{(0)}_{\overline{\kappa}},\Lambda(r))\\
        A_{n+1-r}(\mathcal{X},\Lambda)_0\arrow[rr,"\Delta^r"]&&A^r(\mathcal{X},\Lambda)^0\arrow[u,hook]
    \end{tikzcd}
\end{equation*}
\end{lemma}
As explained in \cite[Construction 2.13]{LiuTriple}, $\mathbb{T}$ also acts on 
$A^j(\mathcal{X},\Lambda)^0$ and $A_j(\mathcal{X},\Lambda)_0,$ and so we also consider the localized potential map
\begin{equation*}
    \Delta^r_\m\colon A_{n+1-r}(\mathcal{X},\Lambda)_{0,\m}\to A^r(\mathcal{X},\Lambda)^0_\m.
\end{equation*}

\begin{definition}[{\cite[Definition 2.7]{LiuTriple}}]\label{nicecoef}
    We say that $\Lambda$ is a \emph{very nice coefficient} for the spectral sequence $E_{\bullet,\m}^{\bullet,\bullet}(r)$ if
    \begin{enumerate}[label=(N\arabic*)]
        \item\label{N1} $E_{\bullet,\m}^{\bullet,\bullet}$ degenerates on the second page;
        \item\label{N2} if $E_{2,\m}^{p,2r-1-p}(r-1)$ has a non-trivial subquotient on which $\Gamma_{\kappa}$ acts trivially, then $p=1.$
        \item[(N3)]\label{N3} for every subquotient $M$ of $H^{2r}(X_{\overline{\kappa}}^{(0)},\Lambda(r))\oplus H^{2r-2}(X_{\overline{\kappa}}^{(0)},\Lambda(r-1)),$ we have that the canonical map $M^{\Gamma_{\kappa}}\to M_{\Gamma_{\kappa}}$ is an isomorphism.
    \end{enumerate}
\end{definition}
\begin{theorem}[{\cite[Theorem 2.9]{LiuTriple}}]\label{LiuCohomology}
    Suppose $\Lambda$ is a very nice coefficient for the spectral sequence $E_{\bullet,\m}^{\bullet,\bullet}(r).$ Then we have a canonical identification
    \begin{equation*}
        \eta^r_\m\colon \coker\Delta^r_\m\rightiso H^1_{\sing}(K,H_\et^{2r-1}(\mathcal{X}_{\overline{K}},\Lambda)_\m).
    \end{equation*}
\end{theorem}
For every $r\in\Z_{\ge0},$ we have the localized cycle class map
\begin{equation*}
    \mathrm{cl}_\m\colon \mathrm{CH}^r(\mathcal{X}_K)_\m\to H^{2r}_\et(\mathcal{X}_K,\Lambda(r))_\m.
\end{equation*}
We denote by $\mathrm{CH}^r(\mathcal{X}_K,\Lambda)^0_\m$ the submodule of cohomologically trivial cycles, namely the kernel of the composition
\begin{equation*}
    \mathrm{CH}^r(\mathcal{X}_K,\Lambda)_\m\xrightarrow{\mathrm{cl}_\m}H^{2r}_\et(\mathcal{X}_K,\Lambda(r))_\m\to H^{2r}_\et(\mathcal{X}_{\overline{K}},\Lambda(r))_\m.
\end{equation*}
From the Grothendieck spectral sequence, we thus obtain the following edge map, which we call the (localized) Abel--Jacobi map
\begin{equation*}
    \mathrm{AJ}_\m\colon \mathrm{CH}^r(\mathcal{X}_K,\Lambda)^0_\m\to H^1(K,H^{2r-1}_\et(\mathcal{X}_{\overline{K}},\Lambda(r))_\m).
\end{equation*}

\begin{definition}\label{ZariskiClosure}
    Let $Z\in \mathrm{Z}^r(\mathcal{X}_K)$ be an algebraic cycle of codimension $r.$ Denote by $i\colon\overline{\mathrm{supp}(Z)}\subseteq\mathcal{X}$ the Zariski closure of the support of $Z.$ We denote $\mathcal{Z}\in \mathrm{Z}^r(\mathcal{X})$ be the unique cycle on $\mathcal{X}$ of codimension $r$ supported on $\overline{\mathrm{supp}(Z)}$ whose restriction to $\mathcal{X}_K$ is $Z.$
\end{definition}

\begin{theorem}[{\cite[Theorem 2.18]{LiuTriple}}]\label{LiuAbelJacobi}
    Suppose that $\Lambda$ is a very nice coefficient for $E^{\bullet,\bullet}_{\m,\bullet}(r).$ For $Z\in \mathrm{Z}^r(\mathcal{X}_K)$ belonging to $\mathrm{CH}^r(\mathcal{X}_K,\Lambda)^0_\m,$ we have
    \begin{equation*}
        \eta^r_\m(\widetilde{Z})=\partial(\mathrm{AJ}_\m(Z))
    \end{equation*}
    where $\widetilde{Z}$ is the image of $\mathrm{cl}_\m(\mathcal{Z})$ (\Cref{ZariskiClosure}) under
    \begin{equation*}
        H^{2r}_\et(\mathcal{X},\Lambda(r))_\m\to H^{2r}_\et(\mathcal{X}_\kappa,\Lambda(r))_\m\to H^{2r}_\et(\mathcal{X}_{\overline{\kappa}},\Lambda(r))_\m^{\Gamma_{\kappa}}\to H^{2r}_\et(X^{(0)}_{\overline{\kappa}},\Lambda(r))_\m^{\Gamma_{\kappa}},
    \end{equation*}
    which lies in $B^r(\mathcal{X}_\kappa,\Lambda)_\m^0$ by \cite[Lemma 2.17]{LiuTriple}.
\end{theorem}

\subsection{Description in terms of K-theory}\label{Weight-Ktheory}
We continue assuming that $\mathcal{X}$ is \pss{} for a maximal ideal $\m\subseteq\Lambda[\mathbb{T}]$ as above.
\begin{definition}\label{extensiondef}
    Consider $Z\in \mathrm{CH}^r(\mathcal{X}_K)\otimes\Lambda.$ We say that ${}^{\mathbb{L}}\mathcal{Z}\in \mathrm{Gr}^rK_0^{\mathcal{X}_\kappa\cup\mathrm{supp}(Z)}(\mathcal{X})\otimes\Lambda$ is an \emph{extension} of $Z$ if ${}^{\mathbb{L}}\mathcal{Z}\rvert_{\mathcal{X}_K}\in \mathrm{Gr}^rK_0(\mathcal{X}_K)\otimes\Lambda$ coincides with $\O_Z.$
\end{definition}

We have the following extension of \Cref{LiuAbelJacobi}.
\begin{theorem}\label{KtheoryAJ}
    Suppose $\Lambda$ is a very nice coefficient for $E^{\bullet,\bullet}_{\m,\bullet}(r).$ For $Z\in\mathrm{CH}^r(\mathcal{X}_K,\Lambda)^0_\m$ and ${}^{\mathbb{L}}\mathcal{Z}\in \mathrm{Gr}^rK_0^{\mathcal{X}_\kappa\cup\mathrm{supp}(Z)}(\mathcal{X})\otimes\Lambda$ an extension of $Z,$ we define $\widetilde{{}^{\mathbb{L}}\mathcal{Z}}$ to be the image of ${}^{\mathbb{L}}\mathcal{Z}$ under
    \begin{equation*}
        \mathrm{Gr}^rK_0^{\mathcal{X}_\kappa\cup\mathrm{supp}(Z)}(\mathcal{X})\otimes\Lambda\to \mathrm{Gr}^rK_0(X^{(0)}_{\kappa})\otimes\Lambda\xrightarrow{\mathrm{cl}}H^{2r}_\et(X^{(0)}_{\overline{\kappa}},\Lambda(r))^{\Gamma_{\kappa}}.
    \end{equation*}
    Then $\widetilde{{}^{\mathbb{L}}\mathcal{Z}}$ lies in $B^r(\mathcal{X}_\kappa,\Lambda)_\m^0,$ and we have
    \begin{equation*}
        \eta^r_\m(\widetilde{{}^{\mathbb{L}}\mathcal{Z}})=\partial(\mathrm{AJ}_\m(Z)).
    \end{equation*}
\end{theorem}
\begin{proof}
    First note that if $\mathcal{Z}$ is as in \Cref{ZariskiClosure}, we have that $\widetilde{\O_{\mathcal{Z}}}$ is simply $\widetilde{Z}.$ By \cref{KLemma}, we have an exact sequence
    \begin{equation*}
        \mathrm{Gr}^{r-1}K_0(X_\kappa^{(0)})\xrightarrow{\iota_*} \mathrm{Gr}^rK_0(\mathcal{X})\to \mathrm{Gr}^rK_0(\mathcal{X}_K)\to0,
    \end{equation*}
    and as ${}^{\mathbb{L}}\mathcal{Z}$ is an extension of $Z,$ this implies that we must have that ${}^{\mathbb{L}}\mathcal{Z}-\O_{\mathcal{Z}}=\iota_*(Y)$ for some $Y\in \mathrm{Gr}^{r-1}K_0(X_\kappa^{(0)}).$ Hence, by \Cref{LiuAbelJacobi}, it suffices check that the image of the composition
    \begin{equation*}
        \gamma\colon \mathrm{Gr}^{r-1}K_0(X_\kappa^{(0)})\otimes\Lambda\to\mathrm{Gr}^rK_0(\mathcal{X})\otimes\Lambda\to \mathrm{Gr}^rK_0(X^{(0)}_\kappa)\otimes\Lambda\xrightarrow{\mathrm{cl}}H^{2r}_\et(X^{(0)}_{\overline{\kappa}},\Lambda(r))_\m^{\Gamma_{\kappa}}
    \end{equation*}
    is contained in $B^r(\mathcal{X}_\kappa,\Lambda(r))_\m$ and that $\eta^r_\m\circ\gamma=0.$

    Consider the following diagram.
    \begin{equation*}
        \begin{tikzcd}
            \mathrm{Gr}^{r-1}K_0(X_\kappa^{(0)})\arrow[dd,"\cdot(-1)","\sim"']\arrow[r]\arrow[rrrd, bend left, "\gamma"]&[-15pt]\mathrm{Gr}^rK_0(\mathcal{X})\arrow[dr]&[-40pt]&[-15pt]\\
            &&\mathrm{Gr}^rK_0(X^{(0)}_\kappa)\arrow[r,"\mathrm{cl}"]&H^{2r}_\et(X^{(0)}_{\overline{\kappa}},\Lambda(r))\\
            \mathrm{Gr}^{r-1}K_0(X^{(0)}_\kappa)\arrow[r,"\delta_0^*"]\arrow[rrru, bend right, "\gamma_0"]&\mathrm{Gr}^{r-1}K_0(X^{(1)}_\kappa)\arrow[ur,"\delta_{1,*}"]&&
        \end{tikzcd}
    \end{equation*}
    By \Cref{LiuFactor}, the map $\gamma_0$ factors through the potential map $\Delta^r.$ In particular, its image is contained in $B^r(\mathcal{X}_\kappa,\Lambda(r))_\m$ and, by \Cref{LiuCohomology}, it composes to $0$ under $\eta_\m^r.$ Thus, we will finish the proof of the theorem by checking the commutativity of the diagram.
    
    Denote by $\gamma_{i,j}\colon \mathrm{Gr}^{r-1}K_0(X_i)\to \mathrm{Gr}^rK_0(X_j)$ the components of $\gamma,$ and by $\beta_{i,j}\colon \mathrm{Gr}^{r-1}K_0(X_i)\to \mathrm{Gr}^rK_0(X_j)$ the components of $-\gamma_0.$ By the sign convention on $\delta_0^*$ and $\delta_{1,*},$ we have that $\beta_{i,j}$ is
    \begin{equation*}
        \beta_{i,j}\colon\begin{cases}
        \mathrm{Gr}^{r-1}K_0(X_i)\xrightarrow{\iota_{i,j}^*}\mathrm{Gr}^{r-1}K_0(X_i\cap X_j)\xrightarrow{\iota_{j,i,*}}\mathrm{Gr}^rK_0(X_j)&\text{if }i\neq j,\\
        \mathrm{Gr}^{r-1}K_0(X_i)\xrightarrow{\bigoplus_{k\neq i}\iota_{i,k}^*}\bigoplus_{k\neq i}\mathrm{Gr}^{r-1}K_0(X_i\cap X_k)\xrightarrow{-\sum_{k\neq i}\iota_{i,k,*}}\mathrm{Gr}^rK_0(X_i)&\text{if }i=j.
        \end{cases}
    \end{equation*}
    For $i\neq j,$ we have that $\beta_{i,j}=\gamma_{i,j}$ since $X_i$ and $X_j$ intersect transversely. For $i=j,$ note that
    \begin{equation*}
        \O_{\mathcal{X}}(X_1+\cdots+X_m)=\O_{\mathcal{X}}(\mathcal{X}_\kappa)=s^*(\O_{\Spec\O_K}(\kappa))\iso s^*\O_{\Spec\O_K}=\O_{\mathcal{X}}
    \end{equation*}
    where
    $s\colon\mathcal{X}\to \Spec\O_K$ is the structure morphism. This implies that
    \begin{equation*}
        \O_{\mathcal{X}}(X_i)\rvert_{X_i}\iso\O_{\mathcal{X}}(-\sum_{k\neq i}X_k)\rvert_{X_i},
    \end{equation*}
    and thus that $\beta_{i,i}=\gamma_{i,i}.$
\end{proof}

\section{Unitary Rapoport--Zink spaces and special cycles}\label{RZChapter}
Let $p$ be an odd prime, $F^+$ a finite extension of $\Q_p$ with uniformizer $\varpi,$ and $F/F^+$ an unramified quadratic extension of $F^+.$ We denote $(a\mapsto a^\cplx)\in\mathrm{Gal}(F/F^+)$ the nontrivial automorphism.

We denote by $q$ the cardinality of the residue field of $F^+,$ by $\breve{F}$ the completion of the maximal unramified extension of $F,$ with residue field $\kappa\defeq\overline{\F}_p.$

\subsection{Unitary Rapoport--Zink spaces}\label{RZ-RZ}
Let $N\ge1$ be an integer.
\begin{definition}\label{pDivDef}
Fix $0\le t\le N$ and $r,s\in\Z_{\ge0}$ with $N=r+s.$ For a $\Spf\O_{\breve{F}}$-scheme $S,$ a \emph{Hermitian $\O_F$-module of signature $(r,s)$ and type $t$ over $S$} is a triple $(X,\iota,\lambda)$ where
\begin{itemize}
    \item $X$ is a formal $p$-divisible $\O_{F^+}$-module over $S$ of dimension $N$ and relative height $2N.$ We assume that $X$ is strict, i.e. the action of $\O_{F^+}$ on $\Lie X$ coincides with the structure map $\O_{F^+}\to\O_S.$
    \item $\iota\colon\O_F\to\mathrm{End}(X)$ is an action of $\O_F$ extending the action of $\O_{F^+}$ and satisfying the \emph{Kottwitz signature condition}: for $a\in\O_F,$ the characteristic polynomial of $\iota(a)$ on $\Lie X$ is $(T-a)^r(T-a^\cplx)^s\in\O_S[T],$
    \item $\lambda\colon X\to X^\vee$ a polarization on $X$ such that $\ker\lambda\subseteq X[\iota(\varpi)]$ has order $q^{2t}$ and such that for all $a\in\O_F,$ the following diagram commutes.
    \begin{equation*}
        \begin{tikzcd}
        X\arrow{r}{\lambda}\arrow{d}{\iota(a)}&X^\vee\arrow{d}{\iota(a^\cplx)^\vee}\\
        X\arrow{r}{\lambda}&X^\vee
        \end{tikzcd}
    \end{equation*}
\end{itemize}
We denote $\pDiv{F/F^+}{[t],(r,s)}(S)$ the category of such objects, where a morphism $(X_1,\iota_1,\lambda_1)\to (X_2,\iota_2,\lambda_2)$ is a $\O_F$-linear quasi-isogeny $\varphi\colon X_1\dashrightarrow X_2$ such that $\varphi^*(\lambda_2)=\lambda_1.$
\end{definition}

\begin{proposition}
    For $0\le t\le N$ and $r+s=N,$ the category $\pDiv{F/F^+}{[t],(r,s)}(\kappa)$ is nontrivial, and all its objects are quasi-isogenous.
\end{proposition}
\begin{proof}
    This follows by relative Dieudonn\'e theory, see \cite[Lemma 1.14, Proposition 1.15]{Vollaard} and \cite[Section 5]{RSZ-AT}.
\end{proof}
\begin{remark}
    The uniqueness up to quasi-isogeny fails if $N$ is odd and $F/F^+$ is ramified, see \cite[Section 7]{RSZ-AT}.
\end{remark}

\begin{definition}
    Fix an object $(\X^{[t]}_N,\iota_{\X^{[t]}_N},\lambda_{\X^{[t]}_N})\in\pDiv{F/F^+}{[t],(r,s)}(\kappa)$ as the \emph{framing object}. The \emph{unitary Rapoport--Zink space of signature $(r,s)$ and type $t$} is the functor
    \begin{equation*}
        \N{F/F^+}{[t],(r,s)}\to\Spf\O_{\breve{F}}
    \end{equation*}
    for which $\N{F/F^+}{[t],(r,s)}(S)$ is the set of isomorphism classes of tuples $(X,\iota,\lambda,\rho)$ where
    \begin{itemize}
        \item $(X,\iota,\lambda)\in\pDiv{F/F^+}{[t],(r,s)}(S)$ is a Hermitian $\O_F$-module of signature $(r,s)$ and type $t$ over $S,$
        \item if we let $\overline{S}\defeq S_{\kappa},$ then $\rho\colon X\times_S\overline{S}\dashrightarrow\X^{[t]}_N\times_{\kappa}\overline{S}$ is a \emph{framing}, i.e. a morphism in the category $\pDiv{F/F^+}{[t],(r,s)}(\overline{S})$ (which is necessarily of height $0$).
    \end{itemize}
\end{definition}
\begin{theorem}[{\cite[Theorem 2.16]{Rapoport-Zink}, \cite[Proposition 2.17]{Mihatsch}}]
    $\N{F/F^+}{[t],(r,s)}$ is representable by a separated formal scheme locally formally of finite type over $\Spf \O_{\breve{F}}$ of relative dimension $rs.$
\end{theorem}
\begin{definition}
    Given a framing object $(\X^{[t]}_N,\iota_{\X^{[t]}_N},\lambda_{\X^{[t]}_N})\in\pDiv{F/F^+}{[t],(r,s)}(\kappa),$ we consider the algebraic group $J^{[t]}$ over $\O_{F^+}$ given by $\mathrm{Aut}(\X^{[t]}_N,\iota_{\X^{[t]}_N},\lambda_{\X^{[t]}_N}).$ This is such that $J^{[t]}(F^+)$ acts on $\N{F/F^+}{[t],(r,s)}.$
\end{definition}

\begin{theorem}[Canonical lifting, \cite{Gross}]
We have $\N{F/F^+}{[0],(1,0)}\iso\Spf \O_{\breve{F}}$ and $\N{F/F^+}{[0],(0,1)}\iso\Spf \O_{\breve{F}}.$
\end{theorem}
More generally, we have the following consequence of the theories of Serre--Tate and Grothendieck--Messing (see \cite[Proposition 3.4.8]{LTXZZ}).
\begin{theorem}\label{CanonicalLifting+}
    If either $r=0$ or $s=0,$ we have
    \begin{equation*}
        \N{F/F^+}{[t],(r,s)}=J^{[t]}(F^+)/J^{[t]}(\O_{F^+})\times\Spf\O_{\breve{F}}
    \end{equation*}
    where the action of $J^{[t]}(F^+)$ is on the left component, and the framing object corresponds to the identity coset.
\end{theorem}

\subsection{Kudla--Rapoport cycles}
Let $(\E,\iota_{\E},\lambda_{\E})\in\pDiv{F/F^+}{[0],(0,1)}(\kappa),$ and denote by $\mathcal{E}$ the canonical lifting of $\E,$ which is the universal object for the Rapoport--Zink spaces with framing objects $\mathbb{E}.$

\begin{definition}
    The spaces of \emph{special quasi-homomorphisms} is the $F$-vector space
    \begin{equation*}
        \mathbb{V}^{[t]}_N\defeq\mathrm{Hom}_{\O_F}(\E,\X^{[t]}_N)\otimes\Q
    \end{equation*}
    equipped with the Hermitian form $\langle\cdot,\cdot\rangle_{\mathbb{V}^{[t]}_N}$ such that the following diagram commutes
    \begin{equation*}
        \begin{tikzcd}
        \E\arrow{r}{x}\arrow{d}{\iota_{\E}(\langle x,y\rangle_{\mathbb{V}^{[t]}_N})}&\X^{[t]}_N\arrow{r}{\lambda_{\X^{[t]}_N}}&(\X^{[t]}_N)^\vee\arrow{d}{y^\vee}\\
        \E\arrow{rr}{\lambda_{\E}}&&\E^\vee
        \end{tikzcd}
    \end{equation*}
\end{definition}
\begin{remark}[{\cite[Remark 2.1.0.3]{Zhiyu}}]\label{J=U(V)Remark}
    By relative Dieudonn\'e theory, we have that the base change $J^{[t]}\otimes_{\O_{F^+}}F^+$ is isomorphic to the unitary group $U(\mathbb{V}^{[t]}_N),$ and that the isomorphism class of $\mathbb{V}^{[t]}_N$ is determined by the fact that it does not contain vertex lattices of type $t$ as in \Cref{DefVertexLattice}. In particular, $\mathbb{V}^{[0]}$ resp. $\mathbb{V}^{[1]}$ is non-split resp. split.
\end{remark}

\begin{definition}
    For a $\O_F$-submodule $L\subseteq\mathbb{V},$ we define the \emph{Kudla--Rapoport cycle} $\ZZ_-(L)\to\N{F/F^+}{[t],(r,s)}$ to be the closed formal subscheme of $\N{F/F^+}{[t],(r,s)}$ representing the following functor: $\ZZ_-(L)(S)\subseteq\N{F/F^+}{[t],(r,s)}(S)$ is the subset of $(X,\iota,\lambda,\rho)$ such that for any $x\in L,$ the composition
    \begin{equation*}
        \mathcal{E}_S\times_S\overline{S}\xdashrightarrow{\rho_{\mathcal{E}}}\E\times_{\kappa}\overline{S}\xdashrightarrow{x}\X^{[t]}_N\times_{\kappa}\overline{S}\xdashrightarrow{\rho^{-1}}X\times_S\overline{S}
    \end{equation*}
    extends to a homomorphism $\mathcal{E}_S\to X.$

    Similarly, we define the \emph{Kudla--Rapoport cycle} $\ZZ_+(L)\to\N{F/F^+}{[t],(r,s)}$ to be the closed formal subscheme of $\N{F/F^+}{[t],(r,s)}$ representing the following functor: $\ZZ_+(L)(S)\subseteq\N{F/F^+}{[t],(r,s)}(S)$ is the subset of $(X,\iota,\lambda,\rho)$ such that for any $x\in L,$ the composition
    \begin{equation*}
        \mathcal{E}_S\times_S\overline{S}\xdashrightarrow{\rho_{\mathcal{E}}}\E\times_{\kappa}\overline{S}\xdashrightarrow{x}\X^{[t]}_N\times_{\kappa}\overline{S}\xdashrightarrow{\rho^{-1}}X\times_S\overline{S}\xrightarrow{\lambda_X}X^\vee\times_S\overline{S}
    \end{equation*}
    extends to a homomorphism $\mathcal{E}_S\to X^\vee.$
\end{definition}

\begin{example}\label{n0KR}
In the case $r=0$ or $s=0,$ consider $\Lambda\subseteq\mathbb{V}^{[t]}_N$ to be the lattice
\begin{equation*}
    \Hom_{\O_F}(\mathbb{E},\mathbb{X}^{[t]}_N)\subseteq \mathbb{V}^{[t]}_N.
\end{equation*}
Then under \Cref{CanonicalLifting+} and \Cref{J=U(V)Remark}, $\ZZ_-(L)$ resp. $\ZZ_+(L)$ corresponds to the subset of $g\in J^{[t]}(F^+)/J^{[t]}/(\O_{F^+})$ such that $L\subseteq g\Lambda$ resp. $L\subseteq g\Lambda^\vee.$
\end{example}

\subsection{Case of signature \texorpdfstring{$(N-1,1)$}{(N-1,1)}}\label{RZ-BT}
From now on, we assume that $(r,s)=(N-1,1).$
\begin{proposition}[{\cite[Proposition 5.9]{Cho}}\footnote{We note that the cycles $\ZZ_-(L)$ resp. $\ZZ_+(L)$ are denoted $\ZZ(L)$ resp. $\mathcal{Y}(L)$ in \cite{Cho}.}]
    When $L\subseteq\mathbb{V}^{[t]}_N$ is a $\O_F$-lattice of rank $1,$ both $\ZZ_-(L)$ and $\ZZ_+(L)$ are Cartier divisors on $\N{F/F^+,\O_{\breve{F}}}{[t],(N-1,1)}.$
\end{proposition}
\begin{remark}
    When $t=0$ resp. $t=N,$ we have that $\ZZ_-(L)=\ZZ_+(L)$ resp. $\ZZ_-(\varpi L)=\ZZ_+(L)$ are in fact relative Cartier divisors, but this may not be true for other $t.$
\end{remark}

\begin{definition}\label{KRDef}
    Given $\underline{x}=(x_1,\ldots,x_m)\in(\mathbb{V}^{[t]}_N)^m$ and $\pm\in\{+,-\},$ we consider the \emph{derived Kudla--Rapoport cycles} ${}^\L\ZZ_\pm(\underline{x})$ to be the image of $\O_{\ZZ_\pm(x_1)}\otimes^\L\cdots\otimes^\L\O_{\ZZ_\pm(x_m)}$ in $\mathrm{Gr}^mK_0^{\ZZ_\pm(\underline{x})}(\N{F/F^+}{[t],(N-1,1)}).$
\end{definition}

\subsubsection{Bruhat--Tits stratification}
\begin{definition}\label{DefVertexLattice}
    For a Hermitian $F/F^+$-space $V,$ we say a $\O_F$-lattice $\Lambda\subseteq V$ is a \emph{vertex lattice} if $\varpi\Lambda^\vee\subseteq\Lambda\subseteq\Lambda^\vee.$ We say $\Lambda$ has \emph{type} $t(\Lambda)$ where $\dim_{\kappa}\Lambda^\vee/\Lambda=t(\Lambda).$ We denote $\mathrm{Vert}^t(V)$ the set of vertex lattices of type $t$ in $V.$
\end{definition}

\begin{definition}\label{BTDef}
    For a vertex lattice $\Lambda\subseteq\mathbb{V}^{[t]}_N,$ its corresponding \emph{closed Bruhat--Tits strata} is
    \begin{equation*}
        \V(\Lambda)\defeq\begin{cases}
            \ZZ_-(\Lambda)^\red&\text{if }t(\Lambda)>t,\\
            \ZZ_+(\Lambda^\vee)^\red&\text{if }t(\Lambda)<t.
        \end{cases}
    \end{equation*}
\end{definition}
\begin{remark}\label{KRReduced}
    In the case $t=0,$ $\ZZ(\Lambda)$ is already reduced for a vertex lattice $\Lambda,$ as shown in \cite[Theorem B]{LiZhu} or \cite[Theorems 9.4, 10.1]{RTZ}.
\end{remark}
\begin{theorem}[{\cite[Theorem 1.1]{Cho}, \cite[Theorem 2.5.3.5]{Zhiyu}}]
    For a vertex lattice $\Lambda\subseteq\mathbb{V}^{[t]}_N,$ we have that $\V(\Lambda)$ is projective, smooth and geometrically irreducible, and we have
    \begin{equation*}
        \dim\V(\Lambda)=\begin{cases}
            t+\frac{t(\Lambda)-t-1}{2}&\text{if }t(\Lambda)>t,\\
            N-t+\frac{t-1-t(\Lambda)}{2}&\text{if }t(\Lambda)<t.
        \end{cases}
    \end{equation*}
    In the cases $t(\Lambda)=t-1$ resp. $t(\Lambda)=t+1,$ we have that $\V(\Lambda)=\mathbb{P}^{N-t}_{\O_{\breve{F}}}$ resp. $\V(\Lambda)=\mathbb{P}^t_{\O_{\breve{F}}}.$ Moreover, we have
    \begin{equation*}
        (\N{F/F^+}{[t],(N-1,1)})^\red=\bigcup_{\Lambda\subseteq\mathbb{V}_N^{[t]}}\V(\Lambda).
    \end{equation*}
\end{theorem}

\subsubsection{Inductive structure}
\begin{proposition}[{\cite[Proposition 5.10]{Cho}}]\label{RZInd}\
\begin{enumerate}
    \item Let $x\in\mathbb{V}^{[t]}_N$ with $\mathrm{val}(\langle x,x\rangle_{\mathbb{V}^{[t]}_N})=0.$ Then there is an isomorphism $\Phi\colon \ZZ_-(x)\rightiso\N{F/F^+}{[t],(N-2,1)}.$ Moreover, if we identify $\mathbb{V}^{[t]}_N=\mathbb{V}^{[t]}_{N-1}\obot Fx,$ and if $L\subseteq\mathbb{V}^{[t]}_N$ is a subspace of the form $L=L^\flat\oplus 0,$ then
    \begin{equation*}
        \Phi(\ZZ_-(x)\cap\ZZ_?(L^\flat))=\ZZ_?(L)\quad\text{for }?\in\{+,-\}.
    \end{equation*}
    \item Let $y\in\mathbb{V}^{[t]}_N$ with $\mathrm{val}(\langle y,y\rangle_{\mathbb{V}^{[t]}_N})=-1.$ Then there is an isomorphism $\Psi\colon \ZZ_+(y)\rightiso\N{F/F^+}{[t-1],(N-2,1)}.$\footnote{In particular, $\ZZ_+(y)=\emptyset$ in the case $t=0.$} Moreover, if we identify $\mathbb{V}^{[t]}_N=\mathbb{V}^{[t-1]}_{N-1}\obot Fy,$ and if $L\subseteq\mathbb{V}^{[t]}_N$ is a subspace of the form $L=L^\flat\oplus 0,$ then
    \begin{equation*}
        \Psi(\ZZ_+(y)\cap\ZZ_?(L^\flat))=\ZZ_?(L)\quad\text{for }?\in\{+,-\}.
    \end{equation*}
\end{enumerate}
\end{proposition}

\begin{corollary}\label{RZindBT}
    In the setting of the above theorem, we also have
    \begin{enumerate}
        \item If $\Lambda=\Lambda^\flat\obot\O_Fx$ is a vertex lattice with $t(\Lambda)>t,$ then
        \begin{equation*}
            \Phi(\V(\Lambda))=\V(\Lambda^\flat).
        \end{equation*}
        \item If $\Lambda=\Lambda^\flat\obot\varpi\O_F y$ is a vertex lattice with $t(\Lambda)=1+t(\Lambda^\flat)<t,$ then
        \begin{equation*}
            \Psi(\V(\Lambda))=\V(\Lambda^\flat).
        \end{equation*}
    \end{enumerate}
\end{corollary}

\subsection{Intersection numbers of KR cycles and BT strata in signature \texorpdfstring{$(N-1,1)$}{(N-1,1)}}\label{RZ-Intersection}
For ease of notation, we will denote $\mathcal{N}_N\defeq\N{F/F^+}{[0],(N-1,1)}$ resp. $\mathcal{N}_N^1\defeq\N{F/F^+}{[1],(N-1,1)}$ following the notation of \cite{Li-Zhang}.

\subsubsection{Self-dual case}
\begin{definition}
Given $\underline{x}=(x_1,\ldots,x_d)\in(\mathbb{V}_N^{[0]})^d$ for some $d\in\Z_{\ge0}$ with $2d+1\le N,$ we define the function $\Int_{\underline{x}}\colon\mathrm{Vert}^{2d+1}(\mathbb{V}^{[0]}_N)\to\Q$ by
\begin{equation*}
    \Int_{\underline{x}}(\Lambda)=\chi(\mathcal{N}_N,{}^\L\ZZ(\underline{x})\otimes^\L\V(\Lambda)).
\end{equation*}
\end{definition}

Such intersection numbers have (essentially) been computed by \cite{Li-Zhang}.
\begin{theorem}[{\cite[Lemma 6.4.6]{Li-Zhang}}]\label{LiZhangComp}
Let $\underline{x}=(x_1,\ldots,x_d)\in(\mathbb{V}_N^{[0]})^d$ and $\Lambda\in\mathrm{Vert}^{2d+1}(\mathbb{V}^{[0]}_N).$ Denote $L=\mathrm{span}_{\O_F}(\underline{x})\subseteq\mathbb{V}^{[0]}.$ Then we have
\begin{equation*}
    \Int_{\underline{x}}(\Lambda)=\left\{\begin{array}{cl}
        c(d-[\Lambda+L:\Lambda]) & \text{if }L\subseteq\Lambda^\vee,\ L\subseteq L^\vee, \\
        0 & \text{otherwise,}
    \end{array}\right.
\end{equation*}
where $c(k)\defeq\prod_{i=1}^k(1-q^{2i}).$
\end{theorem}
\begin{proof}
\cite[Lemma 6.4.6]{Li-Zhang} implies this in the case that $\underline{x}=(x,\ldots,x).$ However, we note that the general case follows easily by induction.

First, note that $\Int_{\underline{x}}(\Lambda)=0$ unless $x_i\in\Lambda^\vee$ and $\mathrm{val}(\langle x_i,x_i\rangle)\ge0$ for all $i,$ by the same proof of \cite[Lemma 6.2.1]{Li-Zhang}. Now note that
\begin{enumerate}
    \item $\Int_{\underline{x}}(\Lambda)$ is $\Lambda$-invariant in each coordinate by \cite[Lemma 6.4.4]{Li-Zhang},
    \item if $x_i\in\Lambda^\vee\setminus\Lambda$ and $\mathrm{val}(\langle x_i,x_i\rangle)\ge0,$ and if we denote $\Lambda'=\Lambda+\O_Fx_i,$ then $\Lambda'$ is a vertex lattice of type $2d-1,$ and we have
    \begin{equation*}
        \Int_{(x_1,\ldots,x_d)}(\Lambda)=\Int_{(x_1,\ldots,\widehat{x_i},\ldots,x_d)}(\Lambda')
    \end{equation*}
    by the projection formula.
\end{enumerate}
These are enough to prove the claim by induction on $[\Lambda+L\colon\Lambda],$ with the base case being \cite[Lemma 6.4.6]{Li-Zhang} for $\underline{x}=(0,\ldots,0).$
\end{proof}

\subsubsection{Almost self-dual case: balloon strata}
Recall that $\mathrm{Vert}^0(\mathbb{V}^{[1]}_N)$ consist of lattices $\Lambda^\circ\subseteq\mathbb{V}^{[1]}_{N}$ which are self-dual $(\Lambda^\circ)^\vee=\Lambda^\circ.$ Recall that $\V(\Lambda^\circ)\iso\mathbb{P}^{N-1},$ and we denote by $\xi_{\Lambda^\circ}$ the first Chern class of the tautological bundle $\O_{\V(\Lambda^\circ)}(1).$

\begin{definition}\label{RZintballoon}
Given $\underline{x}=(x_1,\ldots,x_d)\in(\mathbb{V}^{[1]}_{N})^d,$ we consider the function
\begin{equation*}
    \nabla^\circ_{\underline{x}}\colon\mathrm{Vert}^0(\mathbb{V}^{[1]}_{N})\to\Q
\end{equation*}
given by
\begin{equation*}
    \nabla^\circ_{\underline{x}}(\Lambda^\circ)\defeq\chi(\mathcal{N}_{N}^1,{}^\L\ZZ_-(\underline{x})\otimes^\L\V(\Lambda^\circ)\otimes^\L\xi_{\Lambda^\circ}^{N-d-1}).
\end{equation*}
\end{definition}
This function $\nabla^\circ_{\underline{x}}$ is readily computable from the results of \cite{Sankaran}.
\begin{proposition}\label{balloonInt}
Let $N\ge2,$ and consider $\Lambda^\circ\in\mathrm{Vert}^0(\mathbb{V}^{[1]}_N).$ For $x\in\mathbb{V}^{[1]}_N,$ we have that $\ZZ_-(x)\rvert_{\V(\Lambda^\circ)}$ resp. $\ZZ_+(x)\rvert_{\V(\Lambda^\circ)}$ are nonzero if and only if $x\in\Lambda^\circ,$ and in this case they are divisors of degree $1$ resp. $-q.$
\end{proposition}
\begin{proof}
The case $N=2$ is proven in \cite[Lemma 2.11]{Sankaran}\footnote{Strictly speaking, the second claim here follows from the second case on the cited result only after applying the involution defined in \cite[Definition 5.3]{Cho}.}. We will deduce the general case from this, using the inductive structure of Rapoport--Zink spaces.

If $N\ge3,$ we can find $y\in\Lambda^\circ$ with $\langle x,y\rangle=0$ and $\mathrm{val}(\langle y,y\rangle)=0.$ Then $\Lambda^\circ=\Lambda^\flat\obot\O_Fy$ for $\Lambda^\flat\in\mathrm{Vert}^0(\mathbb{V}^{[1]}_{N-1})$ and $x=(x^\flat,0)$ for some $x^\flat\in\mathbb{V}^{[1]}_{N-1}.$ By \Cref{RZInd} and \Cref{RZindBT} we have the isomorphism $\Phi\colon\ZZ_-(y)\rightiso\mathcal{N}_{N-1}^1$ where $\Phi(\ZZ_?(x)\rvert_{\ZZ_-(y)})=\ZZ_?(x^\flat)$ and $\Phi(\V(\Lambda^\circ))=\V(\Lambda^\flat),$ from which the claim follows by induction.
\end{proof}
\begin{corollary}\label{nablacirc}
If $L\defeq\mathrm{span}_{\O_F}(\underline{x}),$ then we have
\begin{equation*}
    \nabla^\circ_{\underline{x}}(\Lambda^\circ)=\begin{cases}
        1&\text{if }L\subseteq\Lambda^\circ,\\
        0&\text{otherwise.}
    \end{cases}
\end{equation*}
\end{corollary}

\subsubsection{Auxiliary Rapoport--Zink space}\label{SubsectionAuxRZ}
We follow the discussion of \cite[Section 10.2]{Li-Zhang}.

Fix an $\O_F$-linear isogeny $\alpha\colon\X^{[1]}_N\times\E\to\X^{[0]}_{N+1}$ such that $\ker\alpha\subseteq(\X^{[1]}\times\E)[\varpi]$ and $\alpha^*(\lambda_{\X^{[0]}_{N+1}})=\lambda_{\X^{[1]}_N}\times\varpi\lambda_{\E}.$ Let $x_0\in\mathbb{V}_{N+1}^{[0]}$ be the restriction of $\alpha$ to $\E.$ Then $\langle x_0,x_0\rangle_{\mathbb{V}_{N+1}^{[0]}}=\varpi$ and $\alpha$ induces an identification $\mathbb{V}_{N+1}^{[0]}=\mathbb{V}_N^{[1]}\obot Fx_0.$
\begin{definition}
We denote $\tilde{\mathcal{N}}_N^1\hookrightarrow\mathcal{N}_N^1\times_{\Spf\O_{\breve{F}}}\mathcal{N}_{N+1}$ to be the closed formal subscheme consisting of tuples $(Y,\iota_Y,\lambda_Y,\rho_Y,X,\iota_X,\lambda_X,\rho_X)$ such that the composition
\begin{equation*}
    (Y\times\mathcal{E}_S)\times_{\Spec\overline{k}}\overline{S}\xdashrightarrow{\rho_Y\times\rho_{\mathcal{E}}}(\X_N^{[1]}\times\E)\times_{\Spec\overline{k}}\overline{S}\xrightarrow{\alpha}\X_{N+1}^{[0]}\times_{\Spec\overline{k}}\overline{S}\xdashrightarrow{\rho_X^{-1}}X\times_{\Spec\overline{s}}\overline{S}
\end{equation*}
extends to a homomorphism $Y\times\mathcal{E}_S\to X.$
\end{definition}
This gives us a correspondence.
\begin{equation*}
    \begin{tikzcd}
    &\tilde{\mathcal{N}}_N^1\arrow{dl}[swap]{\pi_1}\arrow{dr}{\pi_2}&&\\
    \mathcal{N}_N^1&&\ZZ(x_0)\arrow[r,hook]&\mathcal{N}_{N+1}
    \end{tikzcd}
\end{equation*}
\begin{definition}
    For $x\in\mathbb{V}_N^{[1]}\subseteq\mathbb{V}_{N+1}^{[0]},$ we denote $\ZZ^\flat(x)\defeq\ZZ(x_0)\cap\ZZ(x),$ viewed as a formal scheme on $\ZZ(x_0).$
\end{definition}

The following is a special case of a conjecture of Kudla and Rapoport from \cite{Kudla-Rapoport-note}, see also \cite[Conjecture 10.4.1]{Li-Zhang}, proved by Li--Rapoport--Zhang.
\begin{theorem}[{\cite[Theorem 14.6.2, Proposition 16.4.1]{LRZ}}]\label{KRConj}
Let $\mathcal{N}_N^{1,ss}\subseteq\mathcal{N}_N^1$ denote the (disjoint) union of the $\V(\Lambda^\circ)\iso\mathbb{P}^{N-1}$ over all self-dual lattices $\Lambda^\circ\in\mathrm{Vert}^0(\mathbb{V}_N^{[1]}).$ Let $\ZZ(x_0)^{ss}\subseteq\ZZ(x_0)$ denote the (disjoint) union of all the points of the form $\V(\Lambda^\circ\oplus\O_Fx_0)$ for $\Lambda^\circ\in\mathrm{Vert}^0(\mathbb{V}^{[1]}_{N}).$ Then
\begin{enumerate}[label=(\roman*)]
    \item The formal scheme $\tilde{\mathcal{N}}_N^{1}$ is regular of dimension $N.$
    \item The morphism $\pi_1$ is finite flat of degree $q+1,$ \'etale away from $\mathcal{N}_N^{1,ss},$ and totally ramified along $\mathcal{N}_N^{1,ss}.$
    \item The morphism $\pi_2$ is proper, and is a blow-up\footnote{Blow-up here is meant in the generalized sense, i.e. a blow-up in an ideal sheaf with support in $Z(x_0)^{ss}.$} in the zero-dimensional subscheme $Z(x_0)^{ss}.$
    \item The exceptional divisor $\tilde{\mathcal{N}}^{1,\mathrm{exc}}$ of the blow-up $\pi_2$ is a reduced Cartier divisor and is isomorphic to $\mathcal{N}_N^{1,ss}$ under $\pi_1.$ In particular, if we denote $\mathbb{P}_{\Lambda^\circ}\defeq\pi_2^{-1}(\V(\Lambda^\circ\oplus\O_Fx_0)),$ then $\pi_1\colon \mathbb{P}_{\Lambda^\circ}\rightiso\V(\Lambda^\circ),$ and we have a decomposition
    \begin{equation*}
        \tilde{\mathcal{N}}^{1,\mathrm{exc}}=\bigsqcup_{\Lambda^\circ\in\mathrm{Vert}^0(\mathbb{V}^{[1]}_N)}\mathbb{P}_{\Lambda^\circ}.
    \end{equation*}
    \item For $\Lambda^\circ\in\mathrm{Vert}^0(\mathbb{V}^{[1]}_N),$ the normal bundle $N_{\V(\Lambda^\circ)/\tilde{\mathcal{N}}^1_N}$ is isomorphic to $\O_{\V(\Lambda^\circ)}(-1).$
\end{enumerate}
\end{theorem}
This conjecture allow us to relate the special divisors in $\mathcal{N}_N^1$ to those on $\mathcal{N}_{N+1}.$ As an example, we have\footnote{We note that the cycles $\ZZ^1_-(L)$ resp. $\ZZ^1_+(L)$ are denoted $\mathcal{Y}(L)$ resp. $\mathcal{Y}'(L)$ in \cite{Li-Zhang}.}
\begin{theorem}[{\cite[Theorem 10.4.3]{Li-Zhang}}]\label{ZCorrespondence}
Let $x\in\mathbb{V}^{[1]}_N.$ Define a locally finite Cartier divisor on $\tilde{\mathcal{N}}_N^1$
\begin{equation*}
    \mathrm{Exp}(x)\defeq\sum_{\substack{\Lambda^\circ\in\mathrm{Vert}^0(\mathbb{V}^{[1]}_N)\\x\in\Lambda^\circ}}\mathbb{P}_{\Lambda^\circ}.
\end{equation*}
Then we have an equality of Cartier divisors on $\tilde{\mathcal{N}}_N^1$
\begin{equation*}
    \pi_1^{-1}(\ZZ_-(x))=\pi_2^{-1}(\ZZ^\flat(x))-\mathrm{Exp}(x).
\end{equation*}
\end{theorem}
Similarly, we can prove
\begin{theorem}\label{VCorrespondence}
Let $\Lambda\in\mathrm{Vert}^{2d}(\mathbb{V}^{[1]}_N)$ for some $d\ge1.$ We let $\Lambda^\sharp\defeq\Lambda\oplus\O_Fx_0.$ Note that $\Lambda^\sharp\in\mathrm{Vert}^{2d+1}(\mathbb{V}^{[0]}_{N+1}),$ and that $\V(\Lambda^\sharp)\subseteq\ZZ(x_0).$ Define a locally finite cycle on $\tilde{\mathcal{N}}_N^1$ by
\begin{equation*}
    \mathrm{Exp}(\Lambda)\defeq\sum_{\substack{\Lambda^\circ\in\mathrm{Vert}^0(\mathbb{V}^{[1]}_{N})\\\Lambda\subseteq\Lambda^\circ\subseteq\Lambda^\vee}}H_{\mathbb{P}_{\Lambda^\circ}}^{N-d-1}
\end{equation*}
where $H_{\mathbb{P}_{\Lambda^\circ}}^{k}$ denotes a hyperplane of codimension $k$ in $\mathbb{P}_{\Lambda^\circ}.$ Then we have an equality
\begin{equation*}
    \pi_1^*(\O_{\V(\Lambda)})=\pi_2^*(\O_{\V(\Lambda^\sharp)})-\O_{\mathrm{Exp}(\Lambda)}
\end{equation*}
in $Gr^{N-d-1}K_0^{\pi_2^{-1}(\V(\Lambda^\sharp))}(\tilde{\mathcal{N}}^1_N).$
\end{theorem}
\begin{proof}
First we note that from the moduli description, we have
\begin{equation*}
    \pi_1^{-1}(\V(\Lambda))\subseteq\pi_1^{-1}(\ZZ_-(\Lambda))\subseteq\pi_2^{-1}(\ZZ(\Lambda^\sharp))=\pi_2^{-1}(\V(\Lambda^\sharp)).
\end{equation*}
Here, we are using that in the case $t=0$ we have that $\ZZ(\Lambda')$ is already reduced for any vertex lattice $\Lambda'$ by \Cref{KRReduced}.

Denote $\widetilde{\V(\Lambda^\sharp)}$ the strict transform of $\V(\Lambda^\sharp)$ under $\pi_2.$ Since $\V(\Lambda^\sharp)$ is integral, so is $\widetilde{\V(\Lambda^\sharp)}.$ Since $\pi_1$ is finite flat, we have that all irreducible components of $\pi_1^{-1}(\V(\Lambda))$ have dimension $d.$ Since $\widetilde{\V(\Lambda^\sharp)}$ is irreducible and $\pi_1^{-1}(\V(\Lambda))$ is not contained in the exceptional divisor, we conclude that
\begin{equation*}
    \pi_1^*(\O_{\V(\Lambda)})=\O_{\widetilde{\V(\Lambda^\sharp)}}+\mathcal{D}'
\end{equation*}
where $\mathcal{D}'$ is supported on $\tilde{\mathcal{N}}^{1,\mathrm{exc}}_N.$ As we also have
\begin{equation*}
    \pi_2^*(\O_{\V(\Lambda^\sharp)})=\O_{\widetilde{\V(\Lambda^\sharp)}}+\mathcal{D}''
\end{equation*}
for some $\mathcal{D}''$ supported on $\tilde{\mathcal{N}}^{1,\mathrm{exc}}_N,$ we conclude that
\begin{equation*}
    \pi_1^*(\O_{\V(\Lambda)})=\pi_2^*(\O_{\V(\Lambda^\sharp)})-\mathcal{D}
\end{equation*}
for some $\mathcal{D}\in Gr^{N-d-1}K_0^{\pi_2^{-1}(\V(\Lambda^\sharp))}(\tilde{\mathcal{N}}^{1,\mathrm{exc}}_N).$ That is, we may write
\begin{equation*}
    \mathcal{D}=\sum_{\Lambda^\circ\in\mathrm{Vert}^0(\mathbb{V}^{[1]}_N)}\mathrm{mult}(\Lambda,\Lambda^\circ)\O_{H^{N-d-1}_{\mathbb{P}_{\Lambda^\circ}}}
\end{equation*}
for some multiplicities $\mathrm{mult}(\Lambda,\Lambda^\circ).$ Note that since $\mathcal{D}$ is supported on $\pi_2^{-1}(\V(\Lambda^\sharp)),$ the multiplicity $\mathrm{mult}(\Lambda,\Lambda^\circ)$ can only be nonzero if the point $\V(\Lambda^\circ\oplus\O_Fx_0)$ is contained in $\V(\Lambda^\sharp),$ that is, if $\Lambda\subseteq\Lambda^\circ.$

Now suppose $\Lambda\subseteq\Lambda^\circ.$ Choose $L=\mathrm{span}_{\O_F}(x_1,\ldots,x_d)\subseteq\mathbb{V}^{[1]}_N$ such that $\Lambda+L=\Lambda^\circ.$ Then $\ZZ_-(L)\subseteq\mathcal{N}^1_N$ does not intersect $\V(\Lambda),$ as
\begin{equation*}
    \ZZ_-(L)\cap\V(\Lambda)\subseteq\ZZ_-(L)\cap\ZZ_-(\Lambda)=\ZZ_-(L+\Lambda)=\ZZ_-(\Lambda^\circ)=\emptyset.
\end{equation*}
Thus, by the projection formula for the finite flat map $\pi_1,$ we get that
\begin{equation*}
    0=\chi\left(\tilde{\mathcal{N}}^1_N,(\pi_2^*(\O_{\V(\Lambda^\sharp)})-\mathcal{D})\otimes^\L\pi_1^*(\O_{\ZZ_-(x_1)})\otimes^\L\cdots\otimes^\L\pi_1^*(\O_{\ZZ_-(x_d)})\right).
\end{equation*}

Note that if $\mathcal{F}\in F^1K_0(\ZZ(x_0)),$ we have by the projection formula for $\pi_2$ that
\begin{equation}\label{DerIntObservation}
    \pi_{2,*}(\O_{H^k_{\mathbb{P}_{\Lambda^\circ}}}\otimes^\L\pi_2^*(\mathcal{F}))=\pi_{2,*}(\O_{H^k_{\mathbb{P}_{\Lambda^\circ}}})\otimes^\L\mathcal{F}=0
\end{equation}
for any $\Lambda^\circ\in\mathrm{Vert}^0(\mathbb{V}^{[1]}_N)$ and $0\le k<n,$ since $\pi_{2,*}(\O_{H^k_{\mathbb{P}_{\Lambda^\circ}}})$ is supported on a zero-dimensional subscheme of $\ZZ(x_0).$

From this observation and by using \Cref{ZCorrespondence} for each of the terms $\pi_1^*(\O_{\ZZ_-(x_i)})$ and collecting terms, we obtain
\begin{equation*}
\begin{split}
    &\chi\left(\tilde{\mathcal{N}}^1_N,\pi_2^*(\O_{\V(\Lambda^\sharp)})\otimes^\L\pi_2^*(\O_{\ZZ_-(x_1)})\otimes^\L\cdots\otimes^\L\pi_2^*(\O_{\ZZ_-(x_d)})\right)\\
    &\qquad=(-1)^d\chi\left(\tilde{\mathcal{N}}^1_N,\mathcal{D}\otimes^\L\mathrm{Exp}(x_1)\otimes^\L\cdots\otimes^\L\mathrm{Exp}(x_d)\right).
\end{split}
\end{equation*}
By the projection formula for $\pi_2,$ the left hand side is $\mathrm{Int}_{(x_1,\ldots,x_d)}(\Lambda^\sharp),$ which by \Cref{LiZhangComp} is
\begin{equation*}
    \mathrm{Int}_{(x_1,\ldots,x_d)}(\Lambda^\sharp)=c(d-[\Lambda+L\colon\Lambda])=c(d-[\Lambda^\circ\colon\Lambda])=c(0)=1.
\end{equation*}
The right hand side, by \Cref{KRConj}(v), is
\begin{equation*}
    \sum_{\substack{\Lambda'\in\mathrm{Vert}^0(\mathbb{V}^{[1]}_N)\\L\subseteq\Lambda'}}\mathrm{mult}(\Lambda,\Lambda').
\end{equation*}
As $\mathrm{mult}(\Lambda,\Lambda')$ can only be nonzero if $\Lambda\subseteq\Lambda',$ and since $\Lambda^\circ=\Lambda+L,$ this is
\begin{equation*}
    \sum_{\substack{\Lambda'\in\mathrm{Vert}^0(\mathbb{V}^{[1]}_N)\\\Lambda^\circ\subseteq\Lambda'}}\mathrm{mult}(\Lambda,\Lambda')=\mathrm{mult}(\Lambda,\Lambda^\circ).
\end{equation*}

Thus we conclude that
\begin{equation*}
    \mathcal{D}=\sum_{\substack{\Lambda^\circ\in\mathrm{Vert}^0(\mathbb{V}^{[1]}_N)\\\Lambda\subseteq\Lambda^\circ}}\O_{H^{N-d-1}_{\mathbb{P}_{\Lambda^\circ}}}
\end{equation*}
and the claim follows.
\end{proof}

\subsubsection{Almost self-dual case: ground strata}
\begin{definition}\label{RZintground}
Given $\underline{x}=(x_1,\ldots,x_d)\in(\mathbb{V}^{[0]}_{N})^d,$ for some $d\in\Z_{\ge1}$ with $2d\le N,$ we consider the function
\begin{equation*}
    \nabla^\bullet_{\underline{x}}\colon\mathrm{Vert}^{2d}(\mathbb{V}^{[1]}_{N})\to\Q
\end{equation*}
given by
\begin{equation*}
    \nabla^\bullet_{\underline{x}}(\Lambda)\defeq\chi(\mathcal{N}_{N}^1,{}^\L\ZZ_-(\underline{x})\otimes^\L\V(\Lambda)).
\end{equation*}
\end{definition}

Under the notation in \Cref{SubsectionAuxRZ}, we have an identification $\mathbb{V}_{N+1}^{[0]}=\mathbb{V}_N^{[1]}\obot\O_Fx_0.$

\begin{theorem}\label{nablabullet}
Given $\underline{x}=(x_1,\ldots,x_d)\in(\mathbb{V}^{[1]}_{N})^d,$ we consider $\underline{x}^\sharp\defeq(x_0,x_1,\ldots,x_r)\in(\mathbb{V}^{[0]}_{N+1})^{d+1}.$ Similarly, given $\Lambda\in\mathrm{Vert}^{2d}(\mathbb{V}^{[1]}_N),$ we denote $\Lambda^\sharp\defeq\Lambda\oplus\O_Fx,$ for which $\Lambda^\sharp\in\mathrm{Vert}^{2d+1}(\mathbb{V}^{[0]}_{N+1}).$ Then we have
\begin{equation*}
    \nabla_{\underline{x}}^\bullet(\Lambda)=\frac{1}{q+1}\left(\Int_{\underline{x}^\sharp}(\Lambda^\sharp)-\#\left\{\Lambda^\circ\in\mathrm{Vert}^0(\mathbb{V}^{[1]}_N)\colon \Lambda+L\subseteq\Lambda^\circ\right\}\right).
\end{equation*}
where we denote $L=\mathrm{span}_{\O_F}(\underline{x}).$
\end{theorem}
\begin{proof}
We start with the projection formula for $\pi_1$:
\begin{equation*}
    \nabla^\bullet_{\underline{x}}(\Lambda)=\frac{1}{\deg(\pi_1)}\chi\left(\tilde{\mathcal{N}}_N^1,\pi_1^*(\O_{\ZZ(x_1)})\otimes^\L\cdots\otimes^\L\pi_1^*(\O_{\ZZ(x_r)})\otimes^\L\pi_1^*(\V(\Lambda))\right).
\end{equation*}
Note that $\deg(\pi_1)=q+1$ by \Cref{KRConj}(ii). By \Cref{ZCorrespondence,VCorrespondence}, and the observation \eqref{DerIntObservation}, this is
\begin{equation*}
    \frac{1}{q+1}(A+B)
\end{equation*}
where
\begin{equation*}
    A\defeq\chi\left(\tilde{\mathcal{N}}^1_N,\pi_2^*(\O_{\ZZ^\flat(x_1)})\otimes^\L\cdots\otimes^\L\pi_2^*(\O_{\ZZ^\flat(x_d)})\otimes^\L\pi_2^*(\O_{\V(\Lambda^\sharp)})\right)
\end{equation*}
and
\begin{equation*}
    B\defeq(-1)^{d+1}\chi\left(\tilde{\mathcal{N}}^1_N,\O_{\mathrm{Exp}(x_1)}\otimes^\L\cdots\otimes^\L\O_{\mathrm{Exp}(x_d)}\otimes^\L\O_{\mathrm{Exp}(\Lambda)}\right).
\end{equation*}

By the projection formula for $\pi_2$ it follows that
\begin{equation*}
    A=\chi\left(\ZZ(x_0),{}^\L\ZZ(L)\otimes^\L\O_{\V(\Lambda^\sharp)}\right)=\chi\left(\mathcal{N}_{N+1},^{\L}\ZZ(L^\sharp)\otimes^\L\O_{\V(\Lambda^\sharp)}\right)=\mathrm{Int}_{\underline{x}^\sharp}(\Lambda^\sharp).
\end{equation*}

By \Cref{KRConj}(v), it follows that
\begin{equation*}
    B=(-1)^{d+1}\sum_{\substack{\Lambda^\circ\in\mathrm{Vert}^0(\mathbb{V}^{[1]}_N)\\x_1,\ldots,x_d\in\Lambda^\circ\\\Lambda\subseteq\Lambda^\circ}}(-1)^d=-\#\left\{\Lambda^\circ\in\mathrm{Vert}^0(\mathbb{V}^{[1]}_N)\colon \Lambda+L\subseteq\Lambda^\circ\right\}.\qedhere
\end{equation*}
\end{proof}

\subsubsection{Almost self-dual case: total intersection}
We assume $N=2r$ is even for this subsection. We denote
\begin{equation*}
    \mathrm{Vert}^\circ(\mathbb{V}^{[1]}_{2r})\defeq\mathrm{Vert}^0(\mathbb{V}^{[1]}_{2r}),\quad \mathrm{Vert}^\bullet(\mathbb{V}^{[1]}_{2r})\defeq\mathrm{Vert}^{2r}(\mathbb{V}^{[1]}_{2r}).
\end{equation*}
We will usually denote their elements by $\Lambda^\circ\in\mathrm{Vert}^\circ(\mathbb{V}^{[1]}_{2r})$ and $\varpi\Lambda^\bullet\in\mathrm{Vert}^\bullet(\mathbb{V}^{[1]}_{2r}),$ that is, such that
\begin{equation*}
    (\Lambda^\circ)^\vee=\Lambda^\circ,\quad(\varpi\Lambda^\bullet)^\vee=\Lambda^\bullet.
\end{equation*}

\begin{definition}
    We denote by
    \begin{equation*}
        T^{\circ\bullet}\colon\mathrm{Map}(\mathrm{Vert}^{\bullet}(\mathbb{V}^{[1]}_{2r}),\Q)\to\mathrm{Map}(\mathrm{Vert}^{\circ}(\mathbb{V}^{[1]}_{2r}),\Q)
    \end{equation*}
    and
    \begin{equation*}
        T^{\bullet\circ}\colon\mathrm{Map}(\mathrm{Vert}^{\circ}(\mathbb{V}^{[1]}_{2r}),\Q)\to\mathrm{Map}(\mathrm{Vert}^{\bullet}(\mathbb{V}^{[1]}_{2r}),\Q)
    \end{equation*}
    the correspondences induced by the subset
    \begin{equation*}
        \{(\Lambda^\circ,\varpi\Lambda^\bullet)\colon \varpi\Lambda^\bullet\subseteq\Lambda^\circ\subseteq\varpi\Lambda^\bullet\}\subseteq\mathrm{Vert}^{\circ}(\mathbb{V}^{[1]}_{2r})\times\mathrm{Vert}^{\bullet}(\mathbb{V}^{[1]}_{2r}).
    \end{equation*}
    We also denote $I^\circ\defeq T^{\circ\bullet}\circ T^{\bullet\circ}\colon \mathrm{Map}(\mathrm{Vert}^{\circ}(\mathbb{V}^{[1]}_{2r}),\Q)\to\mathrm{Map}(\mathrm{Vert}^{\circ}(\mathbb{V}^{[1]}_{2r}),\Q).$
\end{definition}
\begin{definition}\label{RZinttotal}
For $\underline{x}\in(\mathbb{V}^{[1]}_{2r})^r,$ we define the \emph{total intersection} to be the function
\begin{equation*}
    \nabla_{\underline{x}}\colon\mathrm{Vert}^\circ(\mathbb{V}^{[1]}_{2r})\to\Q
\end{equation*}
given by
\begin{equation*}
    \nabla_{\underline{x}}\defeq (q+1)T^{\circ\bullet}(\nabla^\bullet_{\underline{x}})+I^\circ(\nabla^\circ_{\underline{x}}).
\end{equation*}
\end{definition}

\begin{theorem}\label{RZnabla}
Given $\underline{x}=(x_1,\ldots,x_r)\in(\mathbb{V}^{[1]}_{2r})^r,$ denote $L\defeq\mathrm{span}_{\O_F}(\underline{x}).$ Then we have
\begin{equation*}
    \nabla_{\underline{x}}=T^{\circ\bullet}(c_L^\bullet)
\end{equation*}
where $c_L^\bullet\colon\mathrm{Vert}^\bullet(\mathbb{V}^{[1]}_{2r})\to\Z$ is given by
\begin{equation*}
    c_L(\varpi\Lambda^\bullet)=\begin{cases}
        c(r-[\varpi\Lambda^\bullet+L\colon\varpi\Lambda^\bullet])&\text{if }L\subseteq L^\vee,\ L\subseteq\Lambda^\bullet,\\0&\text{otherwise.}
    \end{cases}
\end{equation*}
\end{theorem}
\begin{proof}
\Cref{nablacirc} and \Cref{LiZhangComp,nablabullet} imply that
\begin{equation*}
    \left((q+1)\nabla_{\underline{x}}^\bullet+T^{\bullet\circ}(\nabla_{\underline{x}}^\circ)\right)(\Lambda^\bullet)=\Int_{\underline{x}^\sharp}(\Lambda^{\bullet,\sharp}),
\end{equation*}
but note that
\begin{equation*}
    [\varpi(\Lambda^\bullet)^\sharp+\mathrm{span}_{\O_F}(\underline{x}^\sharp)\colon\varpi(\Lambda^\bullet)^\sharp]=[\varpi\Lambda^\bullet+\mathrm{span}_{\O_F}(\underline{x})\colon\varpi\Lambda^\bullet],
\end{equation*}
and thus the claim follows.
\end{proof}

\begin{remark}\label{nablaNotation}
    In particular, the function $\nabla_{\underline{x}}$ only depends on $L=\mathrm{span}_{\O_F}(\underline{x}),$ so we also denote $\nabla_{\underline{x}}=\nabla_L.$
\end{remark}
\begin{remark}
Throughout this whole section, we used the $\ZZ_-$ cycles in the definitions of the intersection numbers $\nabla_{\underline{x}}^?$ for $?\in\{\ ,\circ,\bullet\}.$ We can also compute the intersections by using the $\ZZ_+$ cycles, or a combination of both. To be precise, fix a $\epsilon\in\{+,-\}^d,$ and consider the cycle
\begin{equation*}
    \ZZ_\epsilon(\underline{x})=\ZZ_{\epsilon_1}(x_1)\otimes^\L\dots\otimes^\L\ZZ_{\epsilon_r}(x_d).
\end{equation*}
Using these, we can define $\nabla^\circ_{\underline{x},\epsilon},$ $\nabla^\bullet_{\underline{x},\epsilon}.$ If $(n,d)=(2r,r),$ we can also define $\nabla_{\underline{x},\epsilon}.$ If $p(\epsilon)$ is the number of $+$ signs, then \Cref{nablacirc} easily changes to
\begin{equation*}
    \nabla^\circ_{\underline{x}}(\Lambda^\circ)=\begin{cases}(-q)^{p(\epsilon)}&\text{if }L\subseteq\Lambda^\circ,\\0&\text{otherwise.}\end{cases}
\end{equation*}
A simple modification of the proof of \Cref{LiZhangComp} (that is, \cite[Theorem 10.4.3]{Li-Zhang}) give us
\begin{equation*}
    \pi^{-1}(\ZZ^1_+(x))=\pi_2^{-1}(\ZZ^\flat(x))+q\mathrm{Exp}(x).
\end{equation*}
With that, \Cref{nablabullet} changes to
\begin{equation*}
    \nabla_{\underline{x},\epsilon}^\bullet(\Lambda)=\frac{1}{q+1}\left(\Int_{\underline{x}^\sharp}(\Lambda^\sharp)-(-q)^{p(\epsilon)}\cdot\#\left\{\Lambda^\circ\in\mathrm{Vert}^0(\mathbb{V}^{[1]}_N)\colon \Lambda+L\subseteq\Lambda^\circ\right\}\right).
\end{equation*}
In particular, in the case $(n,d)=(2r,r),$ we see that the total intersection
\begin{equation*}
    \nabla_{\underline{x},\epsilon}=\nabla_{\underline{x}}
\end{equation*}
is independent of the choice of $\epsilon.$ Indeed, the intersection numbers $\nabla_{\underline{x}}$ are precisely what will show up in an application of \Cref{KtheoryAJ}. Hence this invariance is consistent with the fact that different derived integral models can be used in \Cref{KtheoryAJ} to compute the same quantity.
\end{remark}

\section{Unitary Shimura varieties and Friedberg--Jacquet special cycles}\label{ShimuraChapter}

We fix $F/F^+$ a CM extension of number fields, where $F\subseteq\C.$ We denote $\tau_\infty\colon F\hookrightarrow\C$ the place induced by the fixed inclusion $F\subseteq\C,$ and by $\tau_\infty^+\colon F^+\hookrightarrow\R$ its restriction to $F^+.$

For a rational place $w,$ we denote $\Sigma_w$ resp. $\Sigma_w^+$ the set of places of $F$ resp. $F^+$ above $w.$ Note that $\Sigma_\infty$ caries an action of $\mathrm{Aut}(\C/\Q).$

\subsection{Set-up}\label{Shimura-Setup}
\subsubsection{Unitary Shimura varieties}
\begin{definition}\label{neatDef}
    Consider a linear algebraic group\footnote{We fix such an embedding, but the definitions below do not depend on such choice.} $G\hookrightarrow\mathrm{GL}_N$ over a number field $E,$ and $\square$ a finite set of nonarchimedean places of $E.$
    \begin{enumerate}
        \item For a nonarchimedean place $v$ of $E$ and $g_v\in G(E_v),$ we let $\Gamma(g_v)\subseteq\overline{E}_v^\times$ be the subgroup generated by the eigenvalues of $g_v\in\mathrm{GL}_N(E_v),$ whose torsion subgroup $\Gamma(g_v)_{\mathrm{tors}}$ lies in $\overline{\Q}^\times.$ We say $g=(g_v)\in G(\A_E^{\square,\infty})$ is \emph{neat} if $\bigcap_{v\not\in\square}\Gamma(g_v)=\{1\},$ and a subgroup $K\subseteq G(\A_E^{\square,\infty})$ is \emph{neat} if all its elements are neat.
        \item We define a category $\mathfrak{K}(G)^\square$ whose objects are neat open compact subgroups $K\subseteq G(\A_E^{\square,\infty}),$ and a morphism from $K$ to $K'$ is an element $g\in K\backslash G(\A_E^{\square,\infty})\slash K'$ satisfying $g^{-1}Kg\subseteq K'.$
    \end{enumerate}
    When $W$ is a hermitian or skew-hermitian space and $G$ is understood from context as $U(W)$ or $GU(W),$ we will denote the above categories as simply $\mathfrak{K}(W)^\square.$
\end{definition}

\begin{definition}\label{StandardDef}
    Let $V$ be a hermitian space over $F/F^+$ of rank $N.$
    \begin{enumerate}
        \item We say that $V$ is \emph{standard definite} if it has signature $(N,0)$ at every place in $\Sigma_\infty^+.$
        \item We say that $V$ is \emph{standard indefinite} if it has signature $(N-1,1)$ at $\tau_\infty^+$ and $(N,0)$ at other places in $\Sigma_\infty^+.$
    \end{enumerate}
\end{definition}

For a standard indefinite hermitian space $V$ over $F/F^+$ of rank $N,$ we have a functor
\begin{equation*}
    \Sh(V,\blank)\colon \mathfrak{K}(V)\to\Sch{F}
\end{equation*}
of Shimura varieties for the reductive group $\mathrm{Res}_{F^+/\Q}U(V)$ and the Deligne homomorphism
\begin{equation*}
    h\colon \mathrm{Res}_{\C/\R}\mathbb{G}_m\to(\mathrm{Res}_{F^+/\Q}U(V))\otimes_\Q\R=U(V)(F_{\tau_\infty^+}^+)\times\prod_{\substack{\tau^+\in\Sigma_\infty^+\\\tau^+\neq\tau_\infty^+}}U(V)(F_{\tau^+}^+)
\end{equation*}
given by
\begin{equation*}
    h(z)=\left(\begin{pmatrix}
        I_{N-1}&\\&z^\cplx/z
    \end{pmatrix},I_N,\ldots,I_N\right)
\end{equation*}
where we identify $U(V)(F_{\tau_\infty^+}^+)\hookrightarrow\mathrm{GL}_N(\C)$ via $\tau_\infty\colon F\hookrightarrow\C.$ As usual, for $K\in\mathfrak{K}(V),$ we have
\begin{equation*}
    \Sh(V,K)(\C)=U(V)(F^+)\backslash\left(V(\C)_-/\C^\times\times U(V)(\A_{F^+}^\infty)\right)\slash K
\end{equation*}
where $V(\C)_-$ is the space of negative definite lines in $V\otimes_{F,\tau_\infty}\C.$

For a standard definite hermitian space $V$ over $F/F^+$ of rank $n,$ we consider the functor
\begin{equation*}
    \Sh(V,\blank)\colon\mathfrak{K}(V)\to\Set
\end{equation*}
given by
\begin{equation*}
    \Sh(V,K)\defeq U(V)(F^+)\backslash U(V)(\A_{F^+}^\infty)\slash K.
\end{equation*}

\subsubsection{Unitary abelian schemes}
We follow \cite[Sections 3.3, 3.4]{LTXZZ}.

\begin{definition}\label{CMdef}
    A \emph{generalized CM type} of rank $N$ is an element
    \begin{equation*}
        \Psi=\sum_{\tau\in\Sigma_\infty}r_\tau\cdot \tau\in\Z_{\ge0}[\Sigma_\infty]
    \end{equation*}
    satisfying $r_\tau+r_{\tau^\cplx}=N$ for every $\tau\in\Sigma_\infty.$ Its \emph{reflex field} $F_\Psi\subseteq\C$ is the fixed subfield of the stabilizer of $\Psi$ in $\mathrm{Aut}(\C/\Q).$ A \emph{CM type} is a generalized CM type of rank $1,$ and we say that it contains $\tau$ if $r_\tau=1.$
\end{definition}
\begin{definition}
    The reflexive closure $F_\rflx$ of $F$ is the subfield of $\C$ generated by $F$ and $F_\Phi$ for every CM type $\Phi.$ We denote $F_\rflx^+\defeq(F_\rflx)^{\cplx=1}.$
\end{definition}
\begin{remark}
    $F_\rflx/F_\rflx^+$ is a CM extension of number fields, with $F_\rflx$ Galois over $\Q.$ We have $F_\rflx=F$ if either $F$ is Galois over $\Q$ or contains a quadratic imaginary field.
\end{remark}
\begin{definition}\label{DefSpecialInert}
    A prime $\p$ of $F^+$ is \emph{special inert} if it satisfies:
    \begin{enumerate}
        \item $\p$ is inert in $F,$
        \item the underlying rational prime $p$ of $\p$ is odd and unramified in $F,$
        \item $\p$ is of degree over $\Q,$ that is, $F_\p^+=\Q_p.$
    \end{enumerate}
    We say that $\p$ is \emph{very special inert} if moreover there exist a prime $\p'$ of $F^+_{\mathrm{rflx}}$ above $\p$ satisfying $(F^+_{\mathrm{rflx}})_{\p'}=\Q_p.$
\end{definition}
\begin{remark}
    As mentioned in \cite[Remark 3.3.5]{LTXZZ}, the condition (3) should not be necessary throughout the article, and was imposed there only to simplify notations on Dieudonn\'e modules.
\end{remark}

In what follows, we denote by $\mathbb{P}$ a subring of $\Q.$
\begin{definition}\label{UAbDef}
    For $S\in\Sch{\mathbb{P}},$ we consider the following categories.
    \begin{enumerate}
        \item $\Ab{\mathbb{P}}(S)$ whose objects are $(A,\iota)$ where $A$ is an abelian scheme over $S,$ and $\iota\colon \O_F\to\mathrm{End}_S(A)\otimes\mathbb{P}$ is a homomorphism of unital algebras. A morphism $\varphi\colon (A,\iota_A)\to(B,\iota_B)$ consists of a quasi-homomorphism $\varphi\in\mathrm{Hom}_S(A,B)\otimes\Q$ such that i) $\varphi\circ\iota_A=\iota_B\circ\varphi$ and ii) there exist $c\in\mathbb{P}^\times$ such that $c\varphi\in\mathrm{Hom}_S(A,B)$ is a homomorphism. We consider the functor $(\blank)^\vee\colon\Ab{\mathbb{P}}(S)\to\Ab{\mathbb{P}}(S)^\op$ given by $(A,\iota)^\vee=(A^\vee,(\iota\circ\cplx)^\vee).$
        \item $\UAb{\mathbb{P}}(S)$ whose objects are $(A,\iota,\lambda)$ for $(A,\iota)\in\Ab{\mathbb{P}}(S)$ and
        \begin{equation*}
            \lambda\in\mathrm{Hom}_{\Ab{\mathbb{P}}(S)}((A,\iota),(A,\iota)^\vee).
        \end{equation*}
        satisfies $\lambda=\lambda^\vee.$ In particular, $\lambda$ satisfies $\iota(a^\cplx)^\vee\circ\lambda=\lambda\circ\iota(a)$ for all $a\in\O_F.$ A morphism $\varphi\colon (A,\iota_A,\lambda_A)\to(B,\iota_B,\lambda_B)$ is $\varphi\in\mathrm{Hom}_{\Ab{\mathbb{P}}(S)}((A,\iota_A),(B,\iota_B))$ such that there exists $c\in\mathbb{P}^\times$ with $\varphi^\vee\circ\lambda_B\circ\varphi=c\lambda_A.$ For such a morphism $\varphi,$ we denote $c$ by $c(\varphi)$ its \emph{similitude factor}.
        \item $\IAb{\mathbb{P}}(S)$ the wide subcategory\footnote{A \emph{wide subcategory} $\mathcal{C}$ of $\mathcal{D}$ is such that $\mathrm{Obj}(\mathcal{C})=\mathrm{Obj}(\mathcal{D}).$} of $\Ab{\mathbb{P}}(S)$ where morphisms $\varphi\colon(A,\iota_A)\to(B,\iota_B)$ are required to be quasi-isogenies.
        \item $\IUAb{\mathbb{P}}(S)$ the wide subcategory of $\UAb{\mathbb{P}}(S)$ where morphisms $\varphi\colon(A,\iota_A,\lambda_A)\to(B,\iota_B,\lambda_B)$ are required to be quasi-isogenies.
    \end{enumerate}
    We call objects of $\Ab{\mathbb{P}}(S)$ (or $\IAb{\mathbb{P}}(S)$) \emph{$\O_F$-abelian schemes over $S$}, and objects of $\UAb{\mathbb{P}}(S)$ (or $\IUAb{\mathbb{P}}(S)$) \emph{unitary $\O_F$-abelian schemes over $S.$}
\end{definition}
\begin{definition}
    Let $\Psi=\sum_{\tau\in\Sigma_\infty}r_\tau\cdot\tau$ be a generalized CM type of rank $N.$ For $S\in\Sch{\O_{F_\Psi}\otimes\mathbb{P}},$ we say that an $\O_F$-abelian scheme $(A,\iota)$ over $S$ has \emph{signature type $\Psi$} if for every $a\in\O_F,$ the characteristic polynomial of $\iota(a)$ on $\mathrm{Lie}_{A/S}$ is given by
    \begin{equation*}
        \mathrm{charpoly}(\iota(a)\mid\mathrm{Lie}_{A/S})=\prod_{\tau\in\Sigma_\infty}(T-\tau(a))^{r_\tau}\in\O_S[T].
    \end{equation*}
\end{definition}
\begin{definition}
    Let $\Phi$ be a CM type, $p$ be a rational prime and $K$ be an algebraically closed field of characteristic different than $p$ which is a $\O_{F_\Phi}\otimes\mathbb{P}$-ring. For $(A_0,\iota_0,\lambda_0)$ an unitary $\O_F$-abelian scheme of signature type $\Phi$ over $K,$ we define a $\A_F^{\infty,p}=F\otimes_\Q\A_\Q^{\infty,p}$ skew-hermitian space $\hat{V}^p(A_0)$ given by
    \begin{equation*}
        \hat{V}^p(A_0)\defeq H_1^\et(A_0,\A_\Q^{\infty,p})
    \end{equation*}
    equipped with the pairing
    \begin{equation*}
        (x,y)=\langle x,\lambda_{0,*}y\rangle_{A_0}\in\A_F^{\infty,p}
    \end{equation*}
    where $\langle\cdot,\cdot\rangle_{A_0}\colon H^1_\et(A_0,\A_\Q^{\infty,p})\times H^1_\et(A_0^\vee,\A_\Q^{\infty,p})\to\A_F^{\infty,p}$ is the Weil pairing.
\end{definition}
\begin{definition}\label{DefVhat}
    Let $\Phi$ be a CM type, $p$ be a rational prime and $K$ be an algebraically closed field of characteristic different than $p$ which is a $\O_{F_\Phi}\otimes\mathbb{P}$-ring. For $(A_0,\iota_0,\lambda_0),(A,\iota,\lambda)$ two unitary $\O_F$-abelian schemes over $K,$ where $A_0$ has signature type $\Phi,$ we define a $\A_F^{\infty,p}=F\otimes_\Q\A_\Q^{\infty,p}$ hermitian space $\hat{V}^p(A_0,A)$ given by
    \begin{equation*}
        \hat{V}^p(A_0,A)\defeq\mathrm{Hom}_{F\otimes_\Q\A_\Q^{\infty,p}}\left(H_1^\et(A_0,\A_\Q^{\infty,p}),H_1^\et(A,\A_\Q^{\infty,p})\right)
    \end{equation*}
    equipped with the pairing
    \begin{equation*}
        (x,y)\defeq\iota_0^{-1}\left((\lambda_{0,*})^{-1}\circ y^\vee\circ\lambda_*\circ x\right)\in\iota_0^{-1}\mathrm{End}_{F\otimes_\Q\A_\Q^{\infty,p}}\left(H_1^\et(A_0,\A_\Q^{\infty,p})\right)=\A_F^{\infty,p}.
    \end{equation*}
\end{definition}

\subsubsection{CM moduli scheme}
Let $p$ be a rational prime that is unramified in $F.$ We follow \cite[Section 3.5]{LTXZZ}.
\begin{definition}
We consider a subtorus $T_0\subseteq\mathrm{Res}_{\O_F/\Z}(\mathbb{G}_m)\otimes\Z_{(p)}$ such that for every $\Z_{(p)}$ ring $R$ we have
\begin{equation*}
    T_0(R)=\{a\in\O_F\otimes R\colon \mathrm{Nm}_{F/F^+}(a)\in R^\times\}.
\end{equation*}
\end{definition}
\begin{remark}\label{Ker1Remark}
    Let $W_0$ be a skew-hermitian space over $\O_F\otimes\Z_{(p)}.$ Then $GU(W_0)$ is canonically isomorphic to $T_0.$ Moreover, the set of similarity classes of rank $1$ $\O_F\otimes\Z_{(p)}$ skew-hermitian spaces that are locally similar to $W_0$ is canonically isomorphic to
    \begin{equation*}
    \ker^1(T_0)\defeq\ker\left(H^1(\Q,T_0)\to\prod_vH^1(\Q_v,T_0)\right).
    \end{equation*}
\end{remark}
\begin{definition}
Let $\Phi$ be a CM type. A skew-hermitian space $W_0$ over $\O_F\otimes\Z_{(p)}$ of rank $1$ has \emph{type} $\Phi$ if for every $x\in W_0$ and totally imaginary $a\in F^\times$ satisfying $\Im\tau(a)>0$ for all $\tau\in\Phi,$ we have $\langle ax,x\rangle_{W_0}\ge0.$
\end{definition}

\begin{definition}[{\cite[Definition 3.5.4]{LTXZZ}}]
    Let $\Phi$ be a CM type and $W_0$ a skew-hermitian space over $\O_F\otimes\Z_{(p)}$ of rank $1$ and type $\Phi.$ For an open compact $K_0^p\subseteq T_0(\A_\Q^{\infty,p}),$ we define a functor $\mathcal{T}_p^1(W_0,\blank)\colon\mathfrak{K}(W_0)^p\to\PSch{\O_{F_\Phi}\otimes\Z_{(p)}}$ where for a test scheme $S,$ we let $\mathcal{T}_p^1(W_0,K_0^p)(S)$ be the groupoid of triples $(A_0,\iota_0,\lambda_0,\eta_0^p)$ where
    \begin{itemize}
        \item $(A_0,\iota_0,\lambda_0)\in\IUAb{\Z_{(p)}}(S)$ has signature type $\Phi$ and is such that $\lambda_0$ is an isomorphism in $\IUAb{\Z_{(p)}}(S).$
        \item $\eta_0^p$ is a $K_0^p$-level structure, that is, for a chosen geometric point $s$ on every connected component of $S,$ a $\pi_1(s,S)$-invariant $K_0^p$-orbit of similitude
        \begin{equation*}
            \eta_0^p\colon W_0\otimes_\Q\A_\Q^{\infty,p}\to \hat{V}^p(A_{0,s})
        \end{equation*}
        of skew-hermitian spaces over $\A_F^{\infty,p}.$
    \end{itemize}
    An isomorphism between two triples $(A_0,\iota_0,\lambda_0,\eta_0^p)$ and $(A_0',\iota_0',\lambda_0',\eta_0^{p'})$ is an isomorphism $\varphi_0\colon (A_0,\iota_0,\lambda_0)\to(A_0',\iota_0',\lambda_0')$ in $\IUAb{\Z_{(p)}}$ that carries $\eta_0^p$ to $\eta_0^{p'}.$
\end{definition}
\begin{proposition}[{\cite[Lemma 3.4, Remark 3.5(i),(iv)]{RSZ}}]
    If $K_0^p$ is neat, $\mathcal{T}_p^1(W_0,K_0^p)$ is represented by a scheme which is finite and \'etale over $\Spec\O_{F_\Phi}\otimes\Z_{(p)}.$
\end{proposition}

By \Cref{Ker1Remark}, we obtain a map
\begin{equation*}
    w\colon\mathcal{T}_p^1(W_0,K_0^p)(\C)\to\ker^1(T_0)
\end{equation*}
sending $(A_0,\iota_0,\lambda_0,\eta_0^p)$ to its first homology $H_1(A_0(\C),\Z_{(p)}),$ as a skew-hermitian $\O_F\otimes\Z_{(p)}$ space induced by $\lambda_0.$ Note that it is everywhere locally similar to $W_0$: i) it has signature type $\Phi,$ ii) $H_1(A_0(\C),\Z_{(p)})\otimes_\Q\A_\Q^{\infty,p}$ is isomorphic to $\hat{V}^p(A_0)$ by the singular-\'etale comparison isomorphism, and the existence of $\eta_0^p$ says that this is similar to $W_0\otimes_\Q\A_\Q^{\infty,p}.$
\begin{definition}
    For $K_0^p$ a neat level structure, we define $\mathcal{T}_p(W_0,K_0^p)\subseteq\mathcal{T}_p^1(W_0,K_0^p)$ to be the minimal open and closed subscheme of $\mathcal{T}_p^1(W_0,K_0^p)$ containing $w^{-1}(W_0).$ The group $T_0(\A_\Q^{\infty,p})$ acts on $\mathcal{T}_p(W_0,K_0^p)$ via
    \begin{equation*}
        a\cdot(A_0,\iota_0,\lambda_0,\eta_0^p)=(A_0,\iota_0,\lambda_0,a\circ\eta_0^p),
    \end{equation*}
    with stabilizer $T_0(\Z_{(p)})K_0^p.$ In fact, $T_0(\A_\Q^{\infty,p})/T_0(\Z_{(p)})K_0^p$ is the Galois group of the Galois morphism
    \begin{equation*}
        \mathcal{T}_p(W_0,K_0^p)\to\Spec(\O_{F_\Phi}\otimes\Z_{(p)}).
    \end{equation*}
\end{definition}
\begin{definition}
    For $K_0^p$ a neat level structure, denote by $\mathfrak{T}$ the groupoid associated to the group $T_0(\A_\Q^{\infty,p})/T_0(\Z_{(p)})K_0^p.$ As $\mathcal{T}_p(W_0,K_0^p)$ has an action of $\mathfrak{T},$ we will use the same notation for the following functor
    \begin{equation*}
        \mathcal{T}_p(W_0,K_0^p)\colon\mathfrak{T}\to\Sch{\O_{F_\Phi}\otimes\Z_{(p)}}.
    \end{equation*}
\end{definition}

\subsection{Semi-global RSZ Shimura varieties}\label{Shimura-RSZ}

We fix the following data:
\begin{itemize}[leftmargin=*]
    \item A special inert prime $\p$ of $F^+$ (\Cref{DefSpecialInert}) with underlying rational prime $p.$
    \item A CM type $\Phi$ (\Cref{CMdef}) containing $\tau_\infty.$
    \item A skew-hermitian space $W_0$ of rank $1$ over $\O_F\otimes\Z_{(p)}.$
    \item A neat open compact subgroup $K_0^p\subseteq T_0(\A_\Q^{\infty,p})$ (\Cref{neatDef}).
    \item An isomorphism $i_p\colon\overline{\Q_p}\rightiso\C$ that induces the place $\p$ of $F^+.$
    \item A standard indefinite hermitian space $V$ over $F/F^+$ of rank $N\ge1$ (\Cref{StandardDef}).\footnote{Unlike \cite{LTXZZ}, we prefer to have as starting data the indefinite Hermitian space rather than its $\p$-nearby space.} We will denote $G\defeq U(V)$ as an algebraic group over $F^+.$
\end{itemize}

We denote by $\Q_p^\Phi\subseteq\C$ the composition of $i_p(\Q_p),$ $F$ and $F_\Phi$ (\Cref{CMdef}). We denote $\Z_p^\Phi$ its ring of integers, and $\F_p^\Phi$ its residue field. Since $p$ in unramified in $F,$ $\Q_p^\phi/\Q_p$ is unramified.

\begin{notation}\label{SchemesNotation}
    We will denote schemes over $\Z_p^\Phi$ by calligraphic letters (e.g. $\mathcal{X}$), and denote their base changes to $\Q_p^\Phi$ with a subscript of $\eta$ (e.g. $\mathcal{X}^\eta$) and base changes to $\F_p^\Phi$ by roman letters (e.g. $X$). Furthermore, we will denote the base changes to $\overline{\F}_p$ with an overline (e.g. $\overline{X}$).
\end{notation}

We put $\mathcal{T}_\p\defeq\mathcal{T}_p(W_0,K_0^p)\otimes_{\O_{F_\Phi}\otimes\Z_{(p)}}\Z_p^\Phi.$

\begin{definition}[{\cite[Definition 5.2.1]{LTXZZ}}]\label{moduliProblem}
We consider the functor
\begin{equation*}
    \mathcal{M}_\p(V,\blank)\colon\mathfrak{K}(V)^p\times\mathfrak{T}\to\PSch{\Z_p^\Phi}
\end{equation*}
where for a test scheme $S$ we define $\mathcal{M}_\p(V,K^p)(S)$ to be the groupoid of tuples $(A_0,\iota_0,\lambda_0,\eta_0^p;A,\iota,\lambda,\eta^p)$ where
\begin{enumerate}
    \item $(A_0,\iota_0,\lambda_0,\eta_0^p)\in\mathcal{T}_\p(S)$;
    \item $(A,\iota,\lambda)\in\IUAb{\Z_{(p)}}(S)$ is a unitary $\O_F$-abelian scheme of signature $N\Phi-\tau_\infty+\tau_\infty^\cplx$ over $S,$ such that $\ker\lambda[p^\infty]$ is contained in $A[\p]$ of rank $p^2$;
    \item $\eta^p$ is a $K^p$-level structure: for a chosen geometric point $s$ on every connected component of $S,$ a $\pi_1(S,s)$-invariant $K^p$-orbit of isometries
    \begin{equation*}
        \eta^p\colon V\otimes_\Q\A_\Q^{\infty,p}\rightiso \hat{V}^p(A_{0,s},A_s)
    \end{equation*}
    as $\A_F^{\infty,p}$ Hermitian spaces.
\end{enumerate}
An isomorphism between tuples $(A_0,\iota_0,\lambda_0,\eta_0^p;A,\iota,\lambda,\eta^p)$ and $(A_0',\iota_0',\lambda_0',\eta_0^{p'};A',\iota',\lambda',\eta^{p'})$ is a pair of isomorphisms $\varphi_0\colon(A_0,\iota_0,\lambda_0,\eta_0^p)\to(A_0',\iota_0',\lambda_0',\eta_0^{p'})$ in $\mathcal{T}_\p(S)$ and $\varphi\colon(A,\iota,\lambda)\to(A',\iota',\lambda')$ in $\IUAb{\Z_{(p)}}(S)$ that satisfy $c(\varphi_0)=c(\varphi)$ and that carry $\eta^p$ to $\eta^{p'}.$

A morphism $g\in K^p\backslash G(\A_{F^+}^{\infty,p})/K^{p'}$ of $\mathfrak{K}(V)^p$ maps $\mathcal{M}_\p(V,K^p)(S)$ to $\mathcal{M}_\p(V,K^{p'})$ by changing $\eta^p$ to $\eta^p\circ g,$ and a morphism $a$ of $\mathfrak{T}$ acts on $\mathcal{M}_\p(V,K)(S)$ by changing $\eta^p_0$ to $\eta^p_0\circ a.$
\end{definition}

\begin{theorem}[{\cite[Theorem 5.2.5]{LTXZZ}}]
    For every $K^p\in\mathfrak{K}(V)^p,$ $\mathcal{M}_\p(V,K^p)$ is quasi-projective and strictly semistable over $\mathcal{T}_\p$ of relative dimension $N-1.$
\end{theorem}

\subsubsection{Complex uniformization}
\begin{definition}
    We define a \emph{complex uniformization datum} $z^\eta=(\Lambda_{z^\eta,\q})_{\q\mid p}$ to be, for every prime $\q$ of $F^+$ above $p,$ a $\O_{F_\q}$-lattice $\Lambda_{z^\eta,\q}$ of $V\otimes_{F^+}F^+_\q.$ We assume that for $\q\neq\p,$ the lattice $\Lambda_{z^\eta,\q}$ is self-dual, and for $\q=\p,$ the lattice $\Lambda_{z^\eta,\p}$ is \emph{almost self-dual}. We denote $K_{z^\eta,\q}=\mathrm{Stab}(\Lambda_{z^\eta,\q})\subseteq G(F^+_\q),$ and denote $K_{z^\eta,p}=\prod_{\q\mid p}K_{z^\eta,\q}\subseteq U(V)(F^+\otimes_\Q\Q_p),$ the stabilizer of $\Lambda_{z^\eta,p}=\prod_{\q\mid p}\Lambda_{z^\eta,\q}.$
\end{definition}
\begin{theorem}\label{ComplexUniformization}
    Given a complex uniformization datum $z^\eta,$ we have a ``moduli interpretation isomorphism''
    \begin{equation*}
        \Theta_{z^\eta}\colon\mathcal{M}_\p^\eta(V,\blank)\rightiso\mathrm{Sh}(V,\blank K_p)\times_{\Spec F}\mathcal{T}_\p^\eta.
    \end{equation*}
\end{theorem}
We record its definition on complex points, following the discussion after \cite[Equation (5.2)]{LTXZZ}. Given $z=(A_0,\iota_0,\lambda_0,\eta_0^p;A,\iota,\lambda,\eta^p)\in\mathcal{M}^\eta_\p(V,K^p)(\C),$ we consider
\begin{equation*}
    V_z=\mathrm{Hom}_F\left(H_1(A(\C),\Q),H_1(A_0(\C),\Q)\right)
\end{equation*}
equipped with a hermitian pairing similar to \Cref{DefVhat}. This becomes a standard indefinite hermitian space of rank $N.$ By the comparison between singular and \'etale homology, we have a canonical isometry
\begin{equation*}
    \rho\colon V_z\otimes_\Q\A_\Q^{\infty,p}\rightiso\hat{V}^p(A_0,A).
\end{equation*}
We claim that $V_z$ and $V$ are isometric. Indeed, they are locally isometric since i) they are both standard indefinite, ii) the isometry $\rho$ together with the level structure $\eta^p$ give an isometry outside $\infty$ and $p,$ iii) by the properties of $\lambda_0$ and $\lambda,$ we have a canonical decomposition
\begin{equation*}
    \mathrm{Hom}_{\O_F\otimes\Z_p}\left(H_1^\et(A_0,\Z_p),H_1^\et(A,\Z_p)\right)=\bigoplus_{\q\mid p}\Lambda_{z,\q}
\end{equation*}
of $\O_F\otimes\Z_p$-modules where $\Lambda_{z,\q}$ is self-dual in $V_z\otimes_FF_\q$ for $\q\neq\p,$ and $\Lambda_{z,\p}$ is almost-self dual in $V_z\otimes_FF_\p.$ Now it follows that $V_z$ and $V$ are isometric by the Hasse principle for hermitian spaces. Choose an isometry $\rho_{\mathrm{rat}}\colon V_z\rightiso V.$ Then
\begin{equation*}
    g^p\defeq \eta_{\mathrm{rat}}\circ\rho^{-1}\circ\eta^p\colon V\otimes_\Q\A_\Q^{\infty,p}\rightiso V\otimes_\Q\A_\Q^{\infty,p}
\end{equation*}
is an element of $G(\A_{F^+}^{\infty,p}).$ For every $\q$ above $p,$ choose $g_\q$ such that $g_\q\Lambda_{z^\eta,\q}=\eta_{\mathrm{rat}}\Lambda_{z,\q}.$ This gives an element $g_z=(g^p,(g_\q)_{\q\mid p})\in U(V)(\A_{F^+}^\infty).$ Finally, we consider
\begin{equation*}
    l_z\defeq\{\alpha\in\Hom_F(H_1^{\mathrm{{dR}}}(A_0/\C),H_1^{\mathrm{dR}}(A/\C))\colon \alpha(\omega_{A_0^\vee,\tau_\infty})\subseteq\omega_{A^\vee,\tau_\infty}\}
\end{equation*}
is a line negative definite line. Here, for an $\O_F$-abelian variety $(B,\iota_B),$ $\omega_{B}$ denotes the dual Lie algebra of $B,$ viewed as a submodule of $H_1^{\mathrm{dR}}(B^\vee/\C),$ and $\omega_{B,\tau_\infty}$ denote the $\tau_\infty$-eigencomponent of the $\O_F$ action by $\iota_B.$ Then the first component of $\Theta_{z^\eta}(x)$ is given by the coset
\begin{equation*}
    U(V)(F^+)\left(\eta_{\mathrm{rat}}(l_z),g_zK^pK_p\right),
\end{equation*}
which is independent of the choice of $\eta_{\mathrm{rat}}.$

\subsubsection{Geometry mod \texorpdfstring{$p$}{p} and basic correspondences}
We collect results of \cite[Sections 5.2-5.4]{LTXZZ} about the geometry of the special fiber of $\mathcal{M}_\p.$

Recall that we denote $M_\p(V,\blank)\defeq\mathcal{M}_\p(V,\blank)\otimes_{\Z_p^\Phi}\F_p^\Phi.$
\begin{theorem}\label{Geometrymodp}
We have a decomposition of $M_\p(V,\blank)$ into closed subschemes
\begin{equation*}
    M_\p(V,\blank)=M^\circ_\p(V,\blank)\cup M^\bullet(V,\blank)
\end{equation*}
where $M^\circ_\p(V,\blank)$ and $M^\bullet_\p(V,\blank)$ are smooth over $T_\p.$ We denote by $M^\dagger_\p(V,\blank)$ their intersection. For each $?\in\{\circ,\bullet\},$ there are correspondences
\begin{equation*}
    S_\p^?(V,\blank)\xleftarrow{\pi^?}B_\p^?(V,\blank)\xrightarrow{i^?}M_\p^?(V,\blank)
\end{equation*}
which satisfy the following.
\begin{enumerate}
    \item $S_\p^\circ(V,\blank)\to T_\p$ and $S_\p^\bullet(V,\blank)\to T_\p$ are represented by finite and \'etale schemes.
    \item $\pi^\circ$ is surjective. For a perfect field $\kappa$ containing $\F_p^\Phi,$ the fibers of $\pi^\circ$ are isomorphic to projective spaces $\mathbb{P}^{N-1}_\kappa.$
    \item $i^\circ$ is an isomorphism.
    \item $\pi^\bullet$ is surjective. For a perfect field $\kappa$ containing $\F_p^\Phi,$ the fibers of $\pi^\bullet$ are geometrically irreducible projective smooth schemes over $\kappa$ of dimension $\lfloor N/2\rfloor.$
    \item $i^\bullet$ is, locally on $B^\bullet_\p(V,\blank),$ a closed immersion.
\end{enumerate}
\end{theorem}

Moreover, there are uniformizations of $S^\circ_\p(V,\blank)$ and $S^\bullet_\p(V,\blank)$ as follows.
\begin{definition}
    A \emph{basic uniformization datum} is $z^b=(V^\circ,i,(\Lambda^\circ_\q)_{\q\mid p},\Lambda^\bullet_\p)$ where
    \begin{itemize}
        \item $V^\circ$ is a standard definite hermitian space of rank $N$;
        \item $i\colon V\otimes_\Q\A_f^p\rightiso V^\circ\otimes_\Q\A_f^p$ is an isometry;
        \item for each place $\q$ of $F^+$ above $p,$ $\Lambda_\q\subseteq V^\circ\otimes_{F^+}F^+_\q$ is a self-dual lattice;
        \item $\Lambda^\bullet_\p\subseteq V^\circ\otimes_{F^+}F^+_\p$ is a lattice such that $\Lambda^\circ_\p\subseteq\Lambda^\bullet_\p$ and for a uniformizer $\varpi$ of $\O_{F_\p}$, we have that $\varpi\Lambda^\bullet_\p\subseteq(\Lambda^\bullet_\p)^\vee\subseteq\Lambda^\bullet$ with $\mathrm{length}(\Lambda^\bullet_\p/(\Lambda^\bullet_\p)^\vee)=2\lfloor\frac{N}{2}\rfloor.$
    \end{itemize}
    Given such data, we denote $K_{z^b,\q}^\circ\defeq\mathrm{Stab}(\Lambda^\circ_\q)$ and $K_{z^b,\p}^\bullet\defeq\mathrm{Stab}(\Lambda^\bullet_\p),$ as well as $K_{z^b,p}^\circ=\prod_{\q\mid p}K_\q^\circ$ and $K_{z^b,p}^\bullet=K_\p^\bullet\times\prod_{\q\neq\p}K_\q^\circ.$ For ease of notation, we denote $G^\circ\defeq U(V^\circ).$
\end{definition}
\begin{proposition}[{\cite[Constructions 5.3.6, 5.4.6]{LTXZZ}}]
    Given a basic uniformization datum $z^b,$ there are uniformization maps
    \begin{equation*}
        \Theta_{z^b}\colon S^\circ_\p(V,\blank)(\overline{\F}_p)\rightiso \Sh(V^\circ,i_*(\blank)K_{z^b,p}^\circ)\times T_\p(\overline{\F}_p)
    \end{equation*}
    and
    \begin{equation*}
        \Theta_{z^b}\colon S^\bullet_\p(V,\blank)(\overline{\F}_p)\rightiso \Sh(V^\circ,i_*(\blank)K_{z^b,p}^\bullet)\times T_\p(\overline{\F}_p).
    \end{equation*}
\end{proposition}
\begin{remark}\label{CanonicalBasicUniformization}
    In fact, we have more canonical uniformizations
    \begin{equation*}
        \Theta_{V^\circ,i}\colon S^\circ_\p(V,\blank)(\overline{\F}_p)\rightiso G^\circ(F^+)\backslash\left( \mathrm{Vert}^\circ(V^\circ\otimes_\Q\Q_p)\times G^\circ(\A_{F^+}^{p,\infty})/i_*(K^p)\right)\times T_\p(\overline{\F}_p)
    \end{equation*}
    and
    \begin{equation*}
        \Theta_{V^\circ,i}\colon S^\bullet_\p(V,\blank)(\overline{\F}_p)\rightiso G^\circ(F^+)\backslash \left(\mathrm{Vert}^\bullet(V^\circ\otimes_\Q\Q_p)\times G^\circ(\A_{F^+}^{p,\infty})/i_*(K^p)\right)\times T_\p(\overline{\F}_p)
    \end{equation*}
    where we define (\Cref{DefVertexLattice})
    \begin{equation*}
        \mathrm{Vert}^\circ(V^\circ\otimes_\Q\Q_p)\defeq\prod_{\q\mid p}\mathrm{Vert}^0(V^\circ\otimes_{F^+}F^+_\q)
    \end{equation*}
    and
    \begin{equation*}
        \mathrm{Vert}^\bullet(V^\circ\otimes_\Q\Q_p)\defeq\prod_{\substack{\q\mid p\\\q\neq\p}}\mathrm{Vert}^0(V^\circ\otimes_{F^+}F^+_\q)\times \mathrm{Vert}^{2\lfloor\frac{N}{2}\rfloor}(V\otimes_{F^+}F^+_\p).
    \end{equation*}
    The choice of lattices $\Lambda_\q^\circ$ for $\q\mid p$ and $\Lambda^\bullet_\p$ then identifies
    \begin{equation*}
        \mathrm{Vert}^0(V^\circ\otimes_{F^+}F^+_\q)=G^\circ(F^+_\q)/K^\circ_{z^b,\q},\quad \mathrm{Vert}^{2\lfloor\frac{N}{2}\rfloor}(V\otimes_{F^+}F^+_\p)=G^\circ(F^+_\p)/K^\bullet_{z^b,\p},
    \end{equation*}
    giving the uniformizations $\Theta_{z^b}.$
\end{remark}

\subsubsection{\texorpdfstring{$p$}{p}-adic uniformization}
\begin{definition}
    The \emph{supersingular locus} $M_\p^{ss}(V,\blank)\subseteq M_\p(V,\blank)$ is the closed locus where the abelian variety $A$ is supersingular. We denote by
    \begin{equation*}
        \mathcal{M}_\p^{\wedge,ss}(V,\blank)\colon\mathfrak{K}(V)^p\times\mathfrak{T}\to\PSch{\breve{\Z}_p}
    \end{equation*}
    the formal completion of $\mathcal{M}_\p(V,\blank)$ along the supersingular locus, base changed to $\Spf\breve{\Z}_p.$
\end{definition}

\begin{definition}
    We define a \emph{$p$-adic uniformization datum} for $\mathcal{M}_\p(V,K^p)$ to be a choice of point $\mathbbm{z}\in M_\p^{ss}(V,K^p)(\overline{\F}_p).$ Write $\mathbbm{z}=(\mathbb{A}_0,\bbiota_0,\bblambda_0,\bbeta_0^p,\mathbb{A},\bbiota,\bblambda,\bbeta^p).$ Then we consider (see \Cref{pDivDef})
    \begin{equation*}
        (\mathbb{X}_0,\iota_{\mathbb{X}_0},\lambda_{\mathbb{X}_0})\defeq(\mathbb{A}_0[\p^\infty],\bbiota_0,\bblambda_0)\in\pDiv{F_\p/F^+_\p}{[0],(0,1)}(\overline{\F}_p)
    \end{equation*}
    and
    \begin{equation*}
        (\mathbb{X},\iota_{\mathbb{X}},\lambda_{\mathbb{X}})\defeq(\mathbb{A}[\p^\infty],\bbiota,\bblambda)\in\pDiv{F_\p/F^+_\p}{[1],(N-1,1)}(\overline{\F}_p),\quad (\mathbb{A}[\q^\infty],\bbiota,\bblambda)\in\pDiv{F_\q/F^+_\p}{[0],(N,0)}(\overline{\F}_p)\text{ for }\q\neq\p.
    \end{equation*}
    We denote
    \begin{equation*}
        \mathbb{V}\defeq\mathrm{Hom}_{\Ab{\Z_{(p)}}(\overline{F}_p)}((\mathbb{A}_0,\bbiota_0),(\mathbb{A},\bbiota))\otimes\Q
    \end{equation*}
    which is a standard definite hermitian $F/F^+$ space equipped with the pairing
    \begin{equation*}
        (x,y)\defeq\bbiota_0^{-1}(\bblambda_{0,*}\circ y^\vee\circ\bblambda_*\circ x)\in\bbiota_0^{-1}(\mathrm{End}_{\O_F}(\mathbb{A}_0)\otimes\Q)=F.
    \end{equation*}
    For ease of notation, we denote $\mathbb{G}\defeq U(\mathbb{V}).$
\end{definition}

\begin{remark}
    Note that for $p$-adic uniformization datum $\mathbbm{z}\in M_\p^{ss}(V,K^p)(\overline{\F}_p),$ the spaces $V$ and $\mathbb{V}$ are nearby at $\p.$ Indeed, $\bbeta^p\colon V\otimes_\Q\A_\Q^{\infty,p}\rightiso \mathbb{V}\otimes_\Q\A_\Q^{\infty,p}$ is a $K^p$-orbit of isometries, $\mathbb{V}$ is standard definite by the positivity of the Rosati involution, and for each $\q\neq\p$ above $p,$ the lattices
    \begin{equation*}
        \mathrm{Hom}_{\O_F}((\mathbb{A}_0[\q^\infty],\bbiota_0),(\mathbb{A}[\q^\infty],\bbiota))\subseteq \mathrm{Hom}_{\O_F}((\mathbb{A}_0[\q^\infty],\bbiota_0),(\mathbb{A}[\q^\infty],\bbiota))\otimes\Q=\mathbb{V}\otimes_{F^+}F^+_\q
    \end{equation*}
    are self-dual by the properties of $\bblambda_0$ and $\bblambda.$
\end{remark}

\begin{theorem}[{\cite[Theorem 4.3]{Cho}}]\label{padicUni}
    Given a $p$-adic uniformization datum $\mathbbm{z}\in M_\p^{ss}(V,K^p)(\overline{\F}_p),$ we have a $p$-adic uniformization isomorphism
    \begin{equation*}
        \Theta_{\mathbbm{z}}\colon\mathcal{M}_\p^{\wedge,ss}(V,K^p)\rightiso \mathbb{G}(F^+)\backslash\left(\mathcal{N}\times \mathbb{G}(\A_{F^+}^{p,\infty})/\bbeta^p_*(K^p)\right)\times\mathcal{T}_{\p,\breve{\Z}_p}
    \end{equation*}
    where $\mathcal{N}=\prod_{\q\mid p}\mathcal{N}_\q$ for
    \begin{equation*}
        \mathcal{N}_\p=\N{F/F^+}{[1],(N-1,1)}\quad\text{and}\quad \mathcal{N}_\q=\N{F/F^+}{[0],(N,0)}\text{ for }\q\neq\p,
    \end{equation*}
    the Rapoport--Zink spaces with framing objects $(\mathbb{A}[\q^\infty],\iota,\lambda).$ Note that $\mathbb{G}(F^+_\q)$ acts on $\mathcal{N}_\q$ by \Cref{J=U(V)Remark}.
\end{theorem}
\begin{proof}
This is given in \cite[Theorem 4.3]{Cho}, where we rely on $\Q_p^\Phi/\Q_p$ being unramified for a comparison between absolute and relative Rapoport--Zink spaces \cite[Proposition 3.30]{Cho}.
\end{proof}
\begin{remark}\label{padicUni(N,0)}
    By \Cref{CanonicalLifting+}\footnote{Strictly speaking, this reference only applies to the case $\q$ is inert on $F,$ but the analogous statements also hold if $\q$ is split in $F$ by Serre--Tate and Grothendieck--Messing.}, we have for $\q\neq\p$ that
    \begin{equation*}
        \mathcal{N}_\q=\mathbb{G}(F^+_\q)/K_{\mathbbm{z},\q}\times\Spf\breve{\Z}_p.
    \end{equation*}
    More canonically, we identify this with
    \begin{equation*}
        \mathcal{N}_\q=\mathrm{Vert}^0(\mathbb{V}\otimes_{F^+}F^+_\q)\times\Spf\breve{\Z}_p.
    \end{equation*}
\end{remark}

We record the definition of the $p$-adic uniformization. For a test scheme $S\to\Spf\breve{\Z}_p$ and $z=(A_0,\iota_0,\lambda_0,\eta_0^p,A,\iota,\lambda,\eta^p)\in\mathcal{M}_\p^{ss}(S),$ we can choose $\O_F$-linear quasi-isogenies
\begin{equation*}
    \tilde{\rho}_0\colon A_0\times_S\overline{S}\dashrightarrow\mathbb{A}_0\times\overline{S},\quad\tilde{\rho}\colon A\times_S\overline{S}\dashrightarrow\mathbb{A}\times\overline{S},
\end{equation*}
that preserve the polarizations (i.e. $c(\tilde{\rho}_0)=c(\tilde{\rho})=1$ as in \Cref{UAbDef}). We define $g^p\bbeta^p_*(K^p)$ to be the orbit of the composition
\begin{equation*}
    \hat{V}^p(\mathbb{A}_0,\mathbb{A})\xrightiso{\bbeta^{p,-1}}V\otimes_\Q\mathbb{A}_f^p\xrightiso{\eta^p}\hat{V}^p(A_0,A)\xrightiso{(\rho_0^{-1},\rho)}\hat{V}^p(\mathbb{A}_0,\mathbb{A}).
\end{equation*}
Finally, for each $\q\mid p$ we have quasi-isogenies
\begin{equation*}
    \rho\colon A[\q^\infty]\times_S\overline{S}\dashrightarrow\mathbb{A}[\q^\infty]\times\overline{S}
\end{equation*}
induced from $\tilde{\rho}.$ Then the first component of $\Theta_{\mathbbm{z}}(z)$ is
\begin{equation*}
    \mathbb{G}(F^+)\left((A[\q^\infty],\iota,\lambda,\rho)_{\q\mid p},g^p\bbeta^p_*K^p\right).
\end{equation*}

We also record how the $p$-adic uniformization interacts with the basic correspondences.
\begin{proposition}\label{padic-basic}
    Given a $p$-adic uniformization datum $\mathbbm{z}$ and $?\in\{\circ,\bullet\}$ we have a commutative diagram (see \Cref{CanonicalBasicUniformization} for the bottommost map)
    \begin{equation*}
        \begin{tikzcd}
            \mathcal{M}_\p^{\wedge,ss}(V,K^p)\arrow[r,"\Theta_{\mathbbm{z}}"]&\mathbb{G}(F^+)\backslash\left(\mathcal{N}\times \mathbb{G}(\A_{F^+}^{p,\infty})/\bbeta^p_*(K^p)\right)\times\mathcal{T}_{\p,\breve{\Z}_p}\\
            \overline{B}_\p^?(V,K^p)\arrow[u,"i^?"]\arrow[r,"\Theta_{\mathbbm{z}}"]\arrow[d,"\pi^?"']&\mathbb{G}(F^+)\backslash\left(\displaystyle{\bigsqcup_{\Lambda^?\in\mathrm{Vert}^?(\mathbb{V}\otimes_\Q\Q_p)}\V(\Lambda^?)}\times \mathbb{G}(\A_{F^+}^{p,\infty})/\bbeta^p_*(K^p)\right)\times\mathcal{T}_{\p,\breve{\Z}_p}\arrow[u]\arrow[d]\\
            \overline{S}_\p^?(V,K^p)\arrow[r,"\Theta_{(\mathbb{V},\bbeta^p)}"]&\mathbb{G}(F^+)\backslash\left(\mathrm{Vert}^?(\mathbb{V}\otimes_\Q\Q_p)\times \mathbb{G}(\A_{F^+}^{p,\infty})/\bbeta^p_*(K^p)\right)\times\mathcal{T}_{\p,\breve{\Z}_p}
        \end{tikzcd}
    \end{equation*}
\end{proposition}
\begin{proof}
    This is a straightforward computation of the $p$-adic uniformization from the definitions of $\overline{B}_\p^?$ and $\overline{S}_\p^?$ in \cite[Definitions 5.3.1, 5.3.2, 5.4.1, 5.4.2]{LTXZZ}.
\end{proof}

\subsection{Kudla--Rapoport special cycles}\label{Shimura-Kudla}
Let $n\ge0.$ Our goal for this section is to construct a supply of (derived) special cycles on $\mathcal{M}_\p(V,K^p)$ for which we can explicitly describe their image under $\Theta_{z^\eta}$ and $\Theta_{\mathbbm{z}}.$ These cycles will also involve a choice $\underline{x}\in V^n.$ Given this, we introduce the following notations.

\begin{notation}
    We denote $M=\mathrm{Res}_{\O_{F^+}}^{\O_F}\mathrm{GL}_{n,\O_F}$ as an algebraic group over $\O_{F^+}.$ For $T\in\mathrm{Herm}_{n\times n}(F),$ we denote $U_T\subseteq M_{F^+}$ the unitary group associated to the matrix $T,$ namely $U_T=\{h\in M\colon h^{\cplx,\intercal}\cdot T\cdot h=T\}.$
\end{notation}
Recall that we denoted $G=U(V).$

\begin{definition}
For a hermitian space $W$ over $F/F^+$ and $n\ge1,$ we denote by $\mathfrak{X}_n(W)$ the groupoid of $\underline{x}\in W^n$ with $T(\underline{x})$ totally positive definite, where morphisms are
\begin{equation*}
    \Hom_{\mathfrak{X}_n(W)}(\underline{x},\underline{y})=\{(g,m)\in U(W)(F^+)\times M(F^+)\colon g\underline{x}=\underline{y}\cdot m\}.
\end{equation*}
Note that the composition of morphisms is given by $(g_1,m_1)\circ(g_2,m_2)=(g_1g_2,m_1m_2).$
\end{definition}
The endomorphisms on this groupoid are given by the graph of the following map $\varphi_{\underline{x}}.$
\begin{definition}\label{H1-H2-notation}
Let $W$ be a $F/F^+$-Hermitian space and $\underline{x}\in W^n$ be linearly independent vectors. If ${}_WG$ denotes the unitary group of $W,$ we denote $H_1^{\underline{x}}\subseteq {}_WG$ the subgroup that fixes $\underline{x}$ pointwise. We denote $H^{\underline{x}}\subseteq {}_WG$ the subgroup that fixes the $F$-span of $\underline{x}.$ Note that $H^{\underline{x}}=H_1^{\underline{x}}\times H_2^{\underline{x}}$ where $H_2^{\underline{x}}\subseteq {}_WG$ is the unitary group of the $F$-span of $\underline{x}.$ We denote $\varphi_{\underline{x}}\colon H^{\underline{x}}\twoheadrightarrow H_2^{\underline{x}}\rightiso U_{T(\underline{x})}$ where the isomorphism characterized by
\begin{equation*}
    g\underline{x}=\underline{x}\cdot\varphi_{\underline{x}}(g),\quad\text{for }g\in {}_{\underline{x}}H.
\end{equation*}
\end{definition}
\begin{remark}
    For consistency of notation, we will use the above notation with different fonts as well. Concretely: in \Cref{SubSubSection-padiccycles}, the ambient group will be denoted $\mathbb{G}$ and thus the above subgroups will be denoted $\mathbb{H}^{\underline{\mathbbm{x}}},\mathbb{H}_1^{\underline{\mathbbm{x}}},\mathbb{H}_2^{\underline{\mathbbm{x}}}\subseteq\mathbb{G}.$
\end{remark}

\newcommand{\CYC}[2]{\mathcal{C}\mathpzc{yc}^\eta_{#1,#2}}
\newcommand{\Cyc}[2]{\mathcal{C}\mathpzc{yc}_{#1,#2}}
\newcommand{\cyc}[2]{\mathbb{C}\mathbbm{yc}_{#1,#2}}
\newcommand{\CYCFJ}[2]{\mathcal{C}\mathpzc{yc}^{\mathrm{FJ},\eta}_{#1,#2}}
\newcommand{\CycFJ}[2]{\mathcal{C}\mathpzc{yc}^{\mathrm{FJ}}_{#1,#2}}
\newcommand{\cycFJ}[2]{\mathbb{C}\mathbbm{yc}^{\mathrm{FJ}}_{#1,#2}}

\subsubsection{On the generic fiber}
\begin{definition}
For $K\in\mathfrak{K}(V),$ we consider the functor $\CYC{\blank}{K}\colon\mathfrak{X}_n(V)\to\mathrm{Ab}$ given by
\begin{equation*}
    \CYC{\underline{x}}{K}\defeq\Z\left[H^{\underline{x}}(F^+)H_1^{\underline{x}}(\A_{F^+}^\infty)\backslash G(\A_{F^+}^{\infty})/K\right].
\end{equation*}
Note that the morphisms of $\mathfrak{X}_n(V)$ act by left multiplication (with $M(F^+)$ acting trivially).
\end{definition}

\begin{definition}[Kudla--Rapoport cycles]\label{KRgenfiber}
    There is a natural transformation
    \begin{equation*}
        Z\colon\CYC{\blank}{K}\to \mathrm{CH}^n(\Sh(V,K))
    \end{equation*}
    given as follows. For $\underline{x}\in\mathfrak{X}_n(V),$ we consider the standard indefinite space $V_{\underline{x}}^\perp$ of rank $N-n$ consisting of elements perpendicular to $\underline{x}.$ Then we consider
    \begin{equation*}
        \Sh(V_{\underline{x}}^\perp,H_1(\A_f)^{\underline{x}}\cap gKg^{-1})\to\Sh(V,gKg^{-1})\xrightiso{g}\Sh(V,K)
    \end{equation*}
    and define $Z(g)$ to be this image of the fundamental class of the source.
\end{definition}

\subsubsection{Derived semi-global integral models}
\begin{definition}
For $K^p\in\mathfrak{K}(V)^p,$ we consider the functors $\Cyc{\blank}{K^p}\colon\mathfrak{X}_n(V)\to\mathrm{Ab}$ and $\Cyc{\blank}{K^p}'\colon\mathfrak{X}_n(V)\to\mathrm{Ab}$ given by
\begin{equation*}
\begin{split}
    \Cyc{\underline{x}}{K^p}'&\defeq\Z\left[H^{\underline{x}}(F^+)\backslash\left(M(F^+)\times H_1^{\underline{x}}(\A_{F^+}^{p,\infty})\backslash G(\A_{F^+}^{p,\infty})/K^p\right)\right]\\
    \Cyc{\underline{x}}{K^p}&\defeq\Z\left[H^{\underline{x}}(F^+)\backslash\left(M(F^+\otimes\Q_p)/M(\O_{F^+}\otimes\Z_p)\times H_1^{\underline{x}}(\A_{F^+}^{p,\infty})\backslash G(\A_{F^+}^{p,\infty})/K^p\right)\right]
\end{split}
\end{equation*}
where the action of $H^{\underline{x}}(F^+)$ on the first factor is through $H_2^{\underline{x}}(F^+)$ via $\varphi_{\underline{x}}.$ We have a natural transformation $\Cyc{\blank}{K^p}'\to\Cyc{\blank}{K^p}$ given by pushforward.
\end{definition}
On this subsection, we will define a natural transformation
\begin{equation*}
    {}^\L\ZZ\colon\Cyc{\blank}{K^p}'\to \mathrm{Gr}^nK_0(\mathcal{M}_\p(V,K^p))
\end{equation*}
and will analyze the generic fiber of such cycles in terms of $Z.$
\begin{remark}
    We conjecture that such functor ${}^\L\ZZ$ should descends to $\Cyc{\underline{x}}{K^p},$ although we will not attempt to prove this. This won't be necessary for our purposes, since i) the map $\Cyc{\underline{x}}{K^p}'\to\Cyc{\underline{x}}{K^p}$ is surjective and ii) as we will see, both the complex uniformization and the $p$-adic uniformization of ${}^\L\ZZ$ descends to $\Cyc{\underline{x}}{K^p}.$
\end{remark}

\begin{definition}
Let $K^p\in\mathfrak{K}(V)^p$ and $\underline{x}\in\mathfrak{X}_n(V).$ For $g^p\in G(\A_f^p)/K^p$ and $m_p\in M(\Q_p)/M(\Z_p),$ we consider the special cycle $\ZZ(g^p,m_p,\underline{x})\to\mathcal{M}_\p(V,K^p)$ parametrizing: $(A_0,\iota_0,\lambda_0,\eta^p_0,A,\iota,\lambda,\eta^p)\in\mathcal{M}_\p(V,K^p)$ and $\underline{y}\in\Hom_{\O_F}(A_0^n,A)\otimes\Q$ such that i) $(\eta^p)^{-1}(\underline{y}_*)$ is in the $K^p$-orbit of $(g^p)^{-1}\underline{x}^p,$ ii) $\underline{y}\circ m_p\in\Hom_{\O_F}(A_0^n,A)\otimes\Z_{p}$ and iii) $T(\underline{y})=T(\underline{x}).$
\end{definition}
\begin{remark}\label{RemarkZZDef}
    These cycles are closely related to the cycles in \cite{Kudla-Rapoport}, as, by the next proposition, the variable $m_p$ can be absorbed into $\underline{x}.$ However, we will like to keep track of the variable $m_p,$ as later we will fix $\underline{x}$ and vary $m_p.$
\end{remark}
It's easy to check that these cycles satisfy the invariance properties that we desire of ${}^\L\ZZ.$
\begin{proposition}
We have the following:
\begin{enumerate}
    \item For fixed $m_p$ and $\underline{x},$ the cycle $\ZZ(g^p,m_p,\underline{x})$ only depends on the double coset of $g^p$ in $H_1^{\underline{x}}(\A_{F^+}^{p,\infty})\backslash G(\A_{F^+}^{p,\infty})/K^p.$
    \item If $n\in M(F^+),$ then we have an isomorphism $\ZZ(g^p,n_pm_p,\underline{x})\iso\ZZ(g^p,m_p,\underline{x}\cdot n)$ induced by $\underline{y}\mapsto\underline{y}\circ n.$ Here, $n_p\in M(F^+\otimes\Q_p)$ denotes the image of $n$ under $M(F^+)\to M(F^+\otimes\Q_p).$
    \item If $h\in G(F^+),$ then $\ZZ(g^p,m_p,\underline{x})=\ZZ(h^pg^p,m_p,h\underline{x}).$
\end{enumerate}
\end{proposition}
\begin{proof}
This follows immediately from the moduli description, noting for (3) that $T(h\underline{x})=T(\underline{x}).$
\end{proof}

Now we consider derived versions of these cycles. Note that if $m_p=\mathrm{diag}(\lambda_1,\ldots,\lambda_n),$ then $\ZZ(g^p,m_p,\underline{x})$ is an open and closed subscheme of $\ZZ(g^p,\lambda_1,x_1)\times_{\mathcal{M}_\p(V,K^p)}\cdots\times_{\mathcal{M}_\p(V,K^p)}\ZZ(g^p,\lambda_n,x_n).$

\begin{definition}
    Let $K^p\in\mathfrak{K}(V)^p$ and $\underline{x}\in\mathfrak{X}_n(V).$ For $g^p\in G(\A_{F^+}^{p,\infty})/K^p$ and $m\in M(F^+),$ we consider
    \begin{equation*}
        {}^\L\ZZ(g^p,m,\underline{x})\in \mathrm{Gr}^nK_0^{\ZZ(g^p,m_p,\underline{x})}(\mathcal{M}_\p(V,K^p))
    \end{equation*}
    to be the component of
    \begin{equation*}
        \O_{\ZZ(g^p,1,(\underline{x}\cdot m)_1)}\otimes^\L_{\mathcal{M}_\p(V,K^p)}\cdots\otimes^\L_{\mathcal{M}_\p(V,K^p)}\O_{\ZZ(g^p,1,(\underline{x}\cdot m)_n)}
    \end{equation*}
    supported on $\ZZ(g^p,1,\underline{x}\cdot m)\simeq\ZZ(g^p,m_p,\underline{x}).$
\end{definition}

The above proposition thus implies that this extends to a natural transformation
\begin{equation*}
    {}^\L\ZZ\colon\Cyc{\blank}{K^p}'\to \mathrm{Gr}^nK_0(\mathcal{M}_\p(V,K^p)).
\end{equation*}

\begin{definition}
    A complex uniformization datum $z^\eta=(\Lambda_{z^\eta,\q})_{\q\mid p}$ induces a natural transformation
    \begin{equation*}
        z^\eta\colon\Cyc{\blank}{K^p}\to\CYC{\blank}{K^pK_{z^\eta,p}},
    \end{equation*}
    namely the correspondence induced by the $H^{\underline{x}}(F^+)$-invariant subset of
    \begin{equation*}
        \left(M(F^+\otimes\Q_p)/M(\O_{F^+}\otimes\Z_p)\right)\times \left(H_1^{\underline{x}}(F^+\otimes\Q_p)\backslash G(F^+\otimes\Q_p)/K_{z^\eta,p}\right).
    \end{equation*}
    given by
    \begin{equation*}
        \left\{(m_p,g_p)\colon \underline{x}_p\cdot m_p\in(g_p\Lambda_{z^\eta,p})^n\right\}=\left\{((m_\q)_{\q\mid p},(g_\q)_{\q\mid p})\colon \underline{x}_\q\cdot m_\q\in(g_\q\Lambda_{z^\eta,\q})^n\text{ for all }\q\mid p\right\}.
    \end{equation*}
\end{definition}

\begin{proposition}\label{KRGenericFiber}
Given a complex uniformization datum $z^\eta,$ we have a commutative diagram of natural transformations
\begin{equation*}
\begin{tikzcd}
    \Cyc{\blank}{K^p}'\arrow{r}{{}^\L\ZZ}\arrow{d}{z^\eta}&\mathrm{Gr}^nK_0(\mathcal{M}_\p(V,K^p))\arrow{r}&\mathrm{Gr}^nK_0(\mathcal{M}_\p^\eta(V,K^p))\arrow{d}{\Theta_{z^\eta}}[swap]{\sim}\\
    \CYC{\blank}{K^pK_{z^\eta,p}}\arrow{r}{Z}&\mathrm{CH}^n(\Sh(V,K^pK_{z^\eta,p}))\arrow{r}{\O_{(\blank\times\mathcal{T}_\p^\eta)}}&\mathrm{Gr}^nK_0(\Sh(V,K^pK_{z^\eta,p})\times\mathcal{T}_\p^\eta)
\end{tikzcd}
\end{equation*}

\end{proposition}
\begin{proof}
This follows from simply running through the complex uniformization \Cref{ComplexUniformization}. Given $z=(A_0,\iota_0,\lambda_0,\eta_0^p,A,\iota,\lambda,\eta^p,\underline{y})\in\ZZ(g^p,m,\underline{x})(\C),$ we may always choose $\eta_{nat}\colon V_z\rightiso V$ in such a way that $\eta_{\mathrm{rat}}(\underline{y})=\underline{x}.$ Under all such choices of $\eta_{\mathrm{rat}},$ the points of $\ZZ(g^p,m,\underline{x})(\C)$ correspond to
\begin{equation*}
    H_1^{\underline{x}}(F^+)\backslash\left(V_{\underline{x}}^\perp(\C)_-/\C^\times \times \{g_p\in G(F^+\otimes\Q_p)\colon \underline{x}_p\cdot m_p\in(g_p\Lambda_{z^\eta,p})^n\}\times H_1^{\underline{x}}(\A_{F^+}^{p,\infty})g^pK^p\right).
\end{equation*}
In particular, the cycles $\ZZ(g^p,1,(\underline{x}\cdot m)_i)$ for $1\le i\le n$ intersect properly on the generic fiber, and thus the claim follows from \cite[Lemma B.2(i)]{Zhang-AFL}.
\end{proof}

\subsubsection{\texorpdfstring{$p$}{p}-adic uniformization}\label{SubSubSection-padiccycles}
Fix a $p$-adic uniformization datum $\mathbbm{z}$ with associated nearby space $\mathbb{V}.$ Recall that we denote $\mathbb{G}=U(\mathbb{V}),$ and that $\mathcal{N}$ denotes the Rapoport--Zink space of \Cref{padicUni}.
\begin{definition}
For $\mathbb{K}^p\in\mathfrak{K}(\mathbb{V})^p,$ we consider the functor $\cyc{\blank}{\mathbb{K}^p}\colon\mathfrak{X}_n(\mathbb{V})\to\mathrm{Ab}$ given by
\begin{equation*}
    \cyc{\underline{\mathbbm{x}}}{\mathbb{K}^p}\defeq\Z\left[\mathbb{H}^{\underline{\mathbbm{x}}}(F^+)\backslash\left(M(F^+\otimes\Q_p)/M(\O_{F^+}\otimes\Z_p)\times \mathbb{H}_1^{\underline{\mathbbm{x}}}(\A_{F^+}^{p,\infty})\backslash \mathbb{G}(\A_{F^+}^{p,\infty})/\mathbb{K}^p\right)\right].
\end{equation*}
where the action of $\mathbb{H}^{\underline{\mathbbm{x}}}(F^+)$ factors through $\mathbb{H}_2^{\underline{\mathbbm{x}}}(F^+)$ and is given by $\varphi_{\underline{\mathbbm{x}}}$ on the first factor.
\end{definition}
This is such that we have the Kudla--Rapoport cycles
\begin{equation*}
    {}^\L\ZZ_{+},{}^\L\ZZ_{-}\colon\cyc{\underline{\mathbbm{x}}}{\mathbb{K}^p}\to \mathrm{Gr}^nK_0\left(\mathbb{G}(F^+)\backslash\left(\mathcal{N}\times\mathbb{G}(\A_{F^+}^{p,\infty})/\mathbb{K}^p\right)\right).
\end{equation*}
where the image of $(m,g^p)$ is induced by the image of (see \Cref{KRDef})
\begin{equation*}
    {}^\L\ZZ_\pm(\underline{\mathbbm{x}}_p\cdot m_p)\times \mathbb{H}_1(\A_{F^+}^{p,\infty})g^p\mathbb{K}^p\subseteq\mathcal{N}\times\mathbb{G}(\A_{F^+}^{p,\infty}).
\end{equation*}
Unlike \Cref{RemarkZZDef}, this well is defined as it only depends on $mM(\O_{F^+}\times\Z_p)$ by \cite[Corollary 2.8.2]{Li-Zhang}.

\begin{remark}
Note that if $T(\underline{x})=T(\underline{\mathbbm{x}}),$ then we have an identification
\begin{equation*}
    \Cyc{\underline{x}}{K^p}\rightiso\cyc{\underline{\mathbbm{x}}}{\bbeta^p_*(K^p)}
\end{equation*}
given in cosets representatives by
\begin{equation*}
    (m_p,g^p)\mapsto (m_p,h^p\bbeta^p_*(g^p))
\end{equation*}
where $h^p$ is such that $\bbeta^p(\underline{x})=h^{p,-1}\underline{\mathbbm{x}}.$
\end{remark}

\begin{proposition}\label{KR-padic}
Fix $\underline{x}\in \mathfrak{X}_n(V)$ and $\underline{\mathbbm{x}}\in\mathfrak{X}_n(\mathbb{V})$ with $T(\underline{x})=T(\underline{\mathbbm{x}}).$ Then we have the commutative diagram
\begin{equation*}
\begin{tikzcd}
    \Cyc{\underline{x}}{K^p}'\arrow{d}\arrow{r}{{}^\L\ZZ}&\mathrm{Gr}^nK_0(\mathcal{M}_\p(V,K^p))\arrow{r}{(\blank)\rvert_{\mathcal{M}_\p^{\wedge,ss}(V,K^p)}}&\mathrm{Gr}^nK_0(\mathcal{M}_\p^{\wedge,ss}(V,K^p))\arrow{d}{\Theta_{\mathbbm{z}}}[swap]{\sim}\\
    \cyc{\underline{\mathbbm{x}}}{\bbeta^p_*(K^p)}\arrow{rr}{{}^\L\ZZ_-\times\mathcal{T}_\p}&&\mathrm{Gr}^nK_0\left(\mathbb{G}(F^+)\backslash\left(\mathcal{N}\times \mathbb{G}(\A_{F^+}^{p,\infty})/\bbeta^p_*(K^p)\right)\times\mathcal{T}_\p\right)
\end{tikzcd}
\end{equation*}
\end{proposition}
\begin{proof}
This follows from simply running through the $p$-adic uniformization \Cref{padicUni}. Given $(A_0,\lambda_0,\eta_0^p,A,\lambda,\eta,\underline{y})\in\ZZ(g^p,m_p,\underline{x})(S)$ with $A$ supersingular, we may choose $\tilde{\rho}_0$ and $\tilde{\rho}$ so that $(\tilde{\rho}_0^{-1},\tilde{\rho})(\underline{y})=\underline{\mathbbm{x}},$ since $\underline{x}$ and $\underline{\mathbbm{x}}$ have the same moment matrix. Under all such choices of $\tilde{\rho}_0$ and $\tilde{\rho},$ our cycle corresponds to
\begin{equation*}
    \ZZ(g^p,m_p,\underline{x})^{\wedge,ss}\iso \mathbb{H}_1^{\underline{\mathbbm{x}}}(F^+)\backslash\left(\ZZ_-(\underline{\mathbbm{x}}_p\cdot m_p)\times \mathbb{H}_1^{\underline{\mathbbm{x}}}(\A_{F^+}^{p,\infty})\cdot g^p\bbeta^p_*(K^p)/\bbeta^p_*(K^p)\right)
\end{equation*}
and the claim follows.
\end{proof}

\subsection{Friedberg--Jacquet special cycles}\label{Shimura-FJ}
We continue to fix a $p$-adic uniformization datum $\mathbbm{z}.$
\begin{definition}
    For $\underline{x}\in\mathfrak{X}_n(V),$ $\underline{\mathbbm{x}}\in\mathfrak{X}_n(\mathbb{V})$ and $K\in\mathfrak{K}(V),$ $K^p\in\mathfrak{K}(V)^p,$ $\mathbb{K}^p\in\mathfrak{K}(\mathbb{V})^p,$ we consider
    \begin{equation*}
    \begin{split}
        \CYCFJ{\underline{x}}{K}&\defeq\Z\left[H^{\underline{x}}(\A_{F^+}^{\infty})\backslash G(\A_{F^+}^{\infty})/K\right],\\
        (\CycFJ{\underline{x}}{K^p})'&\defeq\Z\left[U_{T(\underline{x})}(F^+)\backslash M(F^+)\times H^{\underline{x}}(\A_{F^+}^{p,\infty})\backslash G(\A_{F^+}^{p,\infty})/K^p\right],\\
        \CycFJ{\underline{x}}{K^p}&\defeq\Z\left[U_{T(\underline{x})}(F^+\otimes\Q_p)\backslash M(F^+\otimes\Q_p)/M(\O_{F^+}\otimes\Z_p)\times H^{\underline{x}}(\A_{F^+}^{p,\infty})\backslash G(\A_{F^+}^{p,\infty})/K^p\right],\\
        \cycFJ{\underline{\mathbbm{x}}}{\mathbb{K}^p}&\defeq\Z\left[U_{T(\underline{\mathbbm{x}})}(F^+\otimes\Q_p)\backslash M(F^+\otimes\Q_p)/M(\O_{F^+}\otimes\Z_p)\times \mathbb{H}^{\underline{\mathbbm{x}}}(\A_{F^+}^{p,\infty})\backslash \mathbb{G}(\A_{F^+}^{p,\infty})/\mathbb{K}^p\right].
    \end{split}
    \end{equation*}
    We have pullback maps
    \begin{equation*}
        \CYCFJ{\underline{x}}{K}\to\CYC{\underline{x}}{K},\quad \CycFJ{\underline{x}}{K^p}\to\Cyc{\underline{x}}{K^p},\quad\cycFJ{\underline{\mathbbm{x}}}{\mathbb{K}^p}\to \cyc{\underline{\mathbbm{x}}}{\mathbb{K}^p}
    \end{equation*}
    and pushforward map $(\CycFJ{\underline{x}}{K^p})'\to\CycFJ{\underline{x}}{K^p}.$ We define
    \begin{equation*}
        Z^\FJ\colon\CYCFJ{\underline{x}}{K}\to\CYC{\underline{x}}{K}\xrightarrow{Z}\mathrm{CH}^n(\Sh(V,K)).
    \end{equation*}
\end{definition}
\begin{remark}
    We note that $Z^\FJ$ may also be defined similarly to \Cref{KRgenfiber} via sub-Shimura varieties for the group $H^{\underline{x}}.$ Suppose $K\in\mathfrak{K}(V)$ and $g\in G(\A_f)$ and denote $K_{H^{\underline{x}}}=H^{\underline{x}}(\A_f)\cap gKg^{-1},$ $K_1=H_1^{\underline{x}}(\A_f)\cap gKg^{-1}$ and $K_2=H_2^{\underline{x}}(\A_f)\cap gKg^{-1}.$ Then $K_1\times K_2\subseteq K_{H^{\underline{x}}}$ has finite index. Denote $\Sh_{H^{\underline{x}}}(K_{H^{\underline{x}}})$ the Shimura variety for the group $H.$ Then
    \begin{equation*}
        \begin{tikzcd}
            \Sh(V_1,K_1)\times\Sh(V_2,K_2)\arrow[d]&&\\\
            \Sh_{H^{\underline{x}}}(K_{H^{\underline{x}}})\arrow[r]&\Sh(V,gKg^{-1})\arrow[r,"g","\sim"']&\Sh(V,K)
        \end{tikzcd}
    \end{equation*}
    give us that
    \begin{equation*}
        [K_1\times K_2\colon K_{H^{\underline{x}}}]\cdot Z^\FJ(g)=Z(g)\times \Sh(V_2,K_2)=\sum_{h_2\in H_2(\Q)\backslash H_2(\A_f)/K_2}Z(h_2g)
    \end{equation*}
    Similarly, the fiber of $H^{\underline{x}}(\A_f)gK$ is parametrized by coset representatives $h_2g$ as $h_2$ varies through in
    \begin{equation*}
        H^{\underline{x}}(\Q)H_1^{\underline{x}}(\A_f)\backslash H^{\underline{x}}(\A_f)/gKg^{-1}=H_2^{\underline{x}}(\Q)\backslash H_2^{\underline{x}}(\A_f)/\mathrm{pr}_{H_2^{\underline{x}}}(gKg^{-1}),
    \end{equation*}
    and note that $[K_2\colon\mathrm{pr}_{H_2^{\underline{x}}}(gKg^{-1})]=[K_1\times K_2\colon K_H].$
\end{remark}

\begin{definition}\label{ZZFJDef}
    For $x\in\mathfrak{X}_n(V)$ and $K^p\in\mathfrak{K}(V)^p,$ we consider map a $(\CycFJ{\underline{x}}{K^p})'\to\Cyc{\underline{x}}{K^p}'$ to be any choice such that the following is commutative
    \begin{equation*}
    \begin{tikzcd}
        (\CycFJ{\underline{x}}{K^p})'\arrow{r}\arrow{d}&\Cyc{\underline{x}}{K^p}'\arrow{d}\\
        \CycFJ{\underline{x}}{K^p}\arrow{r}&\Cyc{\underline{x}}{K^p}.
    \end{tikzcd}
    \end{equation*}
    Then we define
    \begin{equation*}
        {}^\L\ZZ^\FJ\colon(\CycFJ{\underline{x}}{K^p})'\to\Cyc{\underline{x}}{K^p}'\xrightarrow{{}^\L\ZZ} \mathrm{Gr}^nK_0(\mathcal{M}_\p(V,K^p)).
    \end{equation*}
\end{definition}
\begin{remark}
    As in \Cref{RemarkZZDef}, ${}^\L\ZZ^{\mathrm{FJ}}$ should descend to $\CycFJ{\underline{x}}{K^p},$ and thus should not depend on the choice of map $(\CycFJ{\underline{x}}{K^p})'\to\Cyc{\underline{x}}{K^p}'$ as above. However, for our purposes, such choice will not matter since the complex uniformization and $p$-adic uniformizations of ${}^\L\ZZ^{\mathrm{FJ}}$ do not depend on such choice.
\end{remark}

\begin{definition}\label{CycFJ-CYCFJDef}
Given a complex uniformization datum $z^\eta,$ consider the natural transformation $z^\eta\colon\CycFJ{\underline{x}}{K^p}\to \CYCFJ{\underline{x}}{K^pK_{z^\eta,p}}$ induced by the map
\begin{equation*}
    \Z[U_{T(\underline{x})}(F^+\otimes\Q_p)\backslash M(F^+\otimes\Q_p)/M(\O_{F^+}\otimes\Z_p)]\to \Z[H^{\underline{x}}(F^+\otimes\Q_p)\backslash G(F^+\otimes\Q_p)/K_{z^\eta,p}]
\end{equation*}
given in coset representatives by
\begin{equation*}
    m\mapsto\left(g_p\mapsto\#\{nM(\O_{F^+}\otimes\Z_p)\in U_{T(\underline{x})}(F^+\otimes\Q_p)mM(\O_{F^+}\otimes\Z_p)\colon \underline{x}\cdot n\subseteq g_p\Lambda_{z^\eta,p}\}\right).
\end{equation*}
In other words, consider the lattice $\Lambda_{\underline{x},p}\defeq\mathrm{span}_{\O_F\otimes\Z_p}(\underline{x}\cdot m)$ in the space $V_{\underline{x},p}\defeq\mathrm{span}(\underline{x})\otimes_\Q\Q_p.$ Then the above right hand side is the number of lattices $\Lambda$ in $V_{\underline{x},p}$ which are isomorphic to $\Lambda_{\underline{x},p}$ and lie inside $g_p\Lambda_{z^\eta,p}\cap V_{\underline{x}}.$
\end{definition}

The following are immediate consequences of the previous definitions.
\begin{proposition}\label{FJcomplex}
For $K^p\in\mathfrak{K}(V)^p,$ we have a commutative diagram of natural transformations
\begin{equation*}
    \begin{tikzcd}
        \CycFJ{\blank}{K^p}\arrow{r}\arrow{d}{z^\eta}&\Cyc{\blank}{K^p}\arrow{d}{z^\eta}\\
        \CYCFJ{\blank}{K^pK_{z^\eta,p}}\arrow{r}&\CYC{\blank}{K^pK_{z^\eta,p}}
    \end{tikzcd}
\end{equation*}
\end{proposition}
\begin{proposition}\label{FJpadic}
Let $K^p\in\mathfrak{K}(V)^p.$ Fix a $p$-adic uniformization datum $\mathbbm{z}$ and choose $\underline{x}\in \mathfrak{X}_n(V)$ and $\underline{\mathbbm{x}}\in\mathfrak{X}_n(V^\circ)$ with $T(\underline{x})=T(\underline{\mathbbm{x}}).$ Then we have the commutative diagram
\begin{equation*}
\begin{tikzcd}
    \CycFJ{\underline{x}}{K^p}\arrow[d,"\sim"]\arrow{r}&\Cyc{\underline{x}}{K^p}\arrow[d,"\sim"]\\
    \cycFJ{\underline{\mathbbm{x}}}{\bbeta^p_*(K^p)}\arrow{r}&\cyc{\underline{\mathbbm{x}}}{\bbeta^p_*(K^p)}
\end{tikzcd}
\end{equation*}
\end{proposition}
\section{Computation on spherical functions}\label{ComputationChapter}
\newcommand{\Rel}{\mathrm{Rel}}
\newcommand{\str}{\mathrm{str}}
\newcommand{\inv}{\mathrm{inv}}
Let $F/F^+$ be an unramified quadratic extension of $p$-adic fields, with uniformizer $\varpi$ and where $F^+$ has residue field of size $q.$

For $\epsilon\in\Z^r,$ we denote
\begin{equation*}
    \widetilde{\inv}(\epsilon)\defeq \sum_{1\le i<j\le r}\max(0,\epsilon_i-\epsilon_j),
\end{equation*}
which we also denote $\inv(\epsilon)$ in the case $\epsilon\in\{0,1\}^r.$ We also denote
\begin{equation*}
    \lambda_i(\epsilon)\defeq\#\{j\in\{1,\ldots,r\}\colon \epsilon_j=i\}.
\end{equation*}

In this section, we will study the local harmonic analysis that arises from the arithmetic level raising. As mentioned in the introduction, this will amount to the fact that a basis obtained from arithmetic intersection numbers ($\phi_e^\flat$ in \cref{phiflatDef}) behaves like the basis of a ``Cartan decomposition''. We will spell this out in \Cref{invSat}, but before that we will instead use the results of \cite{CZ,CZ2}. As explained in \cite[Appendix A]{CZ}, these results are tightly related to the inverse Satake transform \cite{Sakellaridis}. While it should be possible to perform the computations of this section in later language, the functions $\phi_e^\flat$ we are dealing with are naturally presented in terms of certain lattice combinatorics, which are closer to the former language.\footnote{We also need to appeal to \cite{CZ2} as the case of interest (for the spherical variety $X=(U(r)\times U(r))\backslash U(2r)$) is not technically covered by the results of \cite{Sakellaridis} since the group $U(r)$ is not split. However, the same methods should apply. Alternatively, it should also be possible to derive the inverse Satake transform from the results of \cite{Hironaka-Komori}.}

\subsection{Straightening relations}
We first recall some notation of \cite{CZ,CZ2}.
\begin{definition}\label{TypDef}
    For $r\ge1,$ we consider $\Typ_r\defeq\Z^r,$ equipped with the partial order where $e\preccurlyeq f$ if $e_i\le f_i$ for all $1\le i\le r.$ For a coefficient ring $R,$ we denote by $R[\Typ_r],$ $R\llbracket\Typ_r],$ and $R[\Typ_r\rrbracket$ the $R$-module of functions $\Typ_r\to R$ whose support is, respectively, finite, bounded above by some $f\in\Typ_r,$ and bounded below by some $f\in\Typ_r.$ For $e\in\Typ_r,$ we denote $\delta_e\in R[\Typ_r]$ the indicator function at $e.$ For $\epsilon\in\Z^r,$ we consider the endomorphism $t(\epsilon)\colon R[\Typ_r]\to R[\Typ_r]$ to be such that $t(\epsilon)(\delta_e)=\delta_{e+\epsilon}.$ We consider the concatenation product $(\blank\star\blank)\colon\Typ_a\times\Typ_b\to\Typ_{a+b}.$
\end{definition}
\begin{definition}
    For $r\ge1,$ we consider $\Typ_r^0\subseteq\Typ_r$ the subset
    \begin{equation*}
        \Typ_r^0\defeq\{e=(e_1,\ldots,e_r)\in\Z^r\colon e_1\ge\cdots\ge e_r\},
    \end{equation*}
    with induced preorder $\preccurlyeq.$ Similarly to \Cref{TypDef}, for a coefficient ring we consider the $R$-modules $R[\Typ_r^0],$ $R\llbracket\Typ_r^0]$ and $R[\Typ_r^0\rrbracket.$ We consider the perfect inner product pairing
    \begin{equation*}
        \langle\blank,\blank\rangle\colon R[\Typ_r^0\rrbracket\times R\llbracket\Typ_r^0]\to R
    \end{equation*}
    which restricts to a nondegenerate pairing
    \begin{equation*}
        \langle\blank,\blank\rangle\colon R[\Typ_r^0]\times R[\Typ_r^0]\to R.
    \end{equation*}
\end{definition}
\begin{proposition}[{\cite[Proposition 4.11]{CZ}}]
    We consider the submodule $\Rel_r\subseteq R[\Typ_r]$ generated under $\star$ by the elements $\Rel(a)\defeq\delta_{(a,a+1)}-\delta_{(a+1,a)}$ for $a\in\Z$ and, for $b>a,$ the elements
    \begin{equation*}
        \Rel(a,b)\defeq\delta_{(a,b)}-\delta_{(a+1,b)}-(-q)^{b-a-1}(\delta_{(b,a)}-\delta_{(b-1,a+1)}). 
    \end{equation*}
    Then we have a natural quotient map $\str\colon R[\Typ_r]/\Rel_r\rightiso R[\Typ_r^0]$ characterized by $\delta_e\mapsto\delta_e$ for $e\in\Typ_r^0.$
\end{proposition}
\begin{remark}
    Note that $\Rel_r=\bigoplus_{k\in\Z}\Rel_r\cap R[\Typ_r^{\Sigma=k}]$ where $\Typ_r^{\Sigma=k}=\{e=(e_1,\ldots,e_r)\in\Typ_r\colon e_1+\cdots+e_r=k\}.$ In particular, $\str$ also naturally extends to quotient maps $\str\colon R[\Typ_r\rrbracket\twoheadrightarrow R[\Typ_r^0\rrbracket,$ given that any element of of $R[\Typ_r\rrbracket$ has finite support along $\Typ_r^{\Sigma=k}$ for each $k\in\Z.$
\end{remark}
\begin{definition}
    For $r\ge1,$ we consider $\Typ_r^{0,\flat}\subseteq\Typ_r$ the subset
    \begin{equation*}
        \Typ_r^{0,\flat}\defeq\{e=(e_1,\ldots,e_r)\in\Z^r\colon e_1\ge\cdots\ge e_r\ge0\},
    \end{equation*}
    with induced preorder $\preccurlyeq.$ Similarly to \Cref{TypDef}, for a coefficient ring we consider the $R$-modules $R[\Typ_r^{0,\flat}]$ and $R\llbracket\Typ_r^{0,\flat}\rrbracket$ (note that $R\llbracket\Typ_r^{0,\flat}]=R[\Typ_r^{0,\flat}]$ and $R\llbracket\Typ_r^{0,\flat}\rrbracket=R[\Typ_r^{0,\flat}\rrbracket$). We consider the perfect inner product pairing
    \begin{equation*}
        \langle\blank,\blank\rangle\colon R\llbracket\Typ_r^{0,\flat}\rrbracket\times R[\Typ_r^{0,\flat}]\to R
    \end{equation*}
    which restricts to a nondegenerate pairing
    \begin{equation*}
        \langle\blank,\blank\rangle\colon R[\Typ_r^{0,\flat}]\times R[\Typ_r^{0,\flat}]\to R.
    \end{equation*}
\end{definition}
\begin{proposition}[{\cite[Proposition 4.2.7]{CZ2}}]
    We consider the submodule $\Rel_r^\flat\subseteq R[\Typ_r]$ generated under $\star$ by $\Rel_r,$ by $\Rel^\flat(1)\defeq\delta_{(-1)}-\delta_{(1)}$ and, for $m>1,$ by
    \begin{equation*}
        \Rel^\flat(m)\defeq\delta_{(-m)}-\delta_{(-m+2)}-q^{m-1}(\delta_{(m)}-\delta_{(m-2)}). 
    \end{equation*}
    Then we have a natural quotient map $\mathrm{str}^\flat\colon R[\Typ_r]/\Rel_r\rightiso R[\Typ_r^{0,\flat}]$ characterized by $\delta_e\mapsto\delta_e$ for $e\in\Typ_r^0.$
\end{proposition}

We also consider a new set of straightening relations.
\begin{proposition}\label{strphi}
    We consider the submodule $\Rel_r^\phi\subseteq R[\Typ_r]$ generated under $\star$ by $\Rel_r,$ by $\Rel^\phi(1)\defeq\delta_{(-1)}-q^2\delta_{(1)}$ and, for $m>1,$ by
    \begin{equation*}
        \Rel^\phi(m)\defeq\delta_{(-m)}-q^2\delta_{(-m+2)}-q^{m-1}(q^2\delta_{(m)}-\delta_{(m-2)}). 
    \end{equation*}
    Then we have a natural quotient map $\mathrm{str}^\phi\colon R[\Typ_r]/\Rel_r\rightiso R[\Typ_r^{0,\flat}]$ characterized by $\delta_e\mapsto\delta_e$ if $e\in\Typ_r^{0,\flat}.$
\end{proposition}
\begin{proof}
    We first note that $t(-1)+q\cdot t(1)$ preserves $\Rel_r^\phi,$ as $(t(-1)+q\cdot t(1))\Rel^\phi(1)=\Rel^\phi(2),$ $(t(-1)+q\cdot t(1))\Rel^\phi(2)=\Rel^\phi(3)+2q\Rel^\phi(1)$ and for $m>2$ we have $(t(-1)+q\cdot t(1))\Rel^\phi(m)=\Rel^\phi(m+1)+q\Rel^\phi(m-1).$ Given that, the proof of \cite[Proposition 4.2.7]{CZ2} carries through as in the case $(\mathrm{uH})^\flat,$ and we are reduced to check that 
    \begin{equation*}
        \mathrm{str}^\phi\left(\delta_{(-1,-2)}-q^2\delta_{(-2,1)}\right)
    \end{equation*}
    is zero. On the one hand, we have
    \begin{equation*}
        \mathrm{str}^\natural(\delta_{(-2,1)})=q^2\delta_{(1,-2)}+(1-q^2)\delta_{(0,-1)}
    \end{equation*}
    and thus
    \begin{equation*}
        \mathrm{str}^\phi(\delta_{(-2,1)})=\mathrm{str}^\natural\left(q^2(q^3\delta_{(1,2)}+q(q-1)\delta_{(1,0)})+(1-q^2)q^2\delta_{(0,1)}\right)=q^5\delta_{(2,1)}+q^2(1-q)\delta_{(1,0)}.
    \end{equation*}
    On the other hand, we have
    \begin{equation*}
        \mathrm{str}^\phi(\delta_{(-1,-2)})=\mathrm{str}^\flat_{\natural}(\mathrm{q}^3\delta_{(-1,2)}+q(q-1)\delta_{(-1,0)})=\mathrm{str}^\phi(q^3(q^2\delta_{(2,-1)}+(1-q^2)\delta_{(1,0)}))+q^3(q-1)\delta_{(1,0)}
    \end{equation*}
    which is $q^7\delta_{(2,1)}+q^4(1-q).$
\end{proof}

\subsection{Spherical functions for \texorpdfstring{$U(W)\backslash\mathrm{GL}(W)$}{U(W)\textbackslash{}GL(W)}}
Let $W$ be a nondegenerate $F/F^+$ hermitian space of rank $r$ for some $r\ge1,$ with unitary group $U(W)\subseteq\mathrm{GL}(W).$

For a choice of hyperspecial subgroup $K\subseteq\mathrm{GL}(W),$ we consider the Hecke algebra $\mathbb{T}_r\defeq\Z[K\backslash\mathrm{GL}(W)/K]$ with elements $T_{i,r}=\mathbbm{1}[K\varpi^{\mu_i}K]$ where $\mu_i=(\underbrace{1,\ldots,1}_{i},\underbrace{0,\ldots,0}_{r-i})$ for $0\le i\le r.$
\begin{definition}
We denote $\Typ(W)$ the set of isomorphism classes of (full rank) lattices $\Lambda\subseteq W$ up to the action of $U(W).$ Note that we have an identification $\Typ(W)\iso U(W)\backslash\mathrm{GL}(W)/K.$
\end{definition}
\begin{proposition}[{\cite[Proposition 4.2]{CZ}}]\label{Typ}
    We have an injection
    \begin{equation*}
        \typ\colon\Typ(W)\hookrightarrow\Typ_r^0
    \end{equation*}
    given by $\typ(\Lambda)=e=(e_1\ge\cdots\ge e_r)$ where $e$ is the relative position of $\Lambda^\vee$ and $\Lambda.$ In other words, for every $m$ sufficiently large we have
    \begin{equation*}
        \Lambda^\vee/\varpi^m\Lambda\iso\bigoplus_{i=1}^r\O_{F^+}/\varpi^{m+e_i}\O_{F^+}
    \end{equation*}
    as $\O_{F^+}$-modules. Moreover, if we denote $\delta=0$ resp. $\delta=1$ if $W$ is split resp. nonsplit, then the image of $\typ$ is precisely the set of $e=(e_1\ge\cdots\ge e_r)$ such that $\sum_{i=1}^re_i\equiv\delta\mod2.$
\end{proposition}
\begin{theorem}[{\cite[Theorem 4.16]{CZ}}]\label{Heckenatural}
    For $0\le i\le r,$ we consider the operators
    \begin{equation*}
        \Delta_{i,r}\defeq\sum_{\substack{\epsilon\in\{0,1\}\\\lambda_1(\epsilon)=i}}q^{\inv(\epsilon)}t(2\epsilon).
    \end{equation*}
    We have that $\Delta_{i,r}$ preserves $\Rel_r.$ Under the injection $\typ\colon\Z[U(W)\backslash\mathrm{GL}(W)/K]\hookrightarrow\Typ_r^0,$ we have that the adjoint of $T_{i,r}$ is $\Delta_{i,r},$ that is, we have
    \begin{equation*}
        \langle T_{i,r}f,g\rangle=\langle f,\mathrm{str}^\natural(\Delta_{i,r}g)\rangle
    \end{equation*}
    for all $f,g\in\Z[\Typ_r^0].$
\end{theorem}

\begin{definition}
    Given $e\in\Typ_r^0,$ we consider the element $\phi_e^\natural\in \Z\llbracket\Typ_r^0]$ given by
    \begin{equation*}
        \phi_e^\natural(f)=\sum_{\substack{\Lambda_e\subseteq \Lambda_f\\\typ(\Lambda_e)=e}}c(\dim(\Lambda_f/(\varpi \Lambda_f+\Lambda_e)))
    \end{equation*}
    where $\Lambda_f\subseteq W$ is any fixed lattice with $\typ(\Lambda_f)=f$ and where $c(k)\defeq\prod_{i=1}^k(1-q^{2i}).$
\end{definition}
\begin{proposition}\label{DeltaPhi}
    Denote $\Delta_{r,\phi}\colon\Z[\Typ_r\rrbracket\to\Z[\Typ_r\rrbracket$ the operator induced by
    \begin{equation*}
        \bigstar_{i=1}^r\left(\frac{t(0)-q^{2(r-i+1)}t(2)}{t(0)-q^{2(r-i)}t(2)}\right)=\frac{t(0)-q^{2r}t(2)}{t(0)-q^{2(r-1)}t(2)}\star\frac{t(0)-q^{2(r-1)}t(2)}{t(0)-q^{2(r-2)}t(2)}\star\cdots\star\frac{t(0)-q^2t(2)}{t(0)-t(2)}
    \end{equation*}
    or, in other words, induced by
    \begin{equation*}
        \sum_{\epsilon\in\Z_{\ge0}^r}(1-q^2)^{\lambda_{\neq0}(\epsilon)}q^{\sum_i2(r-i)\epsilon_i}t(2\epsilon).
    \end{equation*}
    Then $\Delta_{r,\phi}$ preserves $\mathrm{Rel}_r^\natural,$ and for all $e,f\in\Typ_r^0$ we have
    \begin{equation*}
        \phi^\natural_e(f)=\langle\mathrm{str}^\natural(\Delta_{r,\phi}\delta_f),\delta_e\rangle.
    \end{equation*}
\end{proposition}
\begin{proof}
For $e\in\Typ_r^0,$ consider $d_e\in\Z\llbracket\Typ_r^0]$ given by
\begin{equation*}
    d_e(f)=\#\{\Lambda_e\subseteq\Lambda\colon\typ(\Lambda_e)=e\}
\end{equation*}
where $\Lambda_f\subseteq W$ is any fixed lattice with $\typ(\Lambda_f)=f.$

We claim that we have
\begin{equation*}
    \phi_e^\natural=\sum_{i=0}^r T_{i,r}q^{i^2}(-q)^{i}d_e\quad\text{and}\quad
    \delta_e=\sum_{i=0}^r T_{i,r}q^{i^2}(-q)^{-i}d_e.
\end{equation*}
For this, we compute $\sum_{i=0}^rT_{i,r}q^{i^2}x^id_e.$ Given $\Lambda_f$ with $\typ(\Lambda_f)=f,$ we have
\begin{equation*}
    (\sum_{i=0}^rT_{i,r}q^{i^2}x^id_e)(f)=\sum_{\substack{\varpi\Lambda_f\subseteq \Lambda\subseteq\Lambda_f}}q^{\dim(\Lambda_f/\Lambda)^2}x^{\dim(\Lambda_f/\Lambda)}\sum_{\substack{\Lambda_e\subseteq\Lambda\\\typ(\Lambda_e)=e}}1.
\end{equation*}
Changing the order of summation, this is
\begin{equation*}
\begin{split}
    &\sum_{\substack{\Lambda_e\subseteq\Lambda_f\\\typ(\Lambda_e)=e}}\sum_{\varpi\Lambda_f+\Lambda_e\subseteq\Lambda\subseteq\Lambda_f}q^{\dim(\Lambda_f/\Lambda)^2}x^{\dim(\Lambda_f/\Lambda)}\\
    &\qquad=\sum_{\substack{\Lambda_e\subseteq\Lambda_f\\\typ(\Lambda_e)=e}}\sum_{i=0}^{\dim(\Lambda_f/(\varpi\Lambda_f+\Lambda_e))}q^{i^2}x^{i}\qbinom{\dim(\Lambda_f/(\varpi\Lambda_f+\Lambda_e))}{i}{q^2}.
\end{split}
\end{equation*}
Now both claims follows from the $q$-binomial theorem
\begin{equation*}
    \sum_{i=0}^nq^{i^2}x^{i}\qbinom{n}{i}{q^2}=\prod_{i=1}^n(1+q^{2i-1}x).
\end{equation*}

This allow us to write $\phi_e^\natural=T_{r,\phi}\delta_e$ where
\begin{equation*}
    T_{r,\phi}\defeq\left(\sum_{i=0}^rT_{i,r}q^{i^2}(-q)^i\right)\left(1-(\sum_{i=1}^rT_{i,r}q^{i^2}(-q)^{-i})+(\sum_{i=1}^rT_{i,r}q^{i^2}(-q)^{-i})^2-\cdots\right).
\end{equation*}
Hence if $\Delta_{r,\phi}$ denotes the adjoint of $T_{r,\phi},$ by \Cref{Heckenatural}, we have
\begin{equation*}
    \Delta_{r,\phi}=\left(\sum_{i=0}^r\Delta_{i,r}q^{i^2}(-q)^i\right)\left(1-(\sum_{i=1}^r\Delta_{i,r}q^{i^2}(-q)^{-i})+(\sum_{i=1}^r\Delta_{i,r}q^{i^2}(-q)^{-i})^2-\cdots\right).
\end{equation*}

Now if we consider
\begin{equation*}
    \Delta_r(x)\defeq\sum_{k=0}^rq^{i^2}\Delta_{i,r}\cdot x^i
\end{equation*}
then we have
\begin{equation*}
    \Delta_r(x)=\sum_{\epsilon\in\{0,1\}^r}q^{\lambda_{1}(\epsilon)^2+2\mathrm{inv}(\epsilon)}x^{\lambda_{1}(\epsilon)}t(2\epsilon).
\end{equation*}
Noting that
\begin{equation*}
    \lambda_{1}(\epsilon)^2+2\mathrm{inv}(\epsilon)=\sum_{i=1}^r(2(r-i)+1)\epsilon_i.
\end{equation*}
we arrive at
\begin{equation*}
    \Delta_r(x)=(t(0)+q^{2r-1}x\cdot t(2))\star(t(0)+q^{2r-3}x\cdot t(2))\star\cdots\star(t(0)+qx\cdot t(2)).
\end{equation*}
Hence
\begin{equation*}
    \Delta_{r,\phi}=\frac{t(0)-q^{2r}t(2)}{t(0)-q^{2(r-1)}t(2)}\star\frac{t(0)-q^{2(r-1)}t(2)}{t(0)-q^{2(r-2)}t(2)}\star\cdots\star\frac{t(0)-q^2t(2)}{t(0)-t(2)},
\end{equation*}
which we can expand as
\begin{equation*}
\begin{split}
    \Delta_{r,\phi}\ &=\bigstar_{i=1}^{r}\left(t(0)-q^{2(r-i)+2}t(2)\right)\left(t(0)+q^{2(r-i)}t(2)+(q^{2(r-i)}t(2))^2+\cdots\right)\\
    &=\bigstar_{i=1}^r\left(t(0)+q^{2(r-i)}(1-q^2)t(2)+q^{4(r-i)}(1-q^2)t(4)+\cdots\right)\\
    &=\sum_{\epsilon\in\Z_{\ge0}^r}(1-q^2)^{\lambda_{\neq0}(\epsilon)}q^{\sum_i2(r-i)\epsilon_i}t(2\epsilon).\qedhere
\end{split}
\end{equation*}
\end{proof}

\subsection{Spherical functions for \texorpdfstring{$(U(W_1)\times U(W_2))\backslash U(W_1\oplus W_2)$}{(U(W1)xU(W2))\textbackslash{}U(W1+W2)}}
Let $W_1,W_2$ be two nondegenerate $F/F^+$ hermitian space of rank $r$ for some $r\ge1,$ such that $W_1\oplus W_2$ is a split hermitian space of rank $2r.$ Denote $G=U(W_1\oplus W_2)$ and $H=U(W_1)\times U(W_2).$

For a choice of a hyperspecial subgroup $K$ of $U(W_1\oplus W_2),$ we consider the Hecke algebra $\mathbb{T}_r^\flat\defeq\Z[K\backslash G/K],$ with Satake homomorphism $\mathrm{Sat}\colon\mathbb{T}_r^\flat\to\Z[\bm{\mu}_1,\ldots,\bm{\mu}_r]^{\mathrm{sym}}$ as in \cite[Definition 3.2.4]{CZ2}.
\begin{proposition}[{\cite[Proposition 3.1.6]{CZ2}}]\label{TypFJ}
    Let $\Lambda$ be the lattice for which $K=\mathrm{Stab}(\Lambda),$ which satisfies $\Lambda^\vee=\varpi^i\Lambda$ for some $i.$ Then we have an injection
    \begin{equation*}
        \typ_K\colon H\backslash G/K\hookrightarrow\Typ_r^{0,\flat}
    \end{equation*}
    given by $\typ_K(HgK)=e=(e_1\ge\cdots\ge e_r)$ where
    \begin{equation*}
        \typ(g\Lambda\cap W_2)=(e_1-i\ge\cdots\ge e_r-i).
    \end{equation*}
    Moreover, if we denote $\delta=0$ resp. $\delta=1$ if $W_2$ is split resp. nonsplit, then the image of $\typ_K$ is precisely the set of $e=(e_1\ge\cdots\ge e_r\ge 0)$ such that $\sum_{i=1}^re_i\equiv\delta\mod 2.$
\end{proposition}

\begin{theorem}[{\cite[Theorem 4.16]{CZ}}]\label{Heckeflat}
    For $0\le i\le r,$ we consider the operators
    \begin{equation*}
        \Delta_{i,r}^\flat\defeq\sum_{\substack{\epsilon\in\{-1,0,1\}\\\lambda_0(\epsilon)=r-i}}q^{2\widetilde{\inv}(\epsilon)}q^{\lambda_1(\epsilon)^2+(r-\lambda_{-1}(\epsilon)^2)-(r-i)^2}t(2\epsilon).
    \end{equation*}
    We have that $\Delta_{i,r}^\flat$ preserves $\Rel_r^\flat.$ Under the injection $\typ_K\colon\Z[H\backslash G/K]\hookrightarrow\Typ_r^0,$ we denote $S_{i,r}^\flat$ the adjoint of $\Delta_{i,r}^\flat,$ that is, the operator such that
    \begin{equation*}
        \langle S_{i,r}^\flat f,g\rangle=\langle f,\mathrm{str}^\natural(\Delta^\flat_{i,r}g)\rangle
    \end{equation*}
    for all $f,g\in\Z[\Typ_r^0].$ Then $S_{i,r}^\flat\in\mathbb{T}_r^\flat,$ and $\mathrm{Sat}(S_{i,r}^\flat)=q^{r^2-(r-i)^2}s_i(\bm{\mu}_1,\ldots,\bm{\mu}_i).$
\end{theorem}
\begin{theorem}[{\cite[Theorem 5.2.6]{CZ}}]\label{Heckeflathalf}
    For $0\le i\le r,$ we consider the operators
    \begin{equation*}
        \Delta_{i,r}^{1/2,\flat}\defeq\sum_{\substack{\epsilon\in\{-1,0,1\}\\\lambda_0(\epsilon)=r-i}}(-q)^{\widetilde{\inv}(\epsilon)}(-q)^{\lambda_1(\epsilon)(\lambda_0(\epsilon)+\lambda_1(\epsilon))}t(\epsilon).
    \end{equation*}
    We have that $\Delta_{i,r}^{1/2,\flat}$ preserves $\Rel_r^\flat.$ Under the injection $\typ_K\colon\Z[H\backslash G/K]\hookrightarrow\Typ_r^0,$ we denote $S_{i,r}^{1/2,\flat}$ the adjoint of $\Delta_{i,r}^{1/2,\flat},$ that is, the operator such that
    \begin{equation*}
        \langle S_{i,r}^{1/2,\flat} f,g\rangle=\langle f,\mathrm{str}^\natural(\Delta^{1/2,\flat}_{i,r}g)\rangle
    \end{equation*}
    for all $f,g\in\Z[\Typ_r^0].$ We denote $\mathbb{T}^{1/2,\flat}_r=\Z[S_{1,r}^{1/2,\flat},\ldots,S_{r,r}^{1/2,\flat}].$ Then we have $\mathbb{T}^\flat_r\subseteq\mathbb{T}^{1/2,\flat}_r$ and we can extend $\mathrm{Sat}\colon\mathbb{T}_r\to\Z[\bm{\mu}_1,\ldots,\bm{\mu}_r]^{\mathrm{sym}}$ to $\mathrm{Sat}\colon\mathbb{T}_r^{1/2,\flat}\to\Z[\sqrt{-q}][\bm{\nu}_1,\ldots,\bm{\nu}_r]^{\mathrm{sym}}$ where $\bm{\mu}_i=-\bm{\nu}_i^2-2$ and $\mathrm{Sat}(S_{i,r}^{1/2,\flat})=(-q)^{\frac{r^2-(r-i)^2}{2}}s_i(\bm{\nu}).$
\end{theorem}

\begin{definition}\label{phiflatDef}
    For $e\in\Typ_r^{0,\flat},$ we denote $\phi^\flat_e\in\Z[\Typ_r^{0,\flat}]$ the image of $\phi^\natural_e$ under the restriction map $\Z\llbracket\Typ_r^0]\to\Z[\Typ_r^{0,\flat}].$ For $N\in\Z_{\ge1}\cup\{\infty\}$ and $0\le k\le r,$ we denote $\Typ_{r,[1,N],0^k}\subseteq\Typ_r^{0,\flat}$ the subset of $e\in\Typ_r^{0,\flat}$ with $\lambda_0(e)=k$ and $e_1\le N.$ We also denote
    \begin{equation*}
        \Delta_{r,[1,N],0^k}\quad\text{resp.}\quad\Phi_{r,[1,N],0^k}
    \end{equation*}
    to be the $\Z$-span of
    \begin{equation*}
        \{\delta_e\colon e\in\Typ_{r,[1,N],0^k}\}\quad\text{resp.}\quad\{\phi^\flat_e\colon e\in\Typ_{r,[1,N],0^k}\}.
    \end{equation*}
    We similarly denote $\Delta_{r,[1,N],0^{\le k}}$ and $\Phi_{r,[1,N],0^{\le k}},$ as well as $\Delta_{r,[1,N],0^{\ge k}}$ and $\Phi_{r,[1,N],0^{\ge k}}.$
\end{definition}

\begin{proposition}\label{flat-phi-computation}
    For $0\le i\le r$ and $e\in\Typ_r^{0,\flat},$ we have that $\Delta_{i,r}^{1/2,\flat}$ preserves $\Rel_r^\phi,$ and that
    \begin{equation*}
        S_{i,r}^{1/2,\flat}\phi^\flat_e=\sum_{g\in\Typ_r^{0,\flat}}\langle\mathrm{str}^\phi(\Delta_{i,r}^{1/2,\flat}\delta_g),\delta_e\rangle\cdot\phi^\flat_g
    \end{equation*}
\end{proposition}
\begin{proof}
    For the first claim, similarly to the proof of \cite[Proposition 5.2.3]{CZ2}, it suffices to check that $t(-1)+q\cdot t(1)$ preserves $\mathrm{Rel}_1^\phi,$ which we did in the proof of \Cref{strphi}.

    For the second claim, the left hand side is
    \begin{equation*}
        \sum_{f\in\Typ_r^{0,\flat}}\langle\delta_f,S_{i,r}^{1/2,\flat}\phi^\flat_e\rangle\cdot\delta_f=\sum_{f\in\Typ_r^{0,\flat}}\langle\mathrm{str}^\flat(\Delta_{i,r}^{1/2,\flat}\delta_f),\phi^\flat_e\rangle\cdot\delta_f=\sum_{f\in\Typ_r^{0,\flat}}\langle\mathrm{str}^\natural\left(\Delta_{r,\phi}\mathrm{str}^\flat(\Delta_{i,r}^{1/2,\flat}\delta_f)\right),\delta_e\rangle\cdot\delta_f
    \end{equation*}
    while the right hand side is
    \begin{equation*}
        \sum_{f\in\Typ_r^{0,\flat}}\sum_{g\in\Typ_r^{0,\flat}}\langle\mathrm{str}^\phi(\Delta_{i,r}^{1/2,\flat}\delta_g),\delta_e\rangle\cdot\langle\delta_f,\phi^\flat_g\rangle\cdot\delta_f=\sum_{f\in\Typ_r^{0,\flat}}\sum_{g\in\Typ_r^{0,\flat}}\langle\mathrm{str}^\phi(\Delta_{i,r}^{1/2,\flat}\delta_g),\delta_e\rangle\cdot\langle\mathrm{str}^\natural(\Delta_{r,\phi}\delta_f),\delta_g\rangle\cdot\delta_f.
    \end{equation*}
    We can write this right hand side as
    \begin{equation*}
        \sum_{f\in\Typ_r^{0,\flat}}\langle\mathrm{str}^\phi\left(\Delta_{i,r}^{1/2,\flat}\mathrm{str}^\natural(\Delta_{r,\phi}\delta_f)\right),\delta_e\rangle\cdot\delta_f=\sum_{f\in\Typ_r^{0,\flat}}\langle\mathrm{str}^\phi\left(\Delta_{r,\phi}\Delta_{i,r}^{1/2,\flat}\delta_f\right),\delta_e\rangle\cdot\delta_f
    \end{equation*}
    as $\Delta_{i,r}^\flat$ and $\Delta_{r,\phi}$ commute and preserve $\mathrm{Rel}^\natural.$ Hence, it remains to see that
    \begin{equation*}
        \mathrm{str}^\natural\left(\Delta_{r,\phi}\mathrm{str}^\flat(\Delta_{i,r}^{1/2,\flat}\delta_f)\right)=\mathrm{str}^\phi\left(\Delta_{r,\phi}\Delta_{i,r}^{1/2,\flat}\delta_f\right).
    \end{equation*}
    Since $\Delta_{r,\phi}\mathrm{str}^\flat(\Delta_{i,r}^{1/2,\flat}\delta_f)$ is supported on $\delta_e$ with $e_r\ge0,$ this is the same as proving
    \begin{equation*}
        \Delta_{r,\phi}\left(\mathrm{str}^\flat(\Delta_{i,r}^{1/2,\flat}\delta_f)-\Delta_{i,r}^{1/2,\flat}\delta_f\right)\in\Rel^\phi
    \end{equation*}
    This follows from the general fact that $\Delta_{r,\phi}$ preserves $\mathrm{Rel}^\natural,$ together with
    \begin{equation*}
        \Delta_{r,\phi}(\delta_{f'}\star\mathrm{Rel}^\flat(m))\in\Rel^\phi
    \end{equation*}
    for all $f'\in\Typ_{r-1}$ and $m\ge1.$ To see this last claim, we denote $\Delta_{r,\phi}=\Delta_{r,\phi}^{<r}\star\frac{t(0)-q^2t(2)}{t(0)-t(2)},$ then we have
    \begin{equation*}
        \Delta_{r,\phi}(\delta_{f'}\star\mathrm{Rel}^\flat(m))=\Delta_{r,\phi}^{<r}(\delta_{f'})\star\frac{t(0)-q^2t(2)}{t(0)-t(2)}\mathrm{Rel}^\flat(m),
    \end{equation*}
    and we can compute that
    \begin{equation*}
        \frac{t(0)-q^2t(2)}{t(0)-t(2)}\mathrm{Rel}^\flat(m)=\mathrm{Rel}^\phi(m).
    \end{equation*}
\end{proof}

\begin{corollary}\label{ConjProof1}
    $\Phi_{r,[1,\infty],0^{\le1}}$ is a $\mathbb{T}_r^{1/2,\flat}$-submodule of $\Z[\Typ_r^{0,\flat}].$ Moreover, its image in $\Z[\frac{1}{q(q+1)}][\Typ_r^{0,\flat}]$ is spanned by $\Phi_{r,[1,\infty]}$ as a $\mathbb{T}^\flat_r\otimes\Z[\frac{1}{q(q+1)}]$-module.
\end{corollary}
\begin{proof}
    To prove that $\Phi_{r,[1,\infty],0^{\le1}}$ is a $\mathbb{T}^{1/2,\flat}_r$-submodule, we need to check that if $g\in\Typ_{r,[1,\infty],0^{\ge2}},$ then $\mathrm{str}^\phi(\Delta_{i,r}^{1/2,\flat}\delta_g)\in\Delta_{[1,\infty],0^{\ge2}}$ for all $0\le i\le r.$ From the fact that $\mathrm{str}^\phi(\delta_{(-1,0^b)})=q^2\delta_{(1,0^b)},$ checking that $\mathrm{str}^\phi(\Delta_{i,r}^{1/2,\flat}\delta_g)\in\Delta_{[1,\infty],0^{\ge2}}$ for $g\in\Typ_{r,[1,\infty],0^{\ge2}}$ reduces to the case $r=2$ and $g=(0,0).$ In this case, we can check that
    \begin{equation*}
        \mathrm{str}^\phi(\Delta_{1,2}^{1/2,\flat}\delta_{(0,0)})=0,\quad \mathrm{str}^\phi(\Delta_{2,2}^{1/2,\flat}\delta_{(0,0)})=q(1+q^2)\delta_{(0,0)}.
    \end{equation*}
    
    Now note that if $f\in\Typ_{r,[1,\infty],0},$ say $f=(f',0)$ for $f'\in\Typ_{r-1,[1,\infty]},$ we have
    \begin{equation*}
        \mathrm{str}^\phi\left(\Delta_1^\flat\delta_f\right)\equiv\mathrm{str}^\phi\left(\delta_{f'}\star(q^{2r}\delta_{(2)}+q^{2r-2}\delta_{(-2)})\right)\mod\Delta_{[1,\infty],0^{\ge1}}
    \end{equation*}
    since $\mathrm{str}^\phi(\delta_{(-1,0)})=q^2\delta_{(1,0)}.$ As we have
    \begin{equation*}
        \mathrm{str}^\phi(q^{2r}\delta_{(2)}+q^{2r-2}\delta_{(-2)})=q^{2r-2}(q(q-1)\delta_{(0)}+q^2(q+1)\delta_{(2)}),
    \end{equation*}
    we conclude that
    \begin{equation*}
        \mathrm{str}^\phi(\Delta^\flat_1\delta_f)\equiv q^{2r}(q+1)\delta_g\mod\Delta_{[1,\infty],0^{\ge1}}
    \end{equation*}
    where $g\in\Typ_r^{0,\flat}$ is such that $\lambda_0(g)=0,$ $\lambda_2(g)=\lambda_2(f)+1$ and $\lambda_i(g)=\lambda_i(f)$ for $i\not\in\{0,2\}.$ Together with the first part of the theorem, this implies that
    \begin{equation*}
        S_{1,r}^\flat\phi^\flat_g\equiv q^{2r}(q+1)\phi^\flat_f\mod\Phi_{r,[1,\infty]},
    \end{equation*}
    which proves the second part of the theorem.
\end{proof}

\begin{definition}\label{Tr-Tr'-def}
    We consider the elements $T_r,T_r'\in \Z[K\backslash G/K]$ given by
    \begin{equation*}
        T_r\defeq\sum_{i=0}^r(-1)^{i}(q^2+1)^iq^{i^2-i}\cdot S_{r-i,r}^\flat,\quad T_r'\defeq\sum_{i=1}^ri(-1)^{i-1}(q^2+1)^{i-1}q^{i^2-i}\cdot S_{r-i,r}^\flat.
    \end{equation*}
    These are such that
    \begin{equation*}
        \mathrm{Sat}(T_r)=q^{r^2}\prod_{i=1}^r\left(\bm{\mu}_i-(q+q^{-1})\right),\quad\mathrm{Sat}(T_r')=q^{r^2-1}\sum_{j=1}^r\prod_{i\neq j}\left(\bm{\mu}_i-(q+q^{-1})\right).
    \end{equation*}
\end{definition}
\begin{theorem}\label{ConjProof2}
    Both $T_r\cdot\Z[\Typ_r^{0,\flat}]$ and $T_r'\cdot\Z[\Typ_r^{0,\flat}]$ are contained in $\Phi_{[1,\infty],0^{\le 1}}.$
\end{theorem}
\begin{proof}
    As in the proof of \cite[Theorem 5.2.6]{CZ}, if we denote
    \begin{equation*}
        S_r^{1/2}(x,y)\defeq\sum_k(-q)^{\frac{k^2}{2}}S_{r-k,r}^{1/2}x^{r-k}y^k,
    \end{equation*}
    we have that its adjoint is induced by
    \begin{equation*}
        \bigstar_{i=1}^r\left(q^{2(r-i)+1}x\cdot t(1)+(-q)^{r-i+\frac{1}{2}}y\cdot t(0)+x\cdot t(-1)\right)
    \end{equation*}
    and that $\mathrm{Sat}(S_r^{1/2}(x,y))=(-q)^{\frac{r^2}{2}}\prod_{i=1}^r(x\bm{\nu}_i+y).$

    We denote $\gamma\defeq\sqrt{-q}+\sqrt{-q}^{-1}.$ We consider the operators $S_r^\pm=S_r(1,\mp\gamma),$ whose adjoints are induced by
    \begin{equation*}
        \Delta_r^\pm\defeq\bigstar_{i=1}^r\left(q^{2(r-i)+1}t(1)\pm(-q)^{r-i}(q+1)\cdot t(0)+t(-1)\right).
    \end{equation*}
    These are such that
    \begin{equation*}
        \mathrm{Sat}(S_r^\pm)=(-q)^{\frac{r^2}{2}}\prod_{i=1}^r(\bm{\nu}_i\mp\gamma).
    \end{equation*}

    We can prove that
    \begin{equation*}
        S_r^\pm\delta_{(0^r)}=\sum_{i=0}^{r}(\pm1)^i(-q)_i\cdot\delta_{(1^{r-i},0^i)}
    \end{equation*}
    where $(x)_n\defeq\prod_{i=1}^n(1-x^i).$ Indeed, it is clear from the definition of $\Delta_r^\pm$ that the coefficient of $\delta_{(1^{r-i},0^i)}$ above only depends on $i$ and is independent of $r,$ and we can verify by induction that
    \begin{equation*}
        \mathrm{str}^\flat(\Delta_r^\pm\delta_{(0^r)})=(-q)_r\sum_{i=0}^r(-q)^{\binom{r-i}{2}}(\pm1)^i\qbinom{r}{i}{-q}\delta_{(1^{r-i},0^i)}.
    \end{equation*}
    To do so, we note that
    \begin{equation*}
        \mathrm{str}^\flat(\delta_{(-1,1^a,0^b)}=(-q)^a\delta_{(1^{a+1},0^b)}+(1-(-q)^a)\delta_{(1^{a-1},0^{b+2})},
    \end{equation*}
    and thus, under the induction hypothesis, we have that $\mathrm{str}^\flat(\Delta_{r+1}^\pm\delta_{(0^{r+1})})$ is
    \begin{equation*}
        (-q)_r\sum_i(-q)^{\binom{r-i}{2}}(\pm1)^i\qbinom{r}{i}{-q}\left(\begin{array}{rl}&(q^{2r+1}+(-q)^{r-i})\delta_{(1^{r-i+1},0^i)}\pm(-q)^r(q+1)\delta_{(1^{r-i},0^{i+1})}\\+&(1-(-q)^{r-i})\delta_{(1^{r-i-1},0^{i+2})}\end{array}\right)
    \end{equation*}
    and we can compute that
    \begin{equation*}
        (q^{2r+1}+(-q)^{r-i})(-q)^{\binom{r-i}{2}}\qbinom{r}{i}{-q}+(-q)^r(q+1)(-q)^{\binom{r-i+1}{2}}\qbinom{r}{i-1}{-q}+(1-(-q)^{r-i+2})(-q)^{\binom{r-i+2}{2}}\qbinom{r}{i-2}{-q}
    \end{equation*}
    is equal to $(1-(-q)^{r+1})(-q)^{\binom{r-i+1}{2}}\qbinom{r+1}{i}{-q}.$

    We can also prove that
    \begin{equation*}
        \phi_{(1^a,0^b)}=\sum_{i=0}^{\lfloor a/2\rfloor}\frac{(-q)_{b+2i}}{(-q)_{b}}\cdot\delta_{(1^{a-2i},0^{b+2i})}.
    \end{equation*}
    By the definition of $\Delta_{r,\phi}$ it is easy to see that if we have $e,f\in\Typ_r^0,$ say $e=(e_1,e')$ and $f=(f_1,f'),$ with $e_1=f_1,$ then $\phi_e(f)=\phi_{e'}(f').$ So it suffices to compute that
    \begin{equation*}
        \phi_{(1^a,0^b)}(0^{a+b})=\begin{cases}
            \frac{(-q)_{a+b}}{(-q)_b}&\text{if }2\mid a,\\
            0&\text{otherwise.}
        \end{cases}
    \end{equation*}
    Namely, given a self-dual lattice $\Lambda$ we are counting
    \begin{equation*}
        \phi_{(1^a,0^b)}(0^{a+b})=\sum_{\substack{L\subseteq\Lambda\\\typ(L)=(1^a,0^b)}}c([\Lambda\colon\varpi\Lambda+L]).
    \end{equation*}
    This is clearly $0$ if $a$ is odd, so assume $a$ is even. Note that all such $L$ also satisfy $\varpi\Lambda\subseteq\varpi L^\vee\subseteq L\subseteq \Lambda,$ with $[\Lambda\colon L]=a/2.$ So by \cite[Proposition 2.6]{CZ} this is
    \begin{equation*}
        \sum_{\substack{L\subseteq\Lambda\\\typ(L)=(1^a,0^b)}}c([\Lambda\colon\varpi\Lambda+L])=c(a/2)\sum_{\substack{\varpi\Lambda\subseteq\varpi L^\vee\subseteq L\subseteq\Lambda\\ [\varpi L^\vee\colon\varpi\Lambda]=a/2}}1=c(a/2)\qbinom{a+b}{a}{-q}\prod_{i=1}^{a/2}(1+q^{2i+1})=\frac{(-q)_{a+b}}{(-q)_b}.
    \end{equation*}

    Hence, by the two computations above we conclude that
    \begin{equation*}
        S_r^\pm\delta_{(0^r)}=\phi_{(1^r)}\pm(q+1)\phi_{(1^{r-1},0)}.
    \end{equation*}
    We note that $T_r=S_r^+S_r^-$ and
    \begin{equation*}
        T_r'=\frac{S_r^+\frac{\partial}{\partial y}S_r^{1/2}(1,\gamma)-S_r^-\frac{\partial}{\partial y}S_r^{1/2}(1,-\gamma)}{2q\gamma},
    \end{equation*}
    which implies that
    \begin{equation*}
        T_r'\delta_{(0^r)}=\frac{\frac{\partial}{\partial y}S_r^{1/2}(1,\gamma)-\frac{\partial}{\partial y}S_r^{1/2}(1,-\gamma)}{2q\gamma}\phi_{(1^r)}-\sqrt{-q}\frac{\frac{\partial}{\partial y}S_r^{1/2}(1,\gamma)+\frac{\partial}{\partial y}S_r^{1/2}(1,-\gamma)}{2q}\phi_{(1^{r-1},0)},
    \end{equation*}
    that is, that $T'_r\delta_{(0^r)}$ is
    \begin{equation*}
        \left(\sum_{i\ge1}2i(-1)^{i-1}q^{2i^2-i}(q+1)^{2(i-1)}S_{r-2i}^{1/2}\right)\phi_{(1^r)}+\left(\sum_{i\ge0}(2i+1)(-1)^iq^{2i^2+i}(q+1)^{2i}S_{r-2i-1}^{1/2}\right)\phi_{(1^{r-1},0)}.
    \end{equation*}
    Since $\delta_{(0^r)}$ generates $\Z[\Typ_r^{0,\flat}]$ under $\mathbb{T}_r^{1/2,\flat}$ by \cite[Corollary 5.3.4]{CZ2} and since $\Phi_{[1,\infty],0^{\le1}}$ is preserved by $\mathbb{T}_r^{1/2,\flat}$ by \Cref{ConjProof1}, this implies our claim.
\end{proof}

\subsection{Relation with local harmonic analysis}\label{invSat}
In this subsection, we recast the previous results of this section into the language of local harmonic analysis. Since we will not use this in the rest of the article, for simplicity and ease of notation we will do so with coefficients in $\C,$ although the results hold integrally (up to possibly inverting $q$ and $q+1$).

Let $\mathcal{G}=U(2r)$ be the algebraic group over $F^+$ which is the unitary group for a split $F/F^+$-Hermitian space $V$ of rank $2r.$ We denote $G=\mathcal{G}(F^+),$ with a hyperspecial subgroup $K.$ We consider $\mathcal{H}^+=U(r)\times U(r)\subseteq G$ corresponding to an orthogonal decomposition of $V$ into two split Hermitian spaces of rank $r.$ Similarly, $\mathcal{H}^-\subseteq G$ correspond to such a decomposition into non-split Hermitian spaces of rank $r.$ We denote $H^\pm=\mathcal{H}^\pm(F^+).$

We consider the $F^+$-variety $X=\mathcal{H}^+\backslash \mathcal{G},$ which is a spherical variety for $\mathcal{G}$ such that
\begin{equation*}
    X(F^+)=H^+\backslash G\sqcup H^-\backslash G.
\end{equation*}
This caries an action of the Hecke algebra $\mathcal{H}_G=C_c^\infty(K\backslash G/K)$ under convolution. Attached to this, we have a lattice of coweights $\Lambda_X$ with Weyl group $W_X$ and a cone of anti-dominant weights $\Lambda_X^+,$ such that we have a relative Cartan decomposition $X(F^+)=\bigsqcup_{\check{\lambda}\in\Lambda_X^+}x_{\check{\lambda}} K$ of its $K$-orbits. This gives an isomorphism
\begin{equation*}
    \text{Cartan}\colon\C[\Lambda_X^+]\rightiso C_c^\infty(X(F^+))^K.
\end{equation*}
Explicitly, we can identify $\Lambda_X=\Typ_r,$ with anti-dominant cone $\Lambda_X^+=\Typ_r^{0,\flat},$ and
\begin{equation*}
    H^+\backslash G/K\sqcup H^-\backslash G/K=X/K=\Lambda_X^+=\Typ_r^{0,\flat}
\end{equation*}
with $\typ_K$ as in \cref{TypFJ}. We consider the elements $\check{\lambda}_i\in\Lambda_X^+$ which correspond to $(0,\ldots,0,1,0,\ldots,0)\in\Typ_r,$ where the $1$ is in the $i$-th position.

We also have a relative Satake transform
\begin{equation*}
    \text{Satake}\colon C_c^\infty(X(F^+))\rightiso \mathcal{H}_X
\end{equation*}
where $\mathcal{H}_X=\C[\Lambda_X]^{W_X}.$ This is compatible with the usual Satake transform $\mathcal{H}_G\rightiso\C[\Lambda_G]^{W_G}$ via $\Lambda_G\subseteq\Lambda_X.$ Explicitly, $\Lambda_G\subseteq\Typ_r$ is the subset where every coordinate is even. For $\check{\lambda}\in\Lambda_X,$ we write $e^{\check{\lambda}}\in\C[\Lambda_X]$ the corresponding element.

\begin{theorem}[Inverse Satake transform]
    The composition $(\text{Satake}\circ\text{Cartan})^{-1}\colon\C[\Lambda_X]^{W_X}\rightiso\C[\Lambda_X^+]$ is given as follows: for $h\in\C[\Lambda_X]^{W_X},$ its image is
    \begin{equation*}
        \left(h\cdot L_X^{1/2}\right)\rvert_{\Lambda_X^+}\in\C[\Lambda_X^+]
    \end{equation*}
    where
    \begin{equation*}
        L_X^{1/2}=\prod_{1\le i<j\le r}\frac{1-e^{\check{\lambda}_j-\check{\lambda}_i}}{1+q^{-1}e^{\check{\lambda}_j-\check{\lambda}_i}}\cdot\prod_{1\le i\le r}\frac{1-e^{-2\check{\lambda}_i}}{1-q^{-1}e^{-2\check{\lambda}_i}}
    \end{equation*}
    is understood to be its power series expansion which is supported in a translate of the cone of dominant coweights, and where $e^{\check{\lambda}}\in\C[\Lambda_X^+]$ is understood to be the function supported on $x_{\check{\lambda}}K$ with value $q^{\sum_{i=1}^r\frac{2(r-i)+1}{2}e_i}(-1)^{\sum_{i=1}^r(r-i)e_i}$ if $\check{\lambda}$ corresponds to $(e_1,\ldots,e_r)\in\Typ_r^{0,\flat}.$
\end{theorem}

Now consider $(\Lambda^+,\Lambda,W)=(\Lambda_X^+,\Lambda_X,W_X),$ but with the nontrivial map
\begin{equation*}
    \phi\colon\C[\Lambda^+]\rightiso\C[\Lambda_X^+]
\end{equation*}
defined by $\phi(1_{x_eK})(x_fK)=\phi_e^\flat(f)$ for $e,f\in\Typ_r^{0,\flat}=\Lambda^+=\Lambda_X^+,$ as defined in \cref{phiflatDef}.
\begin{theorem}[{\cref{flat-phi-computation}}]\label{Lhalf}
    The composition
    \begin{equation*}
        \C[\Lambda]^W=\C[\Lambda_X]^{W_X}\xrightiso{\text{Satake}^{-1}}C_c^\infty(X(F^+))^K\xrightarrow[\sim]{\text{Cartan}^{-1}}\C[\Lambda_X^+]\xrightarrow[\sim]{\phi^{-1}}\C[\Lambda^+]
    \end{equation*}
    is given as follows: for $h\in\C[\Lambda]^W,$ its image is
    \begin{equation*}
        \left(h\cdot L^{1/2}\right)\rvert_{\Lambda^+}\in\C[\Lambda^+]
    \end{equation*}
    where
    \begin{equation*}
        L^{1/2}=\prod_{1\le i<j\le r}\frac{1-e^{\check{\lambda}_j-\check{\lambda}_i}}{1+q^{-1}e^{\check{\lambda}_j-\check{\lambda}_i}}\cdot\prod_{1\le i\le r}\frac{1-e^{-2\check{\lambda}_i}}{1-qe^{-2\check{\lambda}_i}}
    \end{equation*}
    is understood as in the previous theorem.
\end{theorem}
\begin{remark}
    Ultimately, this is because of the description of $\phi^\flat$ in \cref{DeltaPhi}. Note that $\Delta_{r,\phi}$ there is precisely the adjoint of $\frac{L^{1/2}}{L_X^{1/2}}$ under $\langle\cdot,\cdot\rangle_\delta,$ in the notation of \cite[Proposition A.5]{CZ}.
\end{remark}

Finally, consider $\Lambda^{++}\subset\Lambda^+$ to correspond to $\{e\in\Typ_r^{0,\flat}\colon e_r\ge1\}\subset \Typ_r^{0,\flat},$ so that $\C[\Lambda^{++}]\subseteq\C[\Lambda^+]$ corresponds to $\Delta_{r,[1,\infty]}\subseteq\C[\Typ_r^{0,\flat}].$ With this, $\phi(\C[\Lambda^{++}])\subseteq\C[\Lambda_X^+]$ corresponds to $\Phi_{r,[1,\infty]}\subseteq\C[\Typ_r^{0,\flat}].$ Denoting by $T_r$ and $T_r'$ the operators defined in \Cref{Tr-Tr'-def}, we can rephrase the main result of this section.
\begin{theorem}[\cref{ConjProof1,ConjProof2}]
     The $\mathcal{H}_G$-module $\mathcal{H}_G\cdot\phi(\C[\Lambda^{++}])$ contains $T_r\cdot\C[\Lambda_X^+]+T_r'\cdot\C[\Lambda_X^+].$
\end{theorem}

\section{First explicit reciprocity law}\label{ReciprocityChapter}
Let $F/F^+$ a CM extension of number fields, where $F\subseteq\C.$ Recall that for a rational place $w,$ we denote $\Sigma_w^+$ the set of places of $F^+$ above $w,$ and for a place $v$ of $F^+$ or $F,$ we write $\lVert v\rVert$ for the cardinality of the residue field of the completion at $v.$

\subsection{Preliminaries}\label{Reciprocity-Setup}
Let $\Pi$ be an irreducible cuspidal automorphic representation of $\mathrm{GL}_N(\A_F)$ for some $N\ge1.$
\begin{notation}\label{SigmaPi}
    We denote by $\Sigma_\Pi^+$ the smallest set of nonarchimedean places of $F^+$ such that if $v\not\in\Sigma_\Pi^+$ is a nonarchimedean place of $F^+,$ then its underlying rational prime if unramified in $F$ and $\Pi_w$ is unramified for all $w\mid v$ places of $F.$
\end{notation}
\subsubsection{Automorphic representations}
\begin{definition}[{\cite[Definition 1.1.3]{LTXZZ}}]\label{RelevantDefinition}
    We say that an automorphic representation $\Pi$ of $\mathrm{GL}_N(\A_F)$ with $N\ge1$ is \emph{relevant} if
    \begin{enumerate}
        \item $\Pi$ is an irreducible cuspidal automorphic representation;
        \item $\Pi\circ\cplx\iso\Pi^\vee,$ that is, $\Pi$ is conjugate self-dual;
        \item for every archimedean place $w$ of $F,$ $\Pi_w$ is isomorphic to the irreducible principal series representation induced by the characters $(\mathrm{arg}^{1-N},\mathrm{arg}^{3-N},\ldots,\mathrm{arg}^{N-3},\mathrm{arg}^{N-1}),$ where $\mathrm{arg}\colon\C^\times\to\C^\times$ is the argument character $\mathrm{arg}(z)\defeq z/\sqrt{z\overline{z}}.$
    \end{enumerate}
\end{definition}
\begin{remark}
    A relevant representation $\Pi$ is automatically regular algebraic in the sense of \cite[Definition 3.12]{Clozel}.
\end{remark}

We record the following results about attached Galois representations from the works of Chenevier--Harris \cite{Chenevier-Harris} and Caraiani \cite{Caraiani1,Caraiani2}.
\begin{proposition}[{\cite[Proposition 3.2.4]{LTXZZ}}]\label{GaloisExistence}
    Let $\Pi$ be a relevant representation of $\mathrm{GL}_N(\A_F).$
    \begin{enumerate}
        \item For every nonarchimedean place $w$ of $F,$ $\Pi_w$ is tempered.
        \item For every rational prime $\ell$ and every isomorphism $i_\ell\colon\C\rightiso\overline{\Q}_\ell,$ there is a semisimple continuous homomorphism
        \begin{equation*}
            \rho_{\Pi,i_\ell}\colon \Gamma_{F}\to\mathrm{GL}_N(\overline{\Q}_\ell)
        \end{equation*}
        defined up to conjugation, such that for every nonarchimedean place $w$ of $F,$ the Frobenius semisimplification of the associated Weil--Deligne representation of $\rho_{\Pi,\ell}\rvert_{\Gamma_{F_w}}$ corresponds to the irreducible admissible representation $i_\ell\Pi_w\lvert\det\rvert_w^{\frac{1-N}{2}}$ of $\mathrm{GL}_N(F_w)$ under the local Langlands correspondence. Moreover, $\rho_{\Pi,\ell}^\cplx$ and $\rho_{\Pi,\ell}^\vee(1-N)$ are conjugate.
    \end{enumerate}
\end{proposition}

\begin{definition}
    The \emph{coefficient field} $\Q(\Pi)$ of $\Pi$ is defined to be the smallest subfield of $\C$ such that for every $\sigma\in\mathrm{Aut}(\C/\Q(\Pi))$ we have that $\Pi^\infty$ and $\Pi^\infty\otimes_{\C,\sigma}\C$ are isomorphic.
\end{definition}
\begin{definition}\label{StrongCoefDef}
    Let $\Pi$ be relevant. A \emph{strong coefficient field} of $\Pi$ is a number field $E\subseteq\C$ containing $\Q(\Pi)$ such that for every for every nonarchimedean place $\lambda$ of $E,$ there exists a continuous homomorphism
    \begin{equation*}
        \rho_{\Pi,\lambda}\colon\Gamma_{F}\to\mathrm{GL}_N(E_\lambda)
    \end{equation*}
    defined up to conjugation, satisfying that for all $i_\ell\colon\C\rightiso\overline{\Q}_\ell$ inducing $\lambda$ we have that $\rho_{\Pi,\lambda}\otimes_{E_\lambda}\overline{\Q}_\ell$ and $\rho_{\Pi,i_\ell}$ are conjugate.
\end{definition}
\begin{remark}
For $\Pi$ relevant, strong coefficient fields $E$ exist by \cite[Proposition 3.2.5]{Chenevier-Harris}.
\end{remark}

\subsubsection{Satake parameters}
\begin{definition}
    For a nonarchimedean place $w$ of $F$ such that $\Pi_w$ is unramified, let
    \begin{equation*}
        \underline{\alpha}(\Pi_w)=\{\alpha(\Pi_w)_1,\ldots,\alpha(\Pi_w)_N\}\subseteq\C
    \end{equation*}
    be the Satake parameter of $\Pi_w.$ We consider the Satake polynomial
    \begin{equation*}
        P_{\Pi,w}(T)=\prod_{i=1}^N(T-\alpha(\Pi_w)_i)\in\C[T].
    \end{equation*}
\end{definition}

\begin{definition}
    For $v$ a nonarchimedean place of $F^+$ which is unramified over $F,$ we consider the local spherical Hecke algebra
    \begin{equation*}
        \mathbb{T}_{N,v}\defeq\Z[U_{N,v}(\O_{F^+_v})\backslash U_{N,v}(F^+_v)\slash \O_{N,v}(F^+_v)]
    \end{equation*}
    attached to the unitary group $U_{N,v}$ over $\O_{F^+_v}$ associated to the hermitian form given by the antidiagonal matrix
    \begin{equation*}
        \begin{pmatrix}
            0&\cdots&0&1\\0&\cdots&1&0\\\vdots&\iddots&\vdots&\vdots\\1&\cdots&0&0
        \end{pmatrix}.
    \end{equation*}
    Let $A_{N,v}$ be the maximal split diagonal subtorus of $U_{N,v},$ and $W_{N,v}$ be the Weyl group.
\end{definition}
\begin{definition}[{\cite[Construction 3.1.8]{LTXZZ}}]
    Let $\Pi$ be conjugate self-dual. Consider $v\not\in\Sigma^+_\Pi.$ Let $\delta(v)=1$ if $v$ is inert in $F,$ and $\delta(v)=2$ if it is split. We define the Satake homomorphism
    \begin{equation*}
        \phi_{\Pi,v}\colon\mathbb{T}_{N,v}\to\C
    \end{equation*}
    given by the composition of the Satake transform $\mathbb{T}_{N,v}\to\Z\lVert v\rVert^{\pm\delta(v)/2}[X_*(A_{N,v})]^{W_{N,v}}$ with evaluating with the $\delta(v)\cdot N$ Satake parameters $(\underline{\alpha}(\Pi_w))_w$ as $w$ runs through the $\delta(v)$ places of $F$ above $v.$
\end{definition}
\begin{definition}
    For a finite set of nonarchimedean places $\Sigma^+$ of $F^+$ containing all the ramified places, we consider the abstract Hecke algebra
    \begin{equation*}
        \mathbb{T}_N^{\Sigma^+}\defeq\bigotimes_{v\not\in\Sigma^+\cup\Sigma^+_\infty}\mathbb{T}_{N,v}
    \end{equation*}
    with restricted tensor product with respect to the unit elements. If $\Pi$ is conjugate self-dual, then we also have the Satake homomorphism
    \begin{equation*}
        \phi_\Pi\colon\mathbb{T}_N^{\Sigma^+_\Pi}\to\C
    \end{equation*}
    given by the tensor product of the local Satake homomorphisms $\phi_{N,v}.$
\end{definition}
\begin{proposition}[{\cite[Proposition 4.1 and Remark 4.2]{Shin-Templier}}]\label{SatakeHomomorphism}
    When $\Pi$ is relevant, $\phi_\Pi$ takes values in $\O_{\Q(\Pi)}.$ In particular, if a place $v$ of $F^+$ is inert in $F,$ we have $P_{\Pi,v}(T)\in\O_{\Q(\Pi)}\otimes_\Z\Z[\lVert v\rVert^{\pm1}].$
\end{proposition}

\begin{definition}[{\cite[Definition 3.1.5]{LTXZZ}}]
    Assume that $N$ is even. Let $\Pi$ be relevant, $v\not\in\Sigma^+_\Pi$ a place of $F^+$ which is inert in $F.$ Consider a ring morphism $f\colon\O_{\Q(\Pi)}\otimes_\Z\Z[\lVert v\rVert^{\pm1}]\to L.$ We say that
    \begin{enumerate}
        \item $\Pi_v$ is \emph{level-raising special} on $L$ if $f(P_{\Pi,v}(\lVert v\rVert))=0$ and $f(P_{\Pi,v}'(\lVert v\rVert))\in L^\times.$
        \item $\Pi_v$ is \emph{intertwining generic} on $L$ if $f(P_{\Pi,v}(-1))\in L^\times.$
    \end{enumerate}
\end{definition}

\subsection{Arithmetic level raising}\label{Reciprocity-ALR}
\subsubsection{Set-up}\label{ALRSetup}
Let $N=2r$ be an integer with $r\ge1.$ We will assume that
\begin{enumerate}[label=(LR\arabic*)]
\setcounter{enumi}{-1}
    \item\label{LR0} \cite[Hypothesis 3.2.10]{LTXZZ} holds true for such $N.$
\end{enumerate}
We note that this is true if $F^+\neq\Q$ by \cite{KSZ} or if $N=2$ by \cite[Theorem D.6(2)]{LiuFJ} (see \cite[Proposition 3.2.11]{LTXZZ}).

We consider
\begin{itemize}
    \item $\Pi$ a relevant representation of $\mathrm{GL}_N(\mathbb{A}_F)$ (\Cref{RelevantDefinition}).
    \item $E\subseteq\C$ a strong coefficient field $E\subseteq\C$ for $\Pi$ (\Cref{StrongCoefDef}).
    \item A finite set $\Sigma_{\mathrm{min}}^+$ of nonarchimedean places of $F^+$ containing $\Sigma_\Pi^+$ (\Cref{SigmaPi}).
    \item A (possibly empty) finite set $\Sigma_{\mathrm{\lr}}^+$ of nonarchimedean places of $F^+$ that are inert in $F$ and strongly disjoint from $\Sigma_{\mathrm{min}}^+.$
    \item A finite set $\Sigma^+$ of places of $F^+$ containing $\Sigma_{\mathrm{min}}^+\cup\Sigma_{\mathrm{\lr}}^+.$
    \item A prime $\lambda$ of $E,$ with underlying rational prime $\ell.$
    \item A positive integer $m.$
    \item A standard definite hermitian space $V^\circ$ of rank $N$ over $F,$ which is split outside $\Sigma_\infty^+\cup\Sigma_{\mathrm{min}}^+\cup\Sigma_{\mathrm{\lr}}^+,$ and non-split for $\Sigma_{\mathrm{\lr}}^+.$ We consider an object $K^\circ\in\mathfrak{K}(V^\circ)$ which is factorizable $K^\circ=\prod_{v\not\in\Sigma_\infty^+}K^\circ_{v}$ such that i) $K^\circ_{v}$ is hyperspecial for $v\not\in\Sigma_{\mathrm{min}}^+\cup\Sigma_{\mathrm{\lr}}^+,$ given by the stabilizer of a self-dual lattice $\Lambda^\circ_v,$ ii) $K^\circ_v$ is special for $v\in\Sigma_{\mathrm{\lr}}^+,$ and iii) $K^\circ_v$ is transferable for $v\in\Sigma_{\mathrm{min}}^+$ (see \cite[Definition D.2.1]{LTXZZ})\footnote{By \cite[Lemma D.2.2(3)]{LTXZZ}, every sufficiently small open compact subgroup is transferable.}.
\end{itemize}

By \Cref{SatakeHomomorphism} we have the Satake homomorphism
\begin{equation*}
    \phi_{\Pi}\colon \mathbb{T}_N^{\Sigma^+_\Pi}\to\O_E,
\end{equation*}
and for every prime $\lambda$ of $E,$ by we have a global Galois representation
\begin{equation*}
    \rho_{\Pi,\lambda}\colon\Gamma_{F}\to\mathrm{GL}_N(E_\lambda),
\end{equation*}
such that $\rho_{\Pi,\lambda}^c$ and $\rho_{\Pi,\lambda}^\vee(1-N)$ are conjugate.

We require that this data satisfies
\begin{enumerate}[label=(LR\arabic*)]
    \item\label{LR1} $\ell\ge 2(N+1)$ and $\ell$ is unramified in $F.$
    \item\label{LR2} $\Sigma_{\mathrm{min}}^+\cap\Sigma_\ell^+=\emptyset$ and $\ell\nmid\lVert v\rVert(\lVert v\rVert^2-1)$ for all $v\in\Sigma_{\mathrm{\lr}}^+.$
    \item\label{LR3} $\rho_{\Pi,\lambda}\rvert_{\Gamma_{F(\zeta_\ell)}}$ is residually absolutely irreducible.
    \item\label{LR4} $\bar{\rho}_{\Pi,\lambda,+}$ is rigid for $(\Sigma^+_{\mathrm{min}},\Sigma^+_{\mathrm{\lr}}),$ (see \cite[Remark 6.1.7]{LTXZZ} and \cite[Definition 6.3.4]{LTXZZ}),
    \item\label{LR5} $\mathbb{T}_N^{\Sigma^+}\xrightarrow{\phi_\Pi}\O_E\to\O_E/\lambda$ is cohomologically generic (\cite[Definition D.1.1]{LTXZZ}).
\end{enumerate}
\begin{definition}
We consider the ideals of $\mathbb{T}^{\Sigma^+\cup\Sigma^+_p}$
\begin{equation*}
    \begin{split}
        \m&\defeq \mathbb{T}^{\Sigma^+\cup\Sigma^+_p}\cap \ker\left(\mathbb{T}^{\Sigma^+}\xrightarrow{\phi_\Pi}\O_E\to\O_E/\lambda\right)\\
        \n&\defeq \mathbb{T}^{\Sigma^+\cup\Sigma^+_p}\cap \ker\left(\mathbb{T}^{\Sigma^+}\xrightarrow{\phi_\Pi}\O_E\to\O_E/\lambda^m\right)
    \end{split}
\end{equation*}
\end{definition}
\begin{remark}\label{VanishingCohomology}
    Condition \Cref{LR5} implies that $H^i_{\et}(\Sh(V,K)_{\overline{\Q}},\O_\lambda)_\m$ vanishes for $i\neq N-1,$ and that it is a finite free $\O_\lambda$-module for $i=N-1$ (see \cite[Remark 6.1.5]{LTXZZ}).
\end{remark}

We further choose
\begin{itemize}
    \item a special inert prime $\p$ of $F^+,$ with underlying rational prime $p,$
\end{itemize}
and we assume that it satisfies the following.
\begin{enumerate}[label=(P\arabic*)]
    \item\label{P1} $\Sigma^+\cap\Sigma_p^+=\emptyset.$
    \item\label{P2} $\ell\nmid p(p^2-1).$
    \item\label{P3} There exists $i_p\colon\overline{\Q}_p\rightiso\C$ inducing $\p$ and a CM type $\Phi$ containing $\tau_\infty$ for which the compositum of $i_p(\Q_p),$ $F$ and $F_\Phi$ is $\Q_{p^2}.$
    \item\label{P4} $\Pi_\p$ is level raising special under the reduction map $\O_E\to\O_E/\lambda^m.$
    \item\label{P5} $\Pi_\p$ is intertwining generic under the reduction map $\O_E\to\O_E/\lambda.$
    \item\label{P6} The natural map
    \begin{equation*}
        \frac{(\O_E/\lambda^m)[\Sh(V^\circ,K^\circ)]}{\mathbb{T}^{\Sigma^+\cup\Sigma_p^+}\cap\ker\phi_\Pi}\to\frac{(\O_E/\lambda^m)[\Sh(V^\circ,K^\circ)]}{\mathbb{T}^{\Sigma^+}\cap\ker\phi_\Pi}
    \end{equation*}
    is an isomorphism of nontrivial $\O_E/\lambda^m$-modules.
\end{enumerate}
Note that under \Cref{P3}, $i_p$ induces an isomorphism $i_p\colon\Q_p^\Phi\rightiso F_\p=\Q_{p^2}.$ We then also choose
\begin{itemize}
    \item In addition to \Cref{P3}, the data of $W_0,$ $K_0^p$ as in \Cref{Shimura-RSZ}.
    \item $V$ a standard indefinite hermitian space of rank $N$ over $F/F^+$ which is a nearby space to $V^\circ$ at $\p.$ That is, equipped with an isometry $j\colon V\otimes_{F^+}\A_{F^+}^{\infty,\p}\rightiso V^\circ\otimes_{F^+}\A_{F^+}^{\infty,\p},$ as well as a lattice $\Lambda_\p\subseteq V\otimes_{F^+}F^+_\p$ which is almost self-dual. We denote $K_\p\defeq\mathrm{Stab}(\Lambda_\p),$ and $K=j^*(K^{\circ,\p})K_\p.$
\end{itemize}
We also choose the following data for convenience, although the main results of this section do not depend on this choice.
\begin{itemize}
    \item A lattice $\Lambda^\bullet_\p\subseteq V^\circ\otimes_{F^+}F^+_\p$ such that $\Lambda^{\bullet,\vee}=\varpi\Lambda^\bullet$ for $\varpi$ an uniformizer of $\O_{F^+_\p}.$ We denote $K^\bullet_\p\defeq\mathrm{Stab}(\Lambda^\bullet)\subseteq U(V^\circ)(F^+_\p),$ and $K^\bullet=K^{\circ,\p}K^\bullet_\p.$
\end{itemize}
Note also that this gives us a basic uniformization datum $z^b=(V^\circ,j^p,(\Lambda^\circ_\q)_{\q\mid p},\Lambda^\bullet_\p)$ for $V$ as well as a complex uniformization datum $z^\eta=(\Lambda_\p,j^*(\Lambda^\circ_{\q})_{\q\neq\p}).$

Given this data, we apply the discussion of \Cref{ShimuraChapter}. In particular we have the moduli space $\mathcal{M}_\p\defeq\mathcal{M}_\p(V,K^p)\to\mathcal{T}_\p$ (\Cref{moduliProblem}) and the correspondences
\begin{equation*}
    S_\p^?\xleftarrow{\pi^?}B_\p^?\xrightarrow{i^?}M_\p^?
\end{equation*}
for $?\in\{\circ,\bullet\}$ and uniformizations
\begin{equation*}
    \Theta^\eta=\Theta_{z^\eta}\colon \mathcal{M}_\p^\eta\rightiso\Sh(V,K)\times_{\Spec F}\mathcal{T}_\p^\eta
\end{equation*}
and
\begin{equation*}
    \Theta^b=\Theta_{z^b}\colon S_\p^?(\overline{\F}_p)\rightiso\Sh(V^\circ,K^?)\times T_\p(\overline{\F}_p).
\end{equation*}

In order to apply the discussion of \Cref{WeightChapter}, we also record the following.
\begin{proposition}\label{nonproperComparison}
    For every $i$ and $K'\in\mathfrak{K}(V^\circ),$ the specialization maps
    \begin{equation*}
        H^i_{\et,c}(\mathcal{M}^\eta_{\p,\overline{\Q}_p}(V,K'),\O_\lambda)\to H^i_{\et,c}(\overline{M}_\p(V,K'),R\Psi\O_\lambda)
    \end{equation*}
    and
    \begin{equation*}
        H^i_\et(\mathcal{M}^\eta_{\p,\overline{\Q}_p}(V,K'),\O_\lambda)\to H^i_\et(\overline{M}_\p(V,K'),R\Psi\O_\lambda)
    \end{equation*}
    are isomorphisms. Moreover, for each $?\in\{\circ,\bullet,\dagger\}$ and $i,$ the natural maps $H^i_{\et,c}(\overline{M}_\p^?(V,K'),\O_\lambda)_\m\to H^i_\et(\overline{M}_\p^?(V,K'),\O_\lambda)_\m$ are isomorphisms. In particular, the scheme $\mathcal{M}_\p$ is \pss{} (\Cref{pssDef}).
\end{proposition}
\begin{proof}
    When $F^+\neq\Q,$ the Shimura variety $\mathcal{M}_\p$ is proper (\cite[Remark 5.2.8]{LTXZZ}), and thus the claim follows from proper base change.
    
    For the non-proper case, the first claim is due to \cite[Corollary 5.20]{Stroh-Lan}. The second claim for $?\in\{\circ\dagger\}$ follows since $M_\p^\bullet,M_\p^\dagger$ are projective. For $?=\bullet,$ this follows from (the proof of) \cite[Lemma 6.1.11]{LTXZZ}.
\end{proof}

\subsubsection{Weight spectral sequence}
Recall that $\mathcal{M}_\p$ is semistable, and we have the decomposition
\begin{equation*}
    M_\p=M_\p^\circ\cup M_\p^\bullet.
\end{equation*}
In light of \Cref{nonproperComparison}, we thus have the localized weight spectral sequence
\begin{equation*}
    E_{1,\m}^{i,j}\Rightarrow H^{i+j}_\et((\mathcal{M}_\p^\eta)_{\overline{F}},\O_\lambda)\xrightiso{\Theta^\eta}H^{i+j}_\et(\Sh(V,K)_{\overline{F}},\O_\lambda)\times H^0_\et(\mathcal{T}_{\p,\overline{F}},\O_\lambda).
\end{equation*}
where the first page has the rows
\begin{equation*}
    \begin{tikzcd}
        H^{j-2}_\et(\overline{M}_\p^\dagger,\O_\lambda(-1))_\m\arrow{r}{(m^{\dagger\circ}_*,-m^{\dagger\bullet}_*)}\arrow[d,equal]&[+15pt]H^{j}_\et(\overline{M}_\p^\circ,\O_\lambda)_\m\oplus H^{j}_\et(\overline{M}_\p^\bullet,\O_\lambda)_\m\arrow[d,equal]\arrow{r}{m^{\dagger\circ,*}-m^{\dagger\bullet,*}}&[+15pt]H^j_\et(\overline{M}_\p^\dagger,\O_\lambda)_\m\arrow[d,equal]\\[-10pt]
        E_{1,\m}^{-1,j}\arrow{r}{d_{1,\m}^{-1,j}}&E_{1,\m}^{0,j}\arrow{r}{d_{1,\m}^{0,j}}&E_{1,\m}^{1,j}
    \end{tikzcd}
\end{equation*}
\begin{theorem}\label{Cohomologymodp}
We have
    \begin{enumerate}
        \item For $i\not\in\{2r-2,4r-2\},$ the restriction map
        \begin{equation*}
            m^{\dagger\circ,*}\colon H^i_\et(\overline{M}_\p^\circ,\O_\lambda)\to H^i_\et(\overline{M}_\p^\dagger,\O_\lambda)
        \end{equation*}
        is an isomorphism. Dually, for $i\not\in\{2,2r\},$ the map
        \begin{equation*}
            m^{\dagger\circ}_*\colon H^{i-2}_\et(\overline{M}_\p^\dagger,\O_\lambda(-1))\to H^i_\et(\overline{M}_\p^\circ,\O_\lambda)
        \end{equation*}
        is an isomorphism.
        \item $H^{2(r-1)}_\et(\overline{M}_\p^\dagger,\O_\lambda(r-1))$ is a free $\O_\lambda[\mathfrak{T}]$-module with trivial $\Gamma_{\F_{p^2}}$ action.
        \item $H^{2(r-1)}_\et(\overline{M}_\p^\bullet,\O_\lambda(r-1))_\m$ and $H^{2r}_\et(\overline{M}_\p^\bullet,\O_\lambda(r))_\m$ are free $\O_\lambda[\mathfrak{T}]$-module with trivial $\Gamma_{\F_{p^2}}$ action.
        \item $H^{2r-1}_\et(\overline{M}_\p^\bullet,\O_\lambda(r))_\m$ is a free $\O_\lambda[\mathfrak{T}]$-module. Moreover, if $\{\alpha_1^{\pm1},\ldots\alpha_r^{\pm1}\}$ denote the roots of the Satake polynomial $P_{\Pi,\p}\mod\lambda$ in a finite extension of $\O_\lambda/\lambda,$ then the generalized Frobenius eigenvalues of the $(\O_\lambda/\lambda)[\Gamma_{\F_{p^2}}]$-module $H^{2r-1}_\et(\overline{M}_\p^\bullet,\O_\lambda(r))_\m\otimes_{\O_\lambda}\O_\lambda/\lambda$ are contained in $\{p\alpha_1^{\pm1},\ldots,p\alpha_r^{\pm1}\}\setminus\{1,p^2\}.$
        \item $E_{2,\m}^{i,j}$ is a free $\O_\lambda[\mathfrak{T}]$-module, and vanishes if $(i,j)\not\in\{(-1,2r),(0,2r-1),(1,2r-2)\}.$
        \item There exists a positive integer $\mu$ such that
        \begin{equation*}
            H^{2r-1}_\et(\mathcal{M}_{\p,\overline{F}},R\Psi\O_\lambda(r))/\n\iso(\bar{\RR}^{(m)\cplx})^{\oplus\mu}\times\mathfrak{T}
        \end{equation*}
        of $\O_\lambda[\Gamma_F][\mathfrak{T}]$-modules, where $\RR^\cplx$ is the $\Gamma_F$-stable $\O_\lambda$ lattice in $\rho_{\Pi,\lambda}(r)^\cplx,$ which is unique up to homothety by \Cref{L3}, and $\bar{\RR}^{(m)\cplx}$ is its reduction modulo $\lambda^m.$
    \end{enumerate}
\end{theorem}
\begin{proof}
    (1) and (2) follow from (the proof of) \cite[Lemma 5.6.2]{LTXZZ}. (4), (5) and (6) follow from (the proof of) \cite[Theorem 6.3.5]{LTXZZ}.

    For (3), in the case $r=1$ we are looking at the modules $H^0_\et(\overline{M}_\p^\bullet,\O_\lambda)$ and $H^2_\et(\overline{M}_\p^\bullet,\O_\lambda(1)),$ which clearly have trivial Galois action.
    
    If $r>1,$ since $E_{2,\m}^{0,2r}=0$ by (4) and the restriction $m^{\dagger\circ,*}\colon H^{2r}_\et(\overline{M}_\p^\circ,\O_\lambda)_\m\to H^{2r}_\et(\overline{M}_\p^\dagger,\O_\lambda)_\m$ is an isomorphism by (1), we obtain a surjection
    \begin{equation*}
        H^{2(r-1)}_\et(\overline{M}_\p^\dagger,\O_\lambda(r-1))_\m\twoheadrightarrow H^{2r}_\et(\overline{M}_\p^\bullet,\O_\lambda(r))_\m,
    \end{equation*}
    and thus we conclude that $H^{2r}_\et(\overline{M}_\p^\bullet,\O_\lambda(r))_\m$ has trivial Galois action by (2). Dually, since $E_{2,\m}^{0,2(r-1)}=0$ and $m^{\dagger\circ,*}\colon H^{2(r-2)}_\et(\overline{M}_\p^\dagger,\O_\lambda(-1))_\m\to H^{2(r-1)}_\et(\overline{M}_\p^\dagger,\O_\lambda)_\m$ is an isomorphism, we obtain an injection
    \begin{equation*}
        H^{2(r-1)}_\et(\overline{M}_\p^\bullet,\O_\lambda(r-1))_\m\hookrightarrow H^{2(r-1)}_\et(\overline{M}_\p^\dagger,\O_\lambda(r-1))_\m.
    \end{equation*}
    Thus (3) follows.
\end{proof}

We now apply the discussion of \Cref{Weight-Potential}.
\begin{proposition}\label{sing-potential}
$\O_\lambda$ is a very nice coefficient (\Cref{nicecoef}) for the localized spectral sequence $E^{i,j}_{s,\m}(r).$
\end{proposition}
\begin{proof}
(N1) is satisfied by \Cref{Cohomologymodp}(5).

For (N2), we need to check that $E_{2,\m}^{-1,2r}(r-1)$ and $E_{2,\m}^{0,2r-1}(r-1)$ have no nontrivial subquotient on which $\Gamma_{\F_{p^2}}$ acts trivially. For $E_{2,\m}^{-1,2r}(r-1),$ we have
\begin{equation*}
    E_{2,\m}^{-1,2r}(r-1)\subseteq H^{2(r-1)}_\et(\overline{M}_\p^\dagger,\O_\lambda(r-2))_\m
\end{equation*}
which has Galois action given by $\O_\lambda(-1)$ because of \Cref{Cohomologymodp}(2), and thus the claim follow from our assumption \Cref{P2} that $\ell\nmid p^2-1.$ For $E_{2,\m}^{0,2r-1}(r-1),$ we have
\begin{equation*}
    E_{2,\m}^{0,2r-1}(r-1)=H_\et^{2r-1}(\overline{M}_\p^\bullet,\O_\lambda(r-1))_\m,
\end{equation*}
and the claim follows from \Cref{Cohomologymodp}(4) and assumption \Cref{P4}.

For (N3), we will show that $H^{2(r-1)}_\et(\overline{M}_\p^?,\O_\lambda(r-1))_\m$ and $H^{2r}_\et(\overline{M}_\p^?,\O_\lambda(r))_\m$ have trivial $\Gamma_{\F_{p^2}}$ action for each $?\in\{\circ,\bullet\}.$ For $?=\circ$ this is trivial since $\overline{M}_\p^\circ$ is a union of projective spaces by \Cref{Geometrymodp}. For $?=\bullet,$ this is \Cref{Cohomologymodp}(3).
\end{proof}

Now consider the spaces $B^r(\mathcal{M}_\p,\O_\lambda)_\m=\ker d_{1,\m}^{0,2r}(r)$ and $A^r(\mathcal{M}_\p,\O_\lambda)_\m=\coker d_{1,\m}^{-1,2(r-1)}(r-1)$ as in \Cref{ABDef}. Note that they have trivial $\Gamma_{\F_{p^2}}$-action by \Cref{Cohomologymodp}(3). In this case, the potential map (\Cref{LiuFactor}) $\Delta^r_\m\colon A^r(\mathcal{M}_\p,\O_\lambda)_\m\to B^r(\mathcal{M}_\p,\O_\lambda)_\m$ is induced by the composition $d_{1,\m}^{-1,2r}(r)\circ d_{1,\m}^{0,2(r-1)}(r-1).$

We will now identify the module $\coker\Delta^r_\m/\n.$ Recall that we have intertwining operators
\begin{equation*}
    T^{\circ\bullet}\colon \O_\lambda[\Sh(V^\circ,K^\bullet)]\to \O_\lambda[\Sh(V^\circ,K^\circ)],\quad T^{\bullet\circ}\colon \O_\lambda[\Sh(V^\circ,K^\circ)]\to \O_\lambda[\Sh(V^\circ,K^\bullet)]
\end{equation*}
and that we denote $I^\circ\defeq T^{\circ\bullet}\circ T^{\bullet\circ}.$

We consider the maps in cohomology
\newcommand{\inc}[1]{\mathrm{inc}_{#1}^*}
\begin{equation*}
    \inc{\bullet}\colon H_{\et}^{2r}(\overline{M}_\p^\bullet,\O_\lambda(r))\xrightarrow{i^{\bullet*}}H_\et^{2r}(\overline{B}_\p^\bullet,\O_\lambda(r))\xrightarrow{\pi_!^\bullet}H_\et^0(\overline{S}_\p^\bullet,\O_\lambda)\xrightiso{\Theta^b}\O_\lambda[\Sh(V^\circ,K^\bullet)]\times T_\p(\overline{\F}_p),
\end{equation*}
and
\begin{equation*}
\begin{split}
    \inc{\circ}\colon& H_\et^{2r}(\overline{M}_\p^\circ,\O_\lambda(r))\xrightarrow{i^{\circ*}}H_\et^{2r}(\overline{B}_\p^\circ,\O_\lambda(r))\xrightarrow{\cup \xi^{r-1}}H_\et^{2(2r-1)}(\overline{B}_\p^\circ,\O_\lambda(2r-1))\\
    &\xrightarrow{\pi_!^\circ}H_\et^0(\overline{S}_\p^\circ\O_\lambda)\xrightiso{\Theta^b}\O_\lambda[\Sh(V^\circ,K^\circ)]\times T_\p(\overline{\F}_p).
\end{split}
\end{equation*}
We define
\begin{equation*}
    \nabla\colon H^{2r}_\et(\overline{M}_\p^\circ,\O_\lambda(r))\oplus H^{2r}_\et(\overline{M}_\p^\circ,\O_\lambda(r))\to \O_\lambda[\Sh(V^\circ,K^\circ)]\times T_\p(\overline{\F}_p)
\end{equation*}
to be the sum of the two maps
\begin{equation*}
    (p+1)T^{\circ\bullet}\circ\inc{\bullet}\quad\text{and}\quad I^\circ\circ\inc{\circ}.
\end{equation*}

\begin{theorem}\label{first-rec}
We have
\begin{equation*}
    \im(\nabla_\m\circ\Delta^r_\m)\subseteq\n\cdot\O_\lambda[\Sh(V^\circ,K^\circ)]_\m\times T_\p(\overline{\F}_p),
\end{equation*}
and the induced map
\begin{equation*}
    \nabla_{/\n}\colon\mathrm{coker}\ \Delta^r_\m/\n\to\O_\lambda[\Sh(V^\circ,K^\circ)]/\n\times T_\p(\overline{\F}_p)
\end{equation*}
is an isomorphism.
\end{theorem}
\begin{proof}
The image of the composition
\begin{equation*}
    B_r(\mathcal{M}_\p,\O_\lambda)_\m\xrightarrow{\Delta^r_\m}B^r(\mathcal{M}_\p,\O_\lambda)_\m\xrightarrow{\nabla_\m}\O_\lambda[\Sh(V^\circ,K^\circ)]\times T_\p(\overline{\F}_p)
\end{equation*}
is simply the image of the composition $d_{1,\m}^{-1,2r}(r)\circ d_{1,\m}^{0,2(r-1)}(r-1)\circ\nabla_\m.$ By \cite[Lemma 5.9.3(8), Proposition 6.3.1(1)]{LTXZZ}, such image is
\begin{equation*}
    T^{\circ\bullet}((p+1)R^\bullet-T^{\bullet\circ}T^{\circ\bullet})\O_\lambda[\Sh(V^\circ,K^\bullet)]_\m\times T_\p(\overline{\F}_p)
\end{equation*}
where $R^\bullet$ is defined in \cite[Proposition 5.8.8]{LTXZZ}. By \cite[Lemma 6.1.9, Lemma B.3.6]{LTXZZ}, this is
\begin{equation*}
    I^\circ((p+1)R^\circ-I^\circ)\O_\lambda[\Sh(V^\circ,K^\circ)]_\m\times T_\p(\overline{\F}_p).
\end{equation*}
where $R^\circ$ is as in \cite[Lemma B.3.5(2)]{LTXZZ}. We have that $(p+1)R^\circ-I^\circ\in\n$ by \cite[Proposition B.3.5(2)]{LTXZZ} and \Cref{P4}.

Hence $\im(\nabla_\m\circ\Delta^r_\m)\subseteq\n\cdot\O_\lambda[\Sh(V^\circ,K^\circ)]_\m\times T_\p(\overline{\F}_p),$ and we may consider
\begin{equation*}
    \nabla_{/\n}\colon\mathrm{coker}\ \Delta^r_\m/\n\to\O_\lambda[\Sh(V^\circ,K^\circ)]/\n\times T_\p(\overline{\F}_p).
\end{equation*}

Recalling that $B^r(\mathcal{X},\O_\lambda)_\m=\ker d_{1,\m}^{0,2r}=\im d_{1,\m}^{-1,2r},$ \cite[Proposition 6.3.1(3)]{LTXZZ} implies that $\nabla_{/\n}$ is surjective.

To conclude, it remains to prove that the source and target have the same length as $\O_\lambda/\lambda^m$-modules. This follows from \Cref{LiuCohomology} and (the proof of) \cite[Proposition 6.4.1]{LTXZZ}, together with the fact that
\begin{equation*}
    H^1_{\sing}(F_\p,H^{2r-1}_\et(\overline{M}_\p,R\Psi\O_\lambda(r))_{\m})/\n\iso H^1_{\sing}(F_\p,H^{2r-1}_\et(\overline{M}_\p,R\Psi\O_\lambda(r))/\n),
\end{equation*}
which follows by \Cref{Cohomologymodp}(4) in the same way as in the proof of \cite[Theorem 6.3.5]{LTXZZ}.
\end{proof}

Putting it all together, by \Cref{LiuCohomology} and \Cref{sing-potential}, we have an isomorphism
\begin{equation*}
    \Theta^\eta\circ\eta^r_\m\colon\coker\Delta^r_\m\rightiso H^1_{\sing}(F_\p,H^{2r-1}_\et(\Sh(V,K)_{\overline{F}},\O_\lambda(r))_\m\otimes_{\O_\lambda[\Gamma_{F_\p}]} H^0_\et(\mathcal{T}^\eta_{\p,\overline{\Q}_p},\O_\lambda))
\end{equation*}
and noting that $\Theta^\eta,\eta^r_\m$ and $\nabla_{/\n}$ are all $\mathfrak{T}$-invariant, we arrive at the following isomorphism.
\begin{definition}\label{ALRmap}
    We define the \emph{arithmetic level raising map}
    \begin{equation*}
        \varphi_{\lr}\colon H^1_{\sing}(F_\p,H^{2r-1}_\et(\Sh(V,K)_{\overline{F}},\O_\lambda(r))/\n)\rightiso \O_\lambda[\Sh(V^\circ,K^\circ)]/\n
    \end{equation*}
    to be induced by the composition $\nabla_{/\n}\circ(\Theta^\eta\circ\eta^r_\m)^{-1}.$
\end{definition}

\subsection{First reciprocity law}\label{Reciprocity-Reciprocity}
\subsubsection{Set-up}
We keep the data in \Cref{ALRSetup} in order to apply the arithmetic level raising. We also choose
\begin{itemize}
    \item An orthogonal decomposition $V^\circ=V_1^\circ\oplus V_2^\circ$ into rank $r$ hermitian spaces.
    \item An orthogonal decomposition $V=V_1\oplus V_2$ into rank $r$ hermitian spaces.
\end{itemize}
We follow the notation of \Cref{Shimura-Kudla}, where we denote
\begin{equation*}
    G=U(V),\quad H_1=U(V_1)\quad H_2=U(V_2),\quad H=H_1\times H_2
\end{equation*}
as well as analogous notation $G^\circ,H_1^\circ,H_2^\circ,H^\circ.$

We assume\footnote{This is not strictly necessary, but we impose this condition in order to simplify some notation.}
\begin{enumerate}[label=(FJ\arabic*)]
    \item\label{FJ1} The isometry $j\colon V\otimes_{F^+}\A_{F^+}^{\infty,\p}\rightiso V^\circ\otimes_{F^+}\A^{\infty,\p}$ respects the orthogonal decompositions above, and extends to a global isometry $j\colon V_2\rightiso V_2^\circ.$
\end{enumerate}

\subsubsection{Computation of the image of Friedberg--Jacquet cycles}
Recall that $H^{i}_\et(\mathrm{Sh}(V,K)_{\overline{F}},\O_\lambda)_\m=0$ for $i\neq 2r-1$ by \Cref{VanishingCohomology}. Thus, we have the Abel--Jacobi map
\begin{equation*}
    \mathrm{AJ}_\m\colon \mathrm{Z}^r(\mathrm{Sh}(V,K))_\m\to H^1(F,H^{2r-1}_\et(\mathrm{Sh}(V,K)_{\overline{F}},\O_\lambda(r))_\m).
\end{equation*}

The following auxiliary data is not needed for the statement of the main result of this section (\Cref{reciprocitycomputation}), but will be used throughout its proof. We choose
\begin{itemize}
    \item A $p$-adic uniformization datum $\mathbbm{z}.$
    \item An $F$-basis $\underline{x}\in(V_2)^r$ of $V_2.$
    \item A collection of vectors $\underline{\mathbbm{x}}\in\mathbb{V}^r$ satisfying $T(\underline{x})=T(\underline{\mathbbm{x}}).$ We denote $\mathbb{V}_2$ the span of $\underline{\mathbbm{x}},$ and $\mathbb{G},\mathbb{H}_1,\mathbb{H}_2,\mathbb{H}$ notation analogous to above.
\end{itemize}
As preparation for the explicit computation of the arithmetic level raising of Friedberg--Jacquet cycles, we collect many of the results of the previous sections in the following \Cref{Figure1,Figure2}. There, we also define maps $\nabla^{?,\mathrm{FJ}},$ $\bbnabla^?$ and $\bbnabla^{?,\mathrm{FJ}}$ for $?\in\{\circ,\bullet\}$ which we will refer to later. For ease of notation, we denote
\begin{equation*}
    H^1_{\sing}(X)\defeq H^1_{\sing}(F_\p,H^{2r-1}_\et(X_{\overline{F}},\O_\lambda(r))/\n)
\end{equation*}
and
\begin{equation*}
    \Sh^\circ\defeq\Sh(V^\circ,K^\circ),\quad\Sh^\bullet\defeq\Sh(V^\circ,K^\bullet)
\end{equation*}
in such diagrams.

\begin{landscape}
\begin{figure}
\caption{Uniformization on the generic fiber and arithmetic level raising.}\label{Figure1}
\centering
\begin{tikzcd}
    \CYCFJ{\underline{x}}{K}\arrow{d}\Symbol{rrd}{\text{\Cref{FJcomplex}}} &[-30pt]&[-25pt] \arrow[ll,"z^\eta"'] \CycFJ{\underline{x}}{K^p}\arrow{d}\Symbol{rd}{\text{\Cref{ZZFJDef}}}&[-50pt]|[zshift=-25pt]|(\CycFJ{\underline{x}}{K^p})'\arrow[l]\arrow[ddddddddr, dashed, in=20,out=-30, looseness=1, "{\nabla^{\circ,\mathrm{FJ}},\nabla^{\bullet,\mathrm{FJ}}}", pos = 0.3]\arrow[d]&[-30pt]\\
    \CYC{\underline{x}}{K}\arrow{d}{\O_{Z(\blank)}}\Symbol{rrd}{\text{\Cref{KRGenericFiber}}} && \arrow[ll,"z^\eta"'] \Cyc{\underline{x}}{K^p}&|[zshift=-25pt]|\Cyc{\underline{x}}{K^p}'\arrow[l]\arrow[dl,"{}^\L\ZZ"]&\\
    \mathrm{Gr}^rK_0(\Sh(V,K))_\m\arrow[r,hook,"\Theta^{\eta,-1}"]\arrow{d}{\delta_\p\mathrm{loc}_\p\mathrm{AJ}_\m}&\mathrm{Gr}^rK_0(\mathcal{M}_\p^\eta)_\m\Symbol{rdd}{\text{\Cref{KtheoryAJ}}}\arrow[d,"\delta_\p\mathrm{loc}_\p\mathrm{AJ}_\m"']&\arrow{l}\mathrm{Gr}^rK_0(\mathcal{M}_\p)\arrow[dddrr, bend left]\arrow{dd}{\widetilde{(\blank)}}&&\\
    H^1_{\sing}(\Sh(V,K))\arrow[dd,"\varphi_\lr","\sim"']\arrow[r,hook,"\Theta^{\eta,-1}"]\Symbol{rdd}{\text{\Cref{ALRmap}}}&H^1_{\sing}(\mathcal{M}_\p^\eta)&&&\\
    &\coker\Delta^r_\m/\n\arrow[d,"\nabla_{/\n}"',"\sim"]\arrow[u,"\eta^r_\m","\sim"']&\arrow{l}B^r(\mathcal{X},\O_\lambda)_\m\arrow[d,hook]&&\\
    \O_\lambda[\Sh^\circ]/\n\arrow[r,hook]&\O_\lambda[\Sh^\circ]/\n\times T_\p(\overline{\F}_p)&\arrow[l,"\nabla"]H^{2r}_\et(\overline{M}_\p^\circ)_\m\oplus H^{2r}_\et(\overline{M}_\p^\bullet)_\m\arrow[d,"i^{\circ,*}\oplus i^{\bullet,*}"]\arrow[dddl,bend right=20,"\inc{\circ}\oplus\inc{\bullet}"']&&\arrow[ll,"\mathrm{cl}"]\mathrm{Gr}^rK_0(\overline{M}_\p^\circ)\oplus \mathrm{Gr}^rK_0(\overline{M}_\p^\bullet)\arrow[d,"i^{\circ,*}\oplus i^{\bullet,*}"]\\
    &&H^{2r}_\et(\overline{B}_\p^\circ)_\m\oplus H^{2r}_\et(\overline{B}_\p^\bullet)_\m\arrow[d,"\cup\xi^{r-1}\oplus\mathrm{id}"]&&\arrow[ll,"\mathrm{cl}"]\mathrm{Gr}^rK_0(\overline{B}_\p^\circ)\oplus \mathrm{Gr}^rK_0(\overline{B}_\p^\bullet)\arrow[d,"\cup\xi^{r-1}\oplus\mathrm{id}"]\\
    &&H^{2(2r-1)}_\et(\overline{B}_\p^\circ)_\m\oplus H^{2r}_\et(\overline{B}_\p^\bullet)_\m\arrow[d,"\pi^\circ_*\oplus\pi^\bullet_*"]&&\arrow[ll,"\mathrm{cl}"]\mathrm{Gr}^{2r-1}K_0(\overline{B}_\p^\circ)\oplus \mathrm{Gr}^rK_0(\overline{B}_\p^\bullet)\arrow[d,"\pi^\circ_*\oplus\pi^\bullet_*"]\\
    &\arrow{uuu}{(I^\circ,(p+1)T^{\circ\bullet})}\left(\O_\lambda[\Sh^\circ]_\m\oplus \O_\lambda[\Sh^\bullet]_\m\right)\times T_\p(\overline{\F}_p)&\arrow[l,"\Theta^b","\sim"']H^{0}_\et(\overline{S}_\p^\circ)_\m\oplus H^{0}_\et(\overline{S}_\p^\bullet)_\m&&\arrow[ll,"\mathrm{cl}"]K_0(\overline{S}_\p^\circ)\oplus K_0(\overline{S}_\p^\bullet)
\end{tikzcd}
\end{figure}
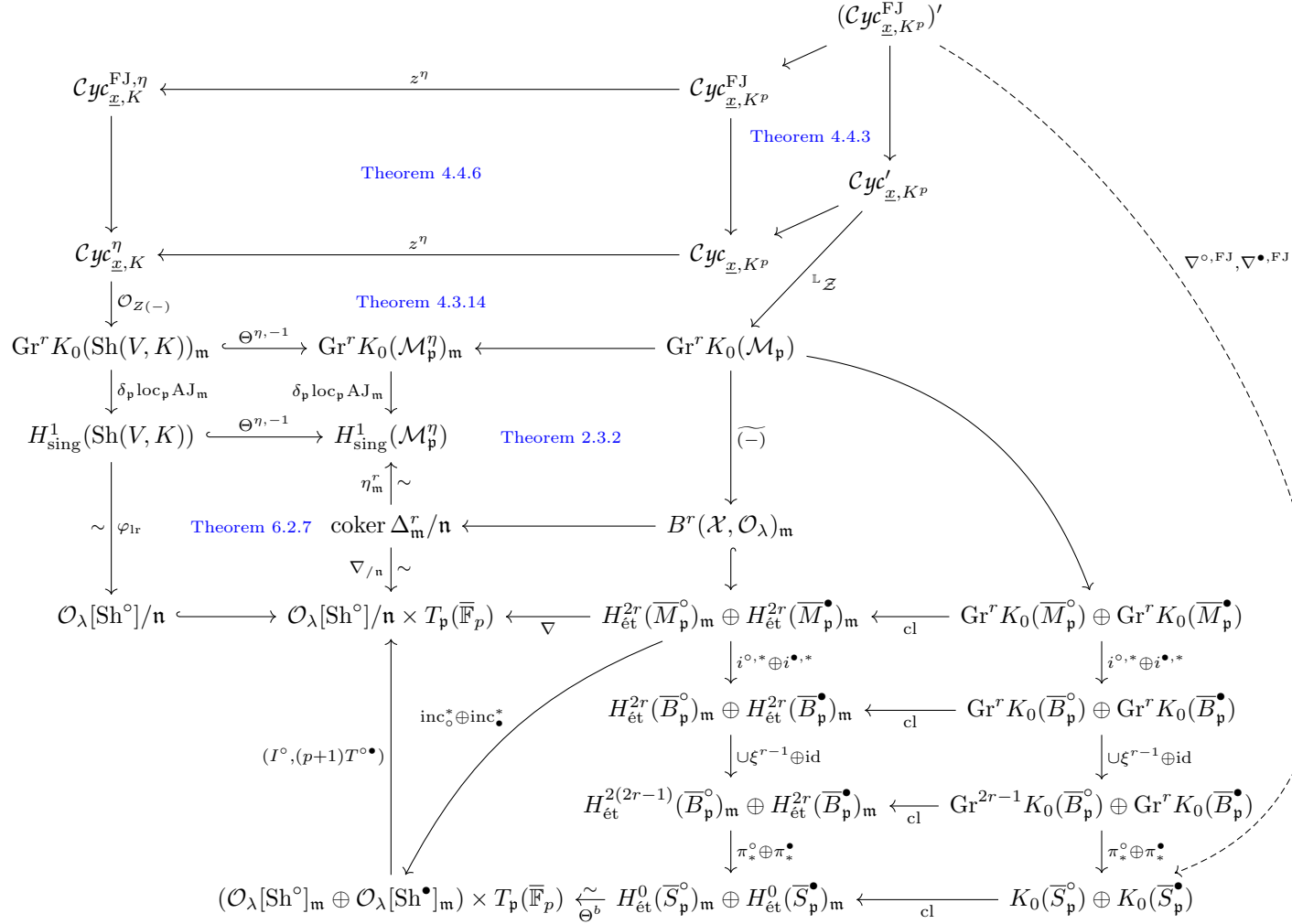
\end{landscape}
\begin{landscape}
\begin{figure}
\caption{$p$-adic uniformization.}\label{Figure2}
\centering
\begin{tikzcd}[ampersand replacement=\&]
    \CycFJ{\underline{x}}{K^p}\Symbol{rrrd}{\text{\Cref{FJpadic}}}\arrow{d}\&[-50pt]|[zshift=-25pt]|(\CycFJ{\underline{x}}{K^p})'\arrow[l]\arrow[d]\arrow[from = 1-1, to = 1-4, crossing over]\&[-40pt]\&[-40pt]\cycFJ{\underline{\mathbbm{x}}}{\bbeta^p_*(K^p)}\arrow[d]\arrow[dddddd, dashed, in=10, out=0, looseness=1.5, "{\bbnabla^{\circ,\mathrm{FJ}},\bbnabla^{\bullet,\mathrm{FJ}}}", pos = 0.4]\\
    \Cyc{\underline{x}}{K^p}\Symbol{rrrd}{\text{\Cref{KR-padic}}}\&|[zshift=-25pt]|\Cyc{\underline{x}}{K^p}'\arrow[l]\arrow[ld, "{}^{\L}\ZZ"]\arrow[from = 2-1, to = 2-4, crossing over]\&\&\cyc{\underline{\mathbbm{x}}}{\bbeta^p_*(K^p)}\arrow{d}{{}^{\L}\ZZ_-\times\mathcal{T}_\p}\arrow[ddddd, dashed, in=10, out=0, looseness=1.5, "{\substack{\bbnabla^\circ,\bbnabla^\bullet\\\text{\Cref{nablacirc}}\\\text{\Cref{nablabullet}}}}"', pos=0.4]\\
    \mathrm{Gr}^rK_0(\mathcal{M}_\p)\arrow{d}\arrow{rr}\Symbol{ddrrr}{\text{\Cref{padic-basic}}}\&\&\mathrm{Gr}^rK_0(\mathcal{M}_\p^{\wedge,ss})\arrow[r,"\Theta_{\mathbbm{z}}","\sim"']\&\mathrm{Gr}^rK_0\left(\mathbb{G}(F^+)\backslash\left(\mathcal{N}\times \mathbb{G}(\A_{F^+}^{p,\infty})/\bbeta^p_*(K^p)\right)\times\mathcal{T}_\p\right)\arrow{dd}\\
    \mathrm{Gr}^rK_0(\overline{M}_\p^\circ)\oplus \mathrm{Gr}^rK_0(\overline{M}_\p^\bullet)\arrow[d,"i^{\circ,*}\oplus i^{\bullet,*}"]\&\&\&\\
    \mathrm{Gr}^rK_0(\overline{B}_\p^\circ)\oplus \mathrm{Gr}^rK_0(\overline{B}_\p^\bullet)\arrow[d,"\cup\xi^{r-1}\oplus\mathrm{id}"]\arrow[rrr,"\Theta_{\mathbbm{z}}","\sim"']\&\&\&\begin{tabular}{r@{\hspace{3pt}}l}&$\mathrm{Gr}^rK_0\left(\mathbb{G}(F^+)\backslash\left(\bigsqcup_{\Lambda^\circ}\V(\Lambda^\circ)\times \mathbb{G}(\A_{F^+}^{p,\infty})/\bbeta^p_*(K^p)\right)\times\mathcal{T}_\p\right)$\\$\oplus$&$\mathrm{Gr}^rK_0\left(\mathbb{G}(F^+)\backslash\left(\bigsqcup_{\Lambda^\bullet}\V(\Lambda^\bullet)\times \mathbb{G}(\A_{F^+}^{p,\infty})/\bbeta^p_*(K^p)\right)\times\mathcal{T}_\p\right)$\end{tabular}\arrow[d,"\cup\xi^{r-1}\oplus\mathrm{id}"]\\
    \mathrm{Gr}^{2r-1}K_0(\overline{B}_\p^\circ)\oplus \mathrm{Gr}^rK_0(\overline{B}_\p^\bullet)\arrow[d,"\pi^\circ_*\oplus\pi^\bullet_*"]\arrow[rrr,"\Theta_{\mathbbm{z}}","\sim"']\Symbol{drrr}{\text{\Cref{padic-basic}}}\&\&\&\begin{tabular}{r@{\hspace{3pt}}l}&$\mathrm{Gr}^{2r-1}K_0\left(\mathbb{G}(F^+)\backslash\left(\bigsqcup_{\Lambda^\circ}\V(\Lambda^\circ)\times \mathbb{G}(\A_{F^+}^{p,\infty})/\bbeta^p_*(K^p)\right)\times\mathcal{T}_\p\right)$\\$\oplus$&$\mathrm{Gr}^rK_0\left(\mathbb{G}(F^+)\backslash\left(\bigsqcup_{\Lambda^\bullet}\V(\Lambda^\bullet)\times \mathbb{G}(\A_{F^+}^{p,\infty})/\bbeta^p_*(K^p)\right)\times\mathcal{T}_\p\right)$\end{tabular}\arrow[d,"\chi(\blank)\oplus\chi(\blank)"]\\
    K_0(\overline{S}_\p^\circ)\oplus K_0(\overline{S}_\p^\bullet)\arrow[rrr,"\Theta_{(\mathbb{V},\bbeta^p)}","\sim"']\&\&\&\begin{tabular}{r@{\hspace{3pt}}l}&$\Z\left[\mathbb{G}(F^+)\backslash\left(\mathrm{Vert}^\circ(\mathbb{V}\otimes_\Q\Q_p)\times \mathbb{G}(\A_{F^+}^{p,\infty})/\bbeta^p_*(K^p)\right)\times\mathcal{T}_\p\right]$\\$\oplus$&$\Z\left[\mathbb{G}(F^+)\backslash\left(\mathrm{Vert}^\bullet(\mathbb{V}\otimes_\Q\Q_p)\times \mathbb{G}(\A_{F^+}^{p,\infty})/\bbeta^p_*(K^p)\right)\times\mathcal{T}_\p\right]$\end{tabular}
\end{tikzcd}
\end{figure}
\end{landscape}

\begin{definition}
    Consider $\q\mid p$ a place of $F^+.$ For $W$ a nondegenerate $F_\q/F^+_\q$-hermitian space of rank $r,$ and $e,f\in\Typ^0_r,$ we consider the following quantities, where $\Lambda$ is a choice of lattice $\Lambda\subseteq W$ with $\mathrm{typ}(\Lambda)=f$ (\Cref{Typ}). If such a lattice $\Lambda$ does not exist, we take these to be identically zero.
    \begin{equation*}
    \begin{split}
        d_e^W(f)&\defeq\#\{L\subseteq \Lambda\colon \mathrm{typ}(L)=e\},\\
        c_e^W(f)&\defeq\sum_{\substack{L\subseteq\Lambda\\\mathrm{typ}(L)=e}}c([\Lambda\colon\varpi\Lambda+L])=\sum_{\substack{L\subseteq\Lambda\\\mathrm{typ}(L)=e}}\prod_{i=1}^{[\Lambda\colon\varpi\Lambda+L]}(1-\lVert\q\rVert^{-2i}).
    \end{split}
    \end{equation*}
\end{definition}
\begin{definition}\label{phiDef}
For $\q\mid p$ and $e\in\Typ_r^{0,\flat},$ we consider the distribution
\begin{equation*}
    \phi_e\in\Z[H(F^+_\q)\backslash G(F^+_\q)/K_\q]
\end{equation*}
given by
\begin{equation*}
    \phi_e(g_\q)\defeq d_e^{V_{2,\q}}(\mathrm{typ}(g_\q\Lambda_\q\cap V_{2,\q})).
\end{equation*}
For $\q=\p,$ we also consider the distributions
\begin{equation*}
    \phi_e^\bullet\in\Z[H^\circ(F^+_\p)\backslash G^\circ(F^+_\p)/K^\bullet_\p],\quad \phi_e^\circ\in\Z[H^\circ(F^+_\p)\backslash G^\circ(F^+_\p)/K_\p^\circ]
\end{equation*}
given by
\begin{equation*}
    \phi_e^\bullet(g^\circ_\p)\defeq c_e^{V_{2,\p}^\circ}(\mathrm{typ}(g_\p\Lambda^\bullet_\p\cap V_{2,\p}^\circ)),\quad\phi_e^\circ\defeq T^{\circ\bullet}(\phi_e^\bullet).
\end{equation*}
\end{definition}

For a set $S,$ we denote $\mathbbm{1}[S]$ the indicator function of $S.$
\begin{definition}
We consider the unitary Friedberg--Jacquet distributions
\begin{equation*}
    \mathcal{P}^{\mathrm{FJ}}\colon\Z[H^\circ(\A_{F^+}^{\infty})\backslash G^\circ(\A_{F^+}^{\infty})/K^\circ]\to\Z[\mathrm{Sh}(V^\circ,K^\circ)]
\end{equation*}
characterized by the following. For $f\in\Z[\Sh(V^\circ,K^\circ)]$ we have
\begin{equation*}
    \langle \mathcal{P}^{\mathrm{FJ}}(\mathbbm{1}[H^\circ(\A_{F^+}^{\infty})x K^\circ]),f\rangle=\frac{1}{\mathrm{vol}_{H^\circ}(H^\circ(\A_{F^+}^{\infty})\cap xK^\circ x^{-1})}\int_{H^\circ(F^+)\backslash H^\circ(\A_{F^+}^{\infty})}f(hx)\d{h}.
\end{equation*}
\end{definition}
\begin{remark}\label{PFJRemark}
Note that
\begin{equation*}
\begin{split}
    &\frac{1}{\mathrm{vol}_{G^\circ}(K^\circ)}\int_{H^\circ(F^+)\backslash G^\circ(\A_{F^+}^{\infty})}\mathbbm{1}[H^\circ(\A_{F^+}^{\infty})xK^\circ](g)\cdot f(g)\\
    &\qquad=\frac{1}{\mathrm{vol}_{H^\circ}(H^\circ(\A_{F^+}^{\infty})\cap xK^\circ x^{-1})}\int_{H^\circ(F^+)\backslash H^\circ(\A_{F^+}^{\infty})}\frac{1}{\mathrm{vol}_{G^\circ}(K^\circ)}\int_{K^\circ} f(hxk)\d{k}\d{h}\\
    &\qquad=\frac{1}{\mathrm{vol}_{H^\circ}(H^\circ(\A_{F^+}^{\infty})\cap xK^\circ x^{-1})}\int_{H^\circ(F^+)\backslash H^\circ(\A_{F^+}^{\infty})}f(hx)\d{h},
    \end{split}
\end{equation*}
and thus $\mathcal{P}^{\mathrm{FJ}}$ is also characterized by
\begin{equation*}
    \langle\mathcal{P}^{\mathrm{FJ}}(\phi),f\rangle=\frac{1}{\mathrm{vol}_{G^\circ}(K^\circ)}\int_{H^\circ(F^+)\backslash G^\circ(\A_{F^+}^{\infty})}\phi(g)f(g).
\end{equation*}
\end{remark}

\begin{theorem}\label{reciprocitycomputation}
For any $\phi^\p\in\Z[H(\A_{F^+}^{\infty,\p})\backslash G(\A_{F^+}^{\infty,\p})/K^\p]$ and $e=(e_1\ge\cdots\ge e_r)\in\Typ_r^{0,\flat}$ we have\footnote{If $e\in\Typ_r^0$ but does not belong to $\Typ_r^{0,\flat},$ then both sides are automatically zero. Similarly, if $\sum_{i=1}^r e_i\equiv1\mod 2,$ both sides are also identically zero.}
\begin{equation*}
    \varphi_\lr(\partial_\p\mathrm{loc}_\p\mathrm{AJ}_\m(Z^{\mathrm{FJ}}(\phi_e\otimes \phi^\p)))=\mathcal{P}^{\mathrm{FJ}}(\phi_e^\circ\otimes j_*(\phi^\p)).
\end{equation*}
\end{theorem}
\begin{proof}
First, we see that
\begin{equation*}
    \mathcal{P}^{\mathrm{FJ}}(\mathbbm{1}[H^\circ(\A_{F^+}^{\infty})xK^\circ])=\left(g\mapsto\sum_{\gamma\in H^\circ(F^+)\backslash G^\circ(F^+)}\mathbbm{1}[H^\circ(\A_{F^+}^{\infty})xK^\circ](\gamma g)\right).
\end{equation*}
This is since
\begin{equation*}
\begin{split}
    &\sum_{g\in G^\circ(F^+)\backslash G^\circ(\A_{F^+}^{\infty})/K^\circ}\sum_{\gamma\in H^\circ(F^+)\backslash G^\circ(F^+)}\mathbbm{1}[H^\circ(\A_{F^+}^{\infty})xK^\circ](\gamma g)\cdot f(g)\\
    &\qquad=\frac{1}{\mathrm{vol}_{G^\circ}(K^\circ)}\int_{G^\circ(F^+)\backslash G^\circ(\A_{F^+}^{\infty})}\sum_{\gamma\in H^\circ(F^+)\backslash G^\circ(F^+)}\mathbbm{1}[H^\circ(\A_{F^+}^{\infty})xK^\circ](\gamma g)\cdot f(\gamma g)\d{g}.
\end{split}
\end{equation*}

Now, by the definitions \Cref{RZintballoon,RZintground}, the maps $\bbnabla^\circ,\bbnabla^\bullet$ in \Cref{Figure2} are given by
\begin{equation*}
\begin{split}
    &\bbnabla^?\left(\mathbbm{1}[m_pM(\O_{F^+}\otimes\Z_p),\mathbbm{g}_1^p\mathbb{K}^p]\right)(\Lambda^?,\mathbbm{g}_2^p\mathbb{K}^p,t)\\
    &\qquad=\sum_{\gamma\in\mathbb{G}(F^+)/\mathbb{H}_1(F^+)}\nabla^?_{\gamma_p\mathrm{span}(\underline{\mathbbm{x}}\cdot m_p)}(\Lambda^?)\cdot\mathbbm{1}[\gamma^p\mathbb{H}_1(\A_{F^+}^{p,\infty})\mathbbm{g}_1^p\mathbb{K}^p](\mathbbm{g}_2^p).
\end{split}
\end{equation*}
Here and throughout, we use the notation $\nabla$ as in \Cref{RZ-Intersection} but for the Rapoport--Zink space $\mathcal{N}=\prod_{\q\mid p}\mathcal{N}_\q,$ that is,
\begin{equation*}
    \nabla^?_{\underline{\mathbbm{x}}}(\Lambda^?)=\begin{cases}
    \nabla^?_{\underline{\mathbbm{x}}_\p}(\Lambda^?_\p)&\text{ if }\underline{\mathbbm{x}}\subseteq\Lambda^?,\\0&\text{ otherwise.}
\end{cases}
\end{equation*}
We are also using the notation as in \cref{nablaNotation}. Finally, for $\gamma\in \mathbb{G}(F^+)$ we will denote $\gamma_p$ resp. $\gamma^p$ for its components above resp. away from $p$ as an element of $\mathbb{G}(\A_{F^+}^{\infty}).$

Thus $\bbnabla^{?,\mathrm{FJ}}\left(\mathbbm{1}[U_{T(\underline{\mathbbm{x}})}(F^+\otimes\Q_p)m_pM(\O_{F^+}\otimes\Z_p),\mathbb{H}(\A_{F^+}^{p,\infty})\mathbbm{g}_1^p\mathbb{K}^p]\right)(\Lambda^?,\mathbbm{g}_2^p\mathbb{K}^p,t)$ is
\begin{equation*}
    \sum_{(n_p,\mathbbm{h}^p)\in \mathbb{H}(F^+)\mathbb{H}_1(\A_{F^+}^{p,\infty})\backslash U_{T(\underline{\mathbbm{x}})}(F^+\otimes\Q_p)\mathbb{H}(\A_{F^+}^{p,\infty})/\mathbb{K}_{\mathbb{H}}}\left(\sum_{\gamma\in\mathbb{G}(F^+)/\mathbb{H}_1(F^+)}\nabla^?_{\gamma_p\cdot\mathrm{span}(\underline{\mathbbm{x}}\cdot n_pm_p)}(\Lambda^?)\cdot\mathbbm{1}[\gamma^p\mathbb{H}_1(\A_{F^+}^{p,\infty})\mathbbm{h}^p\mathbbm{g}_1^p\mathbb{K}^p](\mathbbm{g}_2^p)\right)
\end{equation*}
where 
\begin{equation*}
    \mathbb{K}_{\mathbb{H}}\defeq\left(m_pM(\O_{F^+}\otimes\Z_p)m_p^{-1}\cap U_{T(\underline{\mathbbm{x}})}(F^+\otimes\Q_p)\right)\times\left(\mathbbm{g}_1^p\mathbb{K}^p\mathbbm{g}_1^{p,-1}\cap\mathbb{H}(\A_{F^+}^{p,\infty})\right).
\end{equation*}
This is
\begin{equation*}
    \sum_{(n_p,\mathbbm{h}^p)\in \mathbb{H}_2(F^+)\backslash U_{T(\underline{\mathbbm{x}})}(F^+\otimes\Q_p)\mathbb{H}_2(\A_{F^+}^{p,\infty})/\mathbb{K}_{\mathbb{H}_2}}\left(\sum_{\gamma\in\mathbb{G}(F^+)/\mathbb{H}_1(F^+)}\nabla^?_{\gamma_p\cdot\mathrm{span}(\underline{\mathbbm{x}}\cdot n_pm_p)}(\Lambda^?)\cdot\mathbbm{1}[\gamma^p\mathbb{H}_1(\A_{F^+}^{p,\infty})\mathbbm{h}^p\mathbbm{g}_1^p\mathbb{K}^p](\mathbbm{g}_2^p)\right)
\end{equation*}
where
\begin{equation*}
    \mathbb{K}_{\mathbb{H}_2}\defeq\left(m_pM(\O_{F^+}\otimes\Z_p)m_p^{-1}\cap U_{T(\underline{\mathbbm{x}})}(F^+\otimes\Q_p)\right)\times\mathrm{pr}_{\mathbb{H}_2}\left(\mathbbm{g}_1^p\mathbb{K}^p\mathbbm{g}_1^{p,-1}\cap\mathbb{H}(\A_{F^+}^{p,\infty})\right).
\end{equation*}
Expanding $\sum_{\gamma\in\mathbb{G}(F^+)/\mathbb{H}_1(F^+)}F(\gamma)=\sum_{\gamma\in\mathbb{H}(F^+)\backslash \mathbb{G}(F^+)}\left(\sum_{\delta\in\mathbb{H}_2(F^+)}F(\gamma^{-1}\delta)\right),$ this is
\begin{equation*}
    \sum_{\gamma\in\mathbb{H}(F^+)\backslash\mathbb{G}(F^+)}\left(\sum_{(n_p,\mathbbm{h}^p)\in U_{T(\underline{\mathbbm{x}})}(F^+\otimes\Q_p)\mathbb{H}_2(\A_{F^+}^{p,\infty})/\mathbb{K}_{\mathbb{H}_2}}\Delta^?_{\mathrm{span}(\underline{\mathbbm{x}}\cdot n_pm_p)}(\gamma_p\Lambda^?)\cdot\mathbbm{1}[\mathbb{H}_1(\A_{F^+}^{p,\infty})\mathbbm{h}^p\mathbbm{g}_1^p\mathbb{K}^p](\gamma^p\mathbbm{g}_2^p)\right).
\end{equation*}
Note that the inner sum factorizes, so this becomes
\begin{equation*}
    \sum_{\gamma\in\mathbb{H}(F^+)\backslash\mathbb{G}(F^+)}\left(S_p^?(m_p,\gamma_p\Lambda^?)\cdot S^p(\mathbbm{g}_1^p,\gamma^p\mathbbm{g}_2^p)\right)
\end{equation*}
where the term away from $p$ is
\begin{equation*}
\begin{split}
    &S^p(\mathbbm{g}_1^p,\mathbbm{g}_2^p)=\sum_{\mathbbm{h}^p\in \mathbb{H}_2(\A_{F^+}^{p,\infty})/\mathbb{K}_{\mathbb{H}_2}^p}\mathbbm{1}[\mathbb{H}_1(\A_{F^+}^{p,\infty})\mathbbm{h}^p\mathbbm{g}_1^p\mathbb{K}^p](\mathbbm{g}_2^p)\\
    &\quad= \frac{1}{\mathrm{vol}(\mathbb{K}_{\mathbb{H}_2}^p)}\int_{\mathbb{H}_2(\A_{F^+}^{p,\infty})}\frac{1}{\mathrm{vol}(\mathbb{H}_1(\A_{F^+}^{p,\infty})\cap h_2^p\mathbbm{g}_1^p\mathbb{K}^p(h_2^p\mathbbm{g}_1^p)^{-1})}\int_{\mathbb{H}_1(\A_{F^+}^{p,\infty})}\mathbbm{1}[h_1^ph_2^p\mathbbm{g}_1^p\mathbb{K}^p](\mathbbm{g}_2^p)\d{h_1^p}\d{h_2^p}\\
    &\quad= \frac{1}{\mathrm{vol}(\mathbb{K}_{\mathbb{H}}^p)}\int_{\mathbb{H}(\A_{F^+}^{p,\infty})}\mathbbm{1}[h^p\mathbbm{g}_1^p\mathbb{K}^p](\mathbbm{g}_2^p)\d{h^p}=\mathbbm{1}[\mathbb{H}(\A_{F^+}^{p,\infty})\mathbbm{g}_1^p\mathbb{K}^p](\mathbbm{g}_2^p)
\end{split}
\end{equation*}
and the term at $p$ is
\begin{equation*}
    S_p^?(m_p,\Lambda^?)=\sum_{n_p\in U_{T(\underline{\mathbbm{x}})}(F^+\otimes\Q_p)m_pM(\O_{F^+}\otimes\Z_p)/M(\O_{F^+}\otimes\Z_p)}\Delta_{\mathrm{span}(\underline{\mathbbm{x}}\cdot n_p)}^?(\Lambda^?)=\sum_{\substack{L\in\mathrm{Lat}^r(\mathrm{span}(\underline{\mathbbm{x}})\otimes\Q_p)\\L\sim\mathrm{span}(\underline{\mathbbm{x}}\cdot m_p)}}\nabla_{L}^?(\Lambda^?).
\end{equation*}

Denote $\mathbb{V}_{2,p}\defeq\mathrm{span}(\underline{\mathbbm{x}})\otimes\Q_p.$ If we let
\begin{equation*}
    \bbnabla^{\mathrm{FJ}}\defeq I^\circ\circ\bbnabla^{\circ,\mathrm{FJ}}+(p+1)T^{\circ\bullet}\circ\bbnabla^{\bullet,\mathrm{FJ}},
\end{equation*}
then the above together with \Cref{n0KR} and \Cref{RZnabla} further give us that
\begin{equation*}
\begin{split}
    &\bbnabla^{\mathrm{FJ}}\left(\mathbbm{1}[U_{T(\underline{\mathbbm{x}})}(F^+\otimes\Q_p)m_pM(\O_{F^+}\otimes\Z_p),\mathbb{H}(\A_{F^+}^{p,\infty})\mathbbm{g}_1^p\mathbb{K}^p]\right)(\Lambda^?,\mathbbm{g}_2^p\mathbb{K}^p,t)\\
    &\qquad=\sum_{\gamma\in\mathbb{H}(F^+)\backslash\mathbb{G}(F^+)}S_p(m_p,\gamma_p\Lambda^\circ)\cdot\mathbbm{1}[\mathbb{H}(\A_{F^+}^{p,\infty})\mathbbm{g}_1^p\mathbb{K}^p](\mathbbm{g}_2^p)
\end{split}
\end{equation*}
where $S_p(m_p,\blank)$ is zero unless the lattice $\mathrm{span}(\underline{\mathbbm{x}}\cdot m_p)\subseteq\mathbb{V}_{2,p}$ is integral, in which case it is given by
\begin{equation*}
    S_p(m_p,\blank)=T^{\circ\bullet}\left(\left(\varpi\Lambda_\p^\bullet,(\Lambda^\circ_\q)_{\q\neq\p}\right)\mapsto \sum_{\substack{L\in\mathrm{Lat}^r(\mathbb{V}_{2,p})\\L\sim\mathrm{span}(\underline{\mathbbm{x}}\cdot m_p)\\L_\p\subseteq\Lambda^\bullet_\p,\ L_\q\subseteq\Lambda^\circ_\q\text{ for }\q\neq\p}}c(r-[\varpi\Lambda^\bullet_\p+L_\p\colon\varpi\Lambda^\bullet_\p])\right)
\end{equation*}
Noting that
\begin{equation*}
    r-[\varpi\Lambda^\bullet_\p+L_\p\colon\varpi\Lambda^\bullet_\p]=[\Lambda^\bullet\cap\mathbb{V}_{2,\p}:L_\p+\varpi\Lambda^\bullet\cap\mathbb{V}_{2,\p}],
\end{equation*}
we may rewrite this as
\begin{equation*}
    S_p(m_p,\blank)=\bbphi_{e_\p}^\circ\times\prod_{\q\neq\p}\bbphi_{e_\q}
\end{equation*}
where $(e_\q)_{\q\mid p}=\mathrm{typ}(\mathrm{span}(\underline{\mathbbm{x}}\cdot m_p))$ and $\bbphi^\circ,\bbphi$ are defined similarly to \Cref{phiDef}, but for the space $\mathbb{V}$ instead.

Now let $\nabla^{\mathrm{FJ}}\defeq I^\circ\circ\nabla^{\circ,\mathrm{FJ}}+(p+1)T^{\circ\bullet}\circ\nabla^{\bullet,\mathrm{FJ}}.$ By \Cref{Figure2}, we have
\begin{equation*}
\begin{tikzcd}
(\CycFJ{\underline{x}}{K^p})'\arrow[r]\arrow[d,"\nabla^{\mathrm{FJ}}"]&\CycFJ{\underline{x}}{K^p}\arrow[r,"\sim"]&\cyc{\underline{\mathbbm{x}}}{\bbeta^p_*(K^p)}\arrow[d,"\bbnabla^{\mathrm{FJ}}"]\\
K_0(\overline{S}_\p^\circ)\arrow[rr,"\Theta_{(\mathbb{V},\bbeta^p)}","\sim"']&&\Z[\mathbb{G}(F^+)\backslash\mathrm{Vert}^\circ(\mathbb{V}\otimes_\Q\Q_p)\times \mathbb{G}(\A_{F^+}^{p,\infty})/\bbeta^p_*(K^p)\times\mathcal{T}_\p]
\end{tikzcd}
\end{equation*}
Applying $\Theta_{(V^\circ,j)}$ instead, and in light of \Cref{FJ1}, we get for $m\in M(F^+)$ that if $\mathrm{span}(\underline{x}\cdot m_p)\subseteq V_2\otimes_\Q\Q_p$ has type $(e_\q)_{\q\mid p}$ and is integral,\footnote{As remarked before, if it is not integral, then the term $S_p(m_p,\blank)$ above is automatically zero.} then
\begin{equation*}
    (\Theta_{(V^\circ,j)}\circ\nabla^{\mathrm{FJ}})(\mathbbm{1}[U_{T(\underline{x})}(F^+)m]\times\phi^p)=\mathcal{P}^{\mathrm{FJ}}\left(\phi^\circ_{e_\p}\times\prod_{\q\neq\p}j_*(\phi_{e_\q})\times j_*(\phi^p)\right).
\end{equation*}
Since $\CycFJ{\underline{x}}{K^p}\to\CYCFJ{\underline{x}}{K^p}$ is simply given by \Cref{CycFJ-CYCFJDef}, we have
\begin{equation*}
    (\mathbbm{1}[U_{T(\underline{x})}(F^+\otimes\Q_p)m_pM(\O_{F^+}\otimes\Z_p)]\otimes\phi^p)\mapsto((\phi_{e_\q})_{\q\mid p}\otimes\phi^p).
\end{equation*}
We conclude by \Cref{Figure1} that
\begin{equation*}
    \varphi_\lr(\partial_\p\mathrm{loc}_\p\mathrm{AJ}_\m(Z^{\mathrm{FJ}}(\bigotimes_{\q\mid p}\phi_{e_\q}\otimes \phi^p)))=\mathcal{P}^{\mathrm{FJ}}(\phi_{e_\p}^\circ\otimes\bigotimes_{\q\neq\p}j_*(\phi_{e_\q})\otimes j_*(\phi^p)).
\end{equation*}

As we vary $m\in M(F^+),$ the collection of $\typ(\mathrm{span}(\underline{x}\cdot m_p))$ covers all of $\prod_{\q\mid p}\mathrm{Typ}(V_{2,\q}).$ Now the theorem follows by an inclusion-exclusion argument for the components $\phi_\q$ with $\q\neq\p,$ by using \Cref{TypFJ}.
\end{proof}

\subsubsection{An application of the local computation}
We consider the local spherical Hecke algebra
\begin{equation*}
    \mathbb{T}^\circ\defeq\Z[K^\circ_\p\backslash G^\circ(F^+_\p)/K^\circ_\p].
\end{equation*}
We consider its action on both $\Z[\Sh(V^\circ,K^\circ)]$ and on $\Z[H^\circ(\A_{F^+}^{\infty})\backslash G^\circ(\A_{F^+}^{\infty})/K^\circ]$ via convolution on the right, which we denote by $T*(\blank).$
\begin{proposition}\label{PFJequivariant}
    For any $\phi\in\Z[H^\circ(\A_{F^+}^{\infty})\backslash G^\circ(\A_{F^+}^{\infty})/K^\circ]$ and $T\in\mathbb{T}^\circ,$ we have
    \begin{equation*}
        T*\mathcal{P}^{\mathrm{FJ}}(\phi)=\mathcal{P}^{\mathrm{FJ}}(T*\phi).
    \end{equation*}
\end{proposition}
\begin{proof}
    For $T\in\mathbb{T}^\circ,$ we may consider its adjoint $T'=\left(h\mapsto T(h^{-1})\right).$ Then using \Cref{PFJRemark} we can compute
    \begin{equation*}
        \langle\mathcal{P}^{\mathrm{FJ}}(\phi),T'*f\rangle=\int_{H^\circ(F^+)\backslash G^\circ(\A_{F^+}^{\infty})}\phi(g)\int_{G^\circ(\A_{F^+}^{\infty})}f(gh)T'(h^{-1})\d{h}\d{g}
    \end{equation*}
    is simply
    \begin{equation*}
        \int_{H^\circ(F^+)\backslash G^\circ(\A_{F^+}^{\infty})}\phi(g)\int_{G^\circ(\A_{F^+}^{\infty})}f(gh)T(h)\d{h}\d{g}=\int_{H^\circ(F^+)\backslash G^\circ(\A_{F^+}^{\infty})}f(g)\int_{G^\circ(\A_{F^+}^{\infty})}\phi(gh^{-1})T(h)\d{h}\d{g}
    \end{equation*}
    which is
    \begin{equation*}
        \langle\mathcal{P}^{\mathrm{FJ}}(T*\phi),f\rangle.
    \end{equation*}
    Similarly, one can check that
    \begin{equation*}
        \langle T*\mathcal{P}^{\mathrm{FJ}}(\phi),f\rangle=\langle\mathcal{P}^{\mathrm{FJ}}(\phi),T'*f\rangle,
    \end{equation*}
    and thus the claim follows.
\end{proof}

\begin{theorem}\label{reciprocity}
There exists $\phi_\lr\in\O_\lambda\otimes\CYCFJ{\underline{x}}{K}=\O_\lambda[H(\A_{F^+}^{\infty})\backslash G(\A_{F^+}^{\infty})/K]$ such that
\begin{equation*}
\begin{split}
    &\exp_\lambda\left(\partial_\p\mathrm{loc}_\p\mathrm{AJ}_\m(Z^{\mathrm{FJ}}(\phi_\lr)),\ H^1_\sing(F_\p,H^{2r-1}_\et(\Sh(V,K)_{\overline{F}},\O_\lambda(r))/\n)\right)\\
    &\qquad=\exp_\lambda\left(\mathbbm{1}_{H^\circ},\ \O_\lambda[\Sh(V^\circ,K^\circ)]/\n\right)
\end{split}
\end{equation*}
where $\mathbbm{1}_{H^\circ}\in\O_\lambda[\Sh(V^\circ,K^\circ)]$ is characterized by 
\begin{equation*}
    \langle\mathbbm{1}_{H^\circ},f\rangle=\frac{1}{\mathrm{vol}_{H^\circ}(H^\circ(\A_{F^+}^{\infty})\cap K^\circ)}\int_{H^\circ(F^+)\backslash H^\circ(\A_{F^+}^{\infty})}f(h)\d{h}.
\end{equation*}
\end{theorem}
\begin{proof}
Note that in our previous notation we have $\mathbbm{1}_{H^\circ}=\mathcal{P}^{\mathrm{FJ}}(\mathbbm{1}[H^\circ(\A_{F^+}^{\infty})K^\circ]).$

Under the identification $\O_\lambda[H^\circ(F^+_\p)\backslash G^\circ(F^+_\p)/K^\bullet_\p]\hookrightarrow\O_\lambda[\Typ_r^{0,\flat}]$ of \Cref{Typ}, we have that $\phi_e^\bullet\mapsto\phi^\flat_{e+1}.$

By \Cref{P2}, we have $\ell\nmid p(p+1),$ and thus \Cref{ConjProof1,ConjProof2} give us an expression
\begin{equation*}
    \sum_{e\in I}T_e^\bullet\star\phi^\bullet_e=T'^\bullet(T^{\bullet\circ}\mathbbm{1}[H^\circ(F^+_\p)K^\circ_\p])
\end{equation*}
for some $T_e^\bullet\in\O_\lambda[H^\circ(F_\p^+)\backslash G^\circ(F_\p^+)/K^\bullet_\p]$ and a finite set $I\subseteq\Typ_r^{0,\flat}.$ Applying $T^{\circ\bullet},$ we obtain
\begin{equation}\label{assumptionExpr}
    \sum_{e\in I}T_e^\circ*\phi_e^\circ=(T'^{\circ}\circ I^\circ)(\mathbbm{1}[H^\circ(F^+_\p)K^\circ_\p]),
\end{equation}
for some $T_e^\circ\in\O_\lambda[H^\circ(F_\p^+)\backslash G^\circ(F_\p^+)/K^\circ_\p].$ We define
\begin{equation*}
    (\phi_\lr)_\p\defeq\sum_{e\in I}\phi_\Pi(T_e^\circ)\cdot\phi_e^\circ\in\O_\lambda[H(F^+_\p)\backslash G(F^+_\p)/K_\p]
\end{equation*}
and $\phi_\lr=(\phi_\lr)_\p\otimes \mathbbm{1}[H(\A_{F^+}^{\p,\infty})K^\p].$ By \Cref{reciprocitycomputation}, we have
\begin{equation*}
    \varphi_\lr(\partial_\p\mathrm{loc}_\p\mathrm{AJ}_\m(Z^{\mathrm{FJ}}(\phi_\lr)))=\sum_{e\in I}\phi_\Pi(T_e^\circ)\cdot\mathcal{P}^{\mathrm{FJ}}(\phi_e^\circ\otimes\mathbbm{1}[H^\circ(\A_{F^+}^{\p,\infty})K^{\circ,\p}].
\end{equation*}
By \Cref{P6}, \Cref{PFJequivariant} and \eqref{assumptionExpr}, we have
\begin{equation*}
    \sum_{e\in I}\phi_\Pi(T_e^\circ)\cdot\mathcal{P}^{\mathrm{FJ}}(\phi_e^\circ\otimes\mathbbm{1}[H^\circ(\A_{F^+,f}^\p)K^{\circ,\p}])\equiv \phi_\Pi(T'^\circ)\cdot \phi_\Pi(I^\circ)\cdot\mathcal{P}^{\mathrm{FJ}}(\mathbbm{1}[H^\circ(\A_{F^+}^{\infty})K^\circ])\mod\n.
\end{equation*}
Thus the claim follows since by \Cref{P4} resp. \Cref{P5} we have that $\lambda\nmid\phi_\Pi(T'^\circ)$ resp. $\lambda\nmid\phi_\Pi(I^\circ).$
\end{proof}
\section{Proof of the main theorem}\label{ProofChapter}
Let $F/F^+$ be a CM extension of number fields, where $F\subseteq\C,$ and $N=2r$ an even positive integer. We consider $V^\circ$ a standard definite hermitian $F/F^+$ space of rank $2r$ (\Cref{StandardDef}), with unitary group $G^\circ=U(V^\circ)$ over $F^+.$

We denote by $\pi$ a irreducible cuspidal automorphic representation appearing in
\begin{equation*}
    \varinjlim_{K\in\mathfrak{K}'(V^\circ)}\C[\Sh(V^\circ,K)]
\end{equation*}
where we denote $\mathfrak{K}'(V^\circ)$ the subcategory of $\mathfrak{K}(V^\circ)$ (\Cref{neatDef}) that allows only identity cosets as morphisms. We assume that its base change $\Pi\defeq\mathrm{BC}(\pi)$\footnote{This base change exists by \cite[Corollaire 5.3]{Labesse}.} to $\mathrm{GL}_N(\A_F)$ is irreducible cuspidal. In particular, $\Pi$ is relevant (\Cref{RelevantDefinition}). We take $E$ a strong coefficient field of $\Pi$ (\Cref{StrongCoefDef}), and $\lambda$ a nonarchimedean place of $E,$ with underlying rational prime $\ell.$ We denote $\O_\lambda$ the ring of integers of $E_\lambda,$ and by $\epsilon_\ell\colon\Gamma_{F^+}\to\Z_\ell^\times\to\O_\lambda^\times$ the $\ell$-adic cyclotomic character.

\subsection{Preliminaries on Galois representations}\label{Proof-Preliminary}
\begin{definition}
For an $\ell$-adic coefficient ring $L$ and a place $v$ of $F,$ we denote $\Mod(F_v,L)$ the category of finitely generated $L$-modules equipped with a continuous action of $\Gamma_{F_v}.$ Similarly, we denote $\Mod(F,L)$ the category of finitely generated $L$-modules equipped with a continuous action of $\Gamma_F$ which is unramified at all but finitely many places.
\end{definition}

For $m,m'\in\Z_{\ge1}\cup\{\infty\}$ with $m\le m',$ we denote $\bar{(\blank)}^{(m)}\colon\Mod(F,\O_\lambda/\lambda^{m'})\to\Mod(F,\O_\lambda/\lambda^m)$ the reduction mod $\lambda^m$ map. We also denote $\bar{(\blank)}=\bar{(\blank)}^{(1)}.$

If $w$ is a place of $F$ and $\RR\in\Mod(F_w,L)$ is a torsion module, we consider the local Tate pairing
\begin{equation*}
    \langle\cdot,\cdot\rangle_w\colon H^1(F_w,\RR)\times H^1(F_w,R^*(1))\xrightarrow{\cup} H^2(F_w,E_\lambda/\O_\lambda(1)) \rightiso E_\lambda/\O_\lambda.
\end{equation*}
\begin{proposition}\label{Tatepairing}
    If $w$ is a place of $F$ and $\RR\in\Mod(F_w,L)$ is a torsion module, the restriction of the local Tate pairing $\langle\cdot,\cdot\rangle_w$ to $H^1_\unr(F_w,\RR)\times H^1_\unr(F_w,\RR^*(1))$ vanishes.
\end{proposition}
\begin{proof}
Denote $\kappa_w$ the residue field of $\O_{F_w},$ and $I_w\subseteq\Gamma_{F_w}$ the inertia subgroup. Recall that $H^1_\unr(F_w,\RR)=H^1(\kappa_w,\RR^{I_w})$ and similarly for $\RR^*(1).$ From the compatibility of cup product under inflation, we conclude that the restriction of $\langle\cdot,\cdot\rangle_w$ to $H^1_\unr(F_w,\RR)\times H^1_\unr(F_w,\RR^*(1))$ factors through $H^2(\kappa_w,\RR^{I_{w}}\otimes \RR^*(1)^{I_{w}}),$ which is the zero group as $\kappa_w$ has cohomological dimension $1.$
\end{proof}

We also record the following consequence of global duality.
\begin{proposition}\label{GlobalDuality}
    Let $L$ be an $\ell$-adic coefficient ring and $\RR\in\Mod(F,L)$ be a torsion module. For any $\alpha\in H^1(F,\RR)$ and $\beta\in H^1(F,\RR^*(1)),$ we have
    \begin{equation*}
        \sum_{w}\langle\mathrm{loc}_w(\alpha),\mathrm{loc}_w(\beta)\rangle_w=0
    \end{equation*}
    where the sum is over all places of $F.$
\end{proposition}
\begin{proof}
    This follows from global duality: if we denote $\mathrm{inv}_w\colon H^2(F_w,E_\lambda/\O_\lambda(1))\rightiso E_\lambda/\O_\lambda,$ then for $\gamma=\alpha\cup\beta\in H^2(F,E_\lambda/\O_\lambda(1))$ we have $\sum_w\mathrm{inv}_w(\mathrm{loc}_w\gamma)=0.$
\end{proof}

\subsubsection{Bloch--Kato Selmer groups and local conditions}

\begin{definition}
We consider the following Bloch--Kato Selmer groups.
\begin{enumerate}
    \item For $\RR\in\Mod(F,E_\lambda),$ we define $H^1_f(F,\RR)\subseteq H^1(F,\RR)$ the $E_\lambda$-submodule consisting of elements $s$ such that for every place $w$ of $F,$ we have $\mathrm{loc}_w(s)\in H^1_f(F_w,\RR).$ Here the local Bloch--Kato conditions are
    \begin{equation*}
        H^1_f(F_w,\RR)\defeq\begin{cases}
            H^1_{\unr}(F_w,\RR)&\text{if }w\nmid\ell,\\
            \ker(H^1(F_w,\RR)\to H^1(F_w,\RR\otimes_{\Q_\ell}\mathbb{B}_{\mathrm{cris}})&\text{if }w\mid\ell.
        \end{cases}
    \end{equation*}
    Here $\mathbb{B}_{\mathrm{cris}}$ denotes Fontaine's crystalline period ring for $\Q_\ell.$
    \item For $\RR\in\Mod(F,\O_\lambda)$ a free module, we define $H^1_f(F,\RR)$ to be the inverse image of $H^1_f(F,\RR_\Q)$ under the natural map $H^1(F,\RR)\to H^1(F,\RR_\Q).$
\end{enumerate}
\end{definition}
\begin{proposition}\label{AJisNS}
    Let $X\in\Sch{F}$ a proper smooth scheme of dimension $2r-1.$ Suppose that a commutative monoid $\mathbb{T}$ acts on $X$ via \'etale correspondences as in \cite[Definition 2.11]{LiuTriple}. Let $\m\subseteq\O_\lambda[\mathbb{T}]$ be a maximal ideal, and denote $\RR\defeq H^{2r-1}_\et(X_{\overline{F}_v},\O_\lambda(r))_\m.$ Assume that $Z\in \mathrm{CH}^r(X)_\m$ is cohomologically trivial, so we may consider $\mathrm{AJ}_\m(Z)\in H^1(F,\RR).$ 
    
    Assume also that $\RR$ is a finite free $\O_\lambda$-module. Let $w$ be a place of $F$ such that $X$ has a smooth model $\mathcal{X}\in\Sch{\O_{F_w}}.$ Then
    \begin{equation*}
        \mathrm{AJ}_\m(Z)\in\begin{cases}
            H^1_{\unr}(F_w,\RR)&\text{if }w\nmid\ell,\\
            H^1_f(F_w,\RR)&\text{if }w\mid\ell.
        \end{cases}
    \end{equation*}
\end{proposition}
\begin{proof}
    This is well known, we include a proof for completeness.
    
    We start with the case $w\nmid \ell.$ We have that $\RR$ and $\RR_\Q=H^{2r-1}_\et(X_{\overline{F}_v},E_\lambda(r))_\m$ are unramified by proper and smooth base change. $\RR$ is also pure of weight $-1$ by the Weil conjectures, and we have $\RR_\Q^\vee(1)\iso R_\Q$ by Poincar\'e duality. Purity implies that $H^0(F_w,\RR_\Q)=0,$ and together with local Tate duality and $\RR_\Q^\vee(1)\iso\RR_\Q,$ this also implies that $H^2(F_w,\RR_\Q)=0.$ Since $w\nmid\ell,$ the local Euler characteristic formula implies that $H^1(F_w,\RR_\Q)=0.$ Now denote $I_w\subseteq\Gamma_{F_w}$ the inertia subgroup. As $\RR$ is unramified, we have an injection $H^1(I_w,\RR)\hookrightarrow H^1(I_w,\RR_\Q)$ and from the commutative diagram
    \begin{equation*}
        \begin{tikzcd}
            H^1(F_w,\RR)\arrow[d]\arrow[r]&H^1(F_w,\RR_\Q)=0\arrow[d]\\
            H^1(I_w,\RR)\arrow[r,hook]&H^1(I_w,\RR_\Q)
        \end{tikzcd}
    \end{equation*}
    we conclude that $H^1_{\unr}(F_w,\RR)=H^1(F_w,\RR).$

    For the case $w\mid\ell,$ we first note that without loss of generality we may assume that $Z\in \mathrm{CH}^r(X)\otimes\O_\lambda$ is cohomologically trivial. Then \cite[Theorem 3.1]{Niziol-AJ} (see \cite[Theorem 3.1(i)]{Nekovar}) tell us that $\mathrm{AJ}(Z)\in H^1_f(F_w,\RR_\Q).$ Since $H^1_f(F_w,\RR)$ is defined to be the pre-image of $H^1_f(F_w,\RR_\Q),$ this proves our claim.
\end{proof}

If $\RR\in\Mod(F,E_\lambda)$ is crystalline at $w\mid\ell,$ we have that $H^1_f(F_w,\RR)\subseteq H^1(F_w,\RR)=\mathrm{Ext}^1_{\Mod(F_w,\Q_\ell)}(\Q_\ell,R)$ corresponds to extensions
\begin{equation*}
    0\to R\to M\to \Q_\ell\to0
\end{equation*}
for which $M$ is also crystalline. For torsion modules, we can make a similar definition.
\begin{definition}[{\cite[Section 4]{Niziol}, \cite[Definition 2.2.5]{LTXZZ}}] Let $w$ be a place of $F$ above $\ell.$
    \begin{enumerate}
        \item A torsion module $M\in\Mod(F_w,\Z_\ell)$ is \emph{crystalline with Hodge--Tate weights in $[a,b]$} if there exist a crystalline representation $V\in\Mod(F_w,\Q_\ell)$ and $\Gamma_{F_w}$-stable lattices $\Lambda\subseteq\Lambda'\subseteq V$ with $M\iso\Lambda/\Lambda'.$
        \item A module $M\in\Mod(F_w,\Z_\ell)$ is \emph{crystalline with Hodge--Tate weights in $[a,b]$} if $M/\ell^mM$ is crystalline for all $m\in\Z_{\ge1}.$
        \item For $\RR\in\Mod(F,\O_\lambda)$ crystalline, we define the $\O_\lambda$-submodule
        \begin{equation*}
            H^1_{ns}(F_w,\RR)\subseteq H^1(F_w,\RR)=\mathrm{Ext}^1_{\Mod(F_w,\Z_\ell)}(\Z_\ell,\RR)
        \end{equation*}
        consisting of extensions $0\to R\to M\to\Z_\ell\to0$ such that $M$ is crystalline.
    \end{enumerate}
\end{definition}
\begin{proposition}\label{fnsComparison}
    Let $w$ be a place of $F$ above $\ell.$ Assume that $\RR\in\Mod(F_w,\O_\lambda)$ is crystalline with Hodge--Tate weights in $[a,b]$ with $b-a\le \ell-2.$ Then under the natural map $H^1(F_w,\RR)\to H^1(F_w,\bar{\RR}^{(m)}),$ the submodule $H^1_f(F_w,\RR)$ lands inside $H^1_{ns}(F_w,\bar{\RR}^{(m)}).$
\end{proposition}
\begin{proof}
    This follows from \cite[Proposition 6]{Breuil} (see also \cite[Lemma 2.2.6]{LTXZZ}) upon noting that $H^1_{ns}(F_w,\RR)$ factors through $H^1_{ns}(F_w,\bar{\RR}^{(m)})$ under the natural map $H^1(F_w,\RR)\to H^1(F_w,\bar{\RR}^{(m)}),$ as subquotients of crystalline modules are crystalline.
\end{proof}
\begin{lemma}[{\cite[Lemma 2.2.7]{LTXZZ}}]\label{nsVanishing}
    Let $w$ be a place of $F$ above $\ell.$ Assume that $\RR\in\Mod(F_w,\O_\lambda)$ is a torsion module which is crystalline with Hodge--Tate weights in $[a,b]$ with $b-a\le \frac{\ell-2}{2}.$ Then we restriction of the local Tate pairing
    \begin{equation*}
        \langle\cdot,\cdot\rangle_w\colon H^1(F_w,\RR)\times H^1(F_w,\RR^*(1))\to E_\lambda/\O_\lambda
    \end{equation*}
    to $H^1_{ns}(F_w,\RR)\times H^1_{ns}(F_w,\RR^*(1))$ lands in $\mathfrak{d}_\lambda^{-1}/\O_\lambda$ where $\mathfrak{d}_\lambda\subseteq\O_\lambda$ is the different ideal of $E_\lambda/\Q_\ell.$
\end{lemma}

\subsubsection{Extensions of conjugate self-dual Galois representations}
We now collect some definitions about extensions of conjugate self-dual Galois representations. This will only be used to state the ``big image'' assumption \Cref{L4-2} and in the proof of \Cref{ChebotarevArg}, where the ``big image'' assumption is used in a Chebotarev argument.
\begin{definition}[{\cite[Section 1]{CHT}}]\label{ExtensionDef}
    We consider the group scheme over $\Z$
    \begin{equation*}
        \mathscr{G}_N\defeq(\mathrm{GL}_N\times\mathrm{GL}_1)\rtimes\{1,\cplx\}
    \end{equation*}
    with $\cplx^2=1$ and $\cplx(g,\mu)\cplx=(\mu g^{\intercal,-1},\mu)$ for $g,\mu\in(\mathrm{GL}_N\times\mathrm{GL}_1).$
\end{definition}
\begin{definition}[{\cite[Section 1]{CHT}}]
    Consider $L$ an $\ell$-adic coefficient ring and $\RR\in\Mod(F,L)$ a free module. We denote the $\Gamma_F$ action by $\rho\colon\Gamma_{F}\to\mathrm{GL}(\RR).$ Assume that $\RR^\cplx\iso \RR^{\vee}(1).$ An \emph{extension} of $\rho$ is a continuous homomorphism
    \begin{equation*}
        \rho_+\colon\Gamma_{F^+}\to\mathscr{G}_N(L)
    \end{equation*}
    as follows. For a choice of identification $\RR\iso L^{\oplus N}$ and a choice $\Xi\colon \RR^\cplx\rightiso \RR^\vee(1)$ with $\Xi^{\cplx,\vee}(1)=(-1)^{\mu_\Xi}\Xi$ for some $\mu_\Xi\in\Z/2\Z,$ we obtain a matrix $B\in\mathrm{GL}_N(L)$ satisfying $\rho^\cplx=B\circ\epsilon_\ell\rho^{\vee}\circ B^{-1}$ and $BB^{\intercal,-1}=(-1)^{\mu_\Xi}.$ We then define $\rho_+$ to be given by
    \begin{equation*}
        \rho_+\rvert_{\Gamma_{F}}=(\rho,\epsilon_\ell\rvert_{\Gamma_{F}})1,\qquad \rho_+(\cplx)=(B,(-1)^{\mu_\Xi})\cplx.
    \end{equation*}
\end{definition}
For an extension $\rho_+\colon\Gamma_{F^+}\to\mathscr{G}_N(\O_\lambda)$ as above, we denote $\bar{\rho}_+^{(m)}\colon\Gamma_{F^+}\to\mathscr{G}_N(\O_\lambda/\lambda^m)$ the reductions of $\rho_+.$

\subsection{Admissible primes}\label{Proof-Admissible}
Recall that by \Cref{GaloisExistence} and \Cref{StrongCoefDef}, we have the Galois representation $\rho_{\Pi,\lambda}(r)\in\Mod(F,E_\lambda)$ of rank $2r,$ which is conjugate self-dual:
\begin{equation*}
    \rho_{\Pi,\lambda}(r)^\cplx\iso \rho_{\Pi,\lambda}^\vee(1-2r)(r)=(\rho_{\Pi,\lambda}(r))^\vee(1).
\end{equation*}

\begin{definition}\label{admissibledef}
    Consider a finite set of nonarchimedean places $\Sigma_{\mathrm{min}}^+$ of $F^+$ containing $\Sigma_\Pi^+$ (\Cref{SigmaPi}).
    We say that the prime $\lambda$ of $E$ is \emph{admissible} with respect to $\Pi$ and $\Sigma_{\mathrm{min}}^+$ if the following conditions hold.
    \begin{enumerate}[leftmargin=*, label=(L\arabic*)]
        \item\label{L1} $\ell\ge 4r+2$ and $\ell$ is unramified in $F.$
        \item\label{L2} $\Sigma_{\mathrm{min}}^+$ contains no $\ell$-adic places.
        \item\label{L3} $\rho_{\Pi,\lambda}$ is residually absolutely irreducible.
        \item\label{L4} Under \Cref{L3}, we have a $\Gamma_{F}$-stable $\O_\lambda$-lattice $\RR$ in $\rho_{\Pi,\lambda}(r),$ unique up to homothety, which satisfies $\RR^
        \cplx\iso \RR^\vee(1).$ We assume the following.
        \begin{enumerate}[label=(L4-\arabic*)]
            \item\label{L4-1} At least one of the following holds:
            \begin{enumerate}[label=(L4-1\alph*)]
                \item The image of $\Gamma_{F}$ in $\mathrm{GL}(\overline{\RR})$ contains a nontrivial scalar element, or
                \item $\dim_{\O_\lambda/\lambda}\bar{\RR}\le\min\{\frac{\ell-1}{2},\ell-3\},$\footnote{This first condition follows from \Cref{L1} since $\dim_{\O_\lambda/\lambda}\bar{\RR}=2r.$} $\bar{\RR}$ is a semisimple $(\O_\lambda/\lambda)[\Gamma_F]$-module and we have that $\Hom_{(\O_\lambda/\lambda)[\Gamma_F]}(\mathrm{End}(\overline{\RR}),\RR)=0.$
            \end{enumerate}
            \item\label{L4-2} Choose an extension $\rho_{\Pi,\lambda,+}$ of $\RR$ (see \Cref{ExtensionDef}). Then the  image of the restriction $(\bar{\rho}_{\Pi,\lambda,+},\bar{\epsilon}_\ell)\colon\Gamma_{F_{\rflx}^+}\to\mathscr{G}_{2r}(\O_\lambda/\lambda)\times(\O_\lambda/\lambda)^\times$ contains an element $(\gamma,\xi)$ satisfying
                \begin{enumerate}[label=(L4-2\alph*)]
                    \item $\xi^2-1\neq0$ in $\O_\lambda/\lambda$;
                    \item $\gamma$ belongs to $(\mathrm{GL}_{2r}(\O_\lambda/\lambda)\times(\O_\lambda/\lambda)^\times)\mathbf{c}$ with order coprime to $\ell$;
                    \item letting $h_\gamma\in\mathrm{GL}_N(\O_\lambda/\lambda)$ be the first component of $\gamma^2\in\mathrm{GL}_N(\O_\lambda/\lambda)\times(\O_\lambda/\lambda)^\times,$ then $1$ appears in the eigenvalues of $h_\gamma$ with multiplicity one, and $h_\gamma$ does not have an eigenvalue that is equal to $-1$ in $\O_\lambda/\lambda.$
                \end{enumerate}
        \end{enumerate}
        \item\label{L5} Under \Cref{L3} and the notation in \Cref{L4-2}, the homomorphism $\bar{\rho}_{\Pi,\lambda,+}$ is rigid for $(\Sigma_{\mathrm{min}}^+,\emptyset),$ and $\bar{\RR}\rvert_{\Gamma_{F(\zeta_l)}}$ is absolutely irreducible;
        \item\label{L6} the composite morphism $\mathbb{T}^{\Sigma_{\mathrm{min}}^+}\xrightarrow{\phi_\Pi}\O_\lambda\to\O_\lambda/\lambda^m$ is cohomologically generic (\cite[Definition D.1.1]{LTXZZ}).
    \end{enumerate}
\end{definition}

\begin{remark}\label{admissibleDiscussion}
As mentioned in \Cref{IntroductionRemark}, we expect that all but finitely many primes are admissible as long as $\Pi$ is not a transfer from a smaller group. We have the following results about the above assumptions.
\begin{enumerate}[label=(\alph*)]
    \item If $E=\Q$ and that there is an elliptic curve $A$ over $F^+$ such that for every rational prime $\ell$ of $E$ we have $\rho_{\Pi,\ell}\iso\mathrm{Sym}^{2r-1}H^1_\et(A_{\overline{F}},\Q_\ell)\rvert_{\Gamma_F},$ then as long as $\mathrm{End}(A_{\overline{F}})=\Z,$ we have that \Cref{L3}, \Cref{L4} and \Cref{L5} holds for all but finitely many primes by (the proof of) \cite[Lemmas 8.1.3]{LTXZZ}\footnote{Whose proof carries through (and become simpler) by taking $n_1=1.$}.
    \item If there exists a nonarchimedean place $w$ of $F$ where $\Pi_w$ is supercuspidal, then \Cref{L3} and \Cref{L5} hold for all but finitely many primes by \cite[Theorem 4.2.6]{LTXZZ2}.
    \item If there exists a very special inert prime $\p$ (\Cref{DefSpecialInert}) where $\Pi_\p$ is Steinberg, then \Cref{L4} holds for all but finitely many primes by (the proof of) \cite[Lemma 8.1.4]{LTXZZ}\footnote{Whose proof carries through (and become simpler) by treating $\rho_{\Pi_1,\ell}$ to be the trivial representation.}.
    \item If $F^+\neq\Q,$ \Cref{L6} holds for all but finitely many primes by \cite[Corollary D.1.4]{LTXZZ}, which is an extension of the results of Caraiani--Scholze \cite{Caraiani-Scholze}.
\end{enumerate}
\end{remark}

\subsection{Euler system argument}\label{Proof-Euler}
We first recall and collect some of the considerations in \cite[Section 2]{LTXZZ} as applied to our setting.
\begin{lemma}\label{galoislemmas}\label{ChebotarevArg}
Assume \Cref{L3}, under which we denote $\RR$ to be a $\Gamma_F$-stable lattice of $\rho_{\Pi,\lambda}(r).$ Let $S$ be a free $\O_\lambda$-submodule of $H^1_f(F,\RR)$ of rank $r_S$ whose image in $H^1_f(F,\RR)/H^1_f(F,\RR)_{tor}$ is saturated. Let $m\ge1$ be a positive integer. Then:
\begin{enumerate}
    \item The image $S^{(m)}$ of $S$ in $H^1_{f,\RR}(F,\bar{\RR}^{(m)})$ is a free $\O_\lambda/\lambda^m$ submodule of rank $r_S.$
    \item For a given finite set of places $\Sigma$ of $F,$ there is an integer $m_\Sigma$ depending only on $\Sigma$ and $\RR$ such that $\mathrm{loc}_w(\lambda^{m_\Sigma}S^{(m)})=0$ for every nonarchimedean place $w\in\Sigma$ not above $\ell.$
    \item If \Cref{L4} also holds, then there is an integer $\mathbf{r}_{\RR}$ that depends only on $\RR$ such that the following holds: one can choose a basis $s_1,\ldots,s_{r_S}$ of $S^{(m)}$ and very special inert places $\p_1,\ldots,\p_{r_S}$ of $F^+$ such that:
    \begin{enumerate}
        \item $\p_1,\ldots,\p_{r_S}$ satisfy \Cref{P1,P2,P3,P4,P5,P6}.
        \item $\mathrm{loc}_{\p_i}(s_j)=0$ if $i\neq j,$ and
        \begin{equation*}
            \exp_\lambda(\mathrm{loc}_{\p_i}(s_i),H^1_{\unr}(F,\bar{\RR}^{(m)}))\ge m-(3\cdot 2^{r_S-1}-2)\mathbf{r}_{\RR}.
        \end{equation*}
    \end{enumerate}
\end{enumerate}
\end{lemma}
\begin{proof}
(1) and (2) follow from \cite[Proposition 2.4.6]{LTXZZ}. Note that $\RR_\Q$ is pure of weight $-1$ at every nonarchimedean place above $\ell$ by \cite[Lemma 1.4(3)]{Taylor-Yoshida} and \Cref{GaloisExistence}.

For a totally real finite Galois extension $F'/F^+$ contained in $\C,$ a polynomial $\mathscr{P}\in\Z[T]$ and a positive integer $m,$ we consider the following condition, which we denote $(\mathrm{GI}^m_{F',\mathscr{P}})$: the image of the restriction of the homomorphism $(\bar{\rho}^{(m)}_{\Pi,\lambda,+},\bar{\epsilon}_\ell^{(m)})\colon\Gamma_F\to\mathscr{G}_{2r}(\O_\lambda/\lambda^m)\times(\O_\lambda/\lambda)^\times$ to $\Gamma_{F'}$ contains an element $(\gamma,\xi)$ satisfying
\begin{enumerate}[label=(\alph*)]
\item $\mathscr{P}(\xi)$ is invertible in $\O_\lambda/\lambda^m$;
\item $\gamma$ belongs to $(\mathrm{GL}_{2r}(\O_\lambda/\lambda^m)\times(\O_\lambda/\lambda^m))^\times)\mathbf{c}$ with order prime to $\ell$;
\item letting $h_\gamma\in\mathrm{GL}_{2r}(\O_\lambda/\lambda)$ be the first component of $\gamma^2\in\mathrm{GL}_{2r}(\O_\lambda/\lambda)\times(\O_\lambda/\lambda)^\times,$ then the kernel of $h_\gamma-1$ is free of rank $1$ over $\O_\lambda/\lambda^m,$ and $h_\gamma$ does not have an eigenvalue that is equal to $-1$ in $\O_\lambda/\lambda.$
\end{enumerate}
Similarly to \cite[Lemma 2.7.1]{LTXZZ}, we have that $(\mathrm{GI}^1_{F',\mathscr{P}})$ implies $(\mathrm{GI}^m_{F',\mathscr{P}})$ for all $m\in\Z_{\ge1}.$ Thus, by \ref{L4-2} we have that $(\mathrm{GI}^m_{F^+_\rflx,(T^2-1)})$ holds for all $m.$ 

Now (3), follows from the Chebotarev density theorem together with \ref{L4-1} and \cite[Propositions 2.6.6, 2.6.7]{LTXZZ}\footnote{Where we take $m_0=0$ in the cited propositions.}.
\end{proof}

We are now ready to prove our main theorem \Cref{ThmAIntro}, which we state in a slightly more general form than in the introduction.
\begin{theorem}\label{ThmA}
    Assume \Cref{LR0} for $N=2r.$ Take an orthogonal decomposition $V^\circ=V_1^\circ\obot V_2^\circ.$ We denote $H^\circ\defeq U(V_1^\circ)\times U(V_2^\circ).$ We assume that there is $f\in\pi$ such that
    \begin{equation*}
        \int_{H^\circ(F^+)\backslash H^\circ(\A_{F^+})}f(h)\d{h}\neq0.
    \end{equation*}
    
    Let $K^\circ\in\mathfrak{K}(V^\circ)$ be such that $f\in\pi^{K^\circ}$ and such that it is of the form
    \begin{equation*}
        K^\circ=\prod_{v\in\Sigma_{\mathrm{min}}^+}K^\circ_v\times\prod_{v\not\in\Sigma_{\mathrm{min}}^+\cup\Sigma_\infty^+}U(\Lambda^\circ)(\O_{F^+_v})
    \end{equation*}
    for some finite set of nonarchimedean places $\Sigma_{\mathrm{min}}^+\supseteq\Sigma_{\Pi}^+$ and a self-dual lattice $\Lambda^\circ\subseteq V^\circ\otimes_{F^+}\A_{F^+}^{\Sigma_\infty^+\cup\Sigma_{\mathrm{min}}^+}.$
    
    Then if $\lambda$ if admissible with respect to $\Pi$ and $\Sigma_{\mathrm{min}}^+,$ we have $H^1_f(F,\rho_{\Pi,\lambda}(r))=0.$
\end{theorem}
\begin{remark}
    If such an $f$ exists, one expects to be able to choose $f,K^\circ$ as above with $\Sigma_{\mathrm{min}}^+=\Sigma_\Pi^+.$ If this is the case, then the admissibility condition of $\lambda$ would not depend on the particular choice of $f.$
\end{remark}
\begin{proof}[Proof of \Cref{ThmA}]
    First we note that we may assume that $K^\circ_v$ for $v\in\Sigma_{\mathrm{min}}^+$ are transferable by \cite[Lemma D.2.2(3)]{LTXZZ}, since we may shrink them as necessary.
    
    We may also assume without loss of generality that $f$ is an element
    \begin{equation*}
        f\in\O_E[\Sh(V^\circ,K^\circ)][\ker\phi_\Pi]
    \end{equation*}
    without changing $K^\circ.$ To see this, we have $f\in\C[\Sh(V^\circ,K^\circ)[\ker\phi_\Pi],$ and so we may write $f$ as a finite sum $f=\sum a_if_i$ for some $f_i\in E[\Sh(V^\circ,K^\circ)]$ with $a_i\in\C^\times$ linearly independent over $E.$ Note that we must have $f_i\in E[\Sh(V^\circ,K^\circ)[\ker\phi_\Pi]$ for all $i,$ since $\phi_\Pi$ takes values in $E.$ Then $\int_{[H]}f_i(h)\d{h}$ must be nonzero for some $i,$ and we may scale such $f_i$ so that $f_i\in\O_E[\Sh(V^\circ,K^\circ)][\ker\phi_\Pi].$

    Assume that $\lambda$ is admissible for $\Pi$ and $\Sigma_{\mathrm{min}}^+.$ Under \Cref{L3}, we denote $\RR$ to a Galois invariant lattice of $\rho_{\Pi,\lambda}(r),$ which is unique up to homothety.
    
    Denote $\Sigma$ to be the set of places of $F$ above $\Sigma_{min}^+$ or above $\ell.$ We consider the following quantities:
    \begin{enumerate}
        \item $m_{per}$ is the largest non-negative integer such that
        \begin{equation*}
            \frac{1}{\mathrm{vol}_{H^\circ}(H^\circ(\A_{F^+}^\infty)\cap K^\circ)}\int_{H^\circ(F^+)\backslash H^\circ(\A_{F^+}^\infty)}f'(h)\d{h}\in\lambda^{m_{per}}\O_E
        \end{equation*}
        for all $f'\in\O_E[\mathrm{Sh}(V^\circ,K^\circ)][\ker\phi_\Pi].$
        \item $m_\Sigma$ and $\mathbf{r}_{\RR}$ are the integers in \Cref{galoislemmas} for our choice of $\Sigma.$
        \item $m_{dif}$ is such that the different ideal $\mathfrak{d}_\lambda$ of $E_\lambda$ over $\Q_\ell$ is $\mathfrak{d}_\lambda=\lambda^{m_{dif}}.$
    \end{enumerate}

    Assume by contradiction that we have
    \begin{equation*}
        \dim_{E_\lambda}H^1_f(F,\rho_{\Pi,\lambda}(r))\ge1,
    \end{equation*}
    so that we may choose $S\subseteq H^1_f(F,\RR)$ free of rank $1$ over $\O_\lambda$ such that its image in $H^1_f(F,\RR)/H^1_f(F,\RR)_{tor}$ is saturated. We denote by $s$ a generator of $S.$
    
    We let $m$ be a positive integer such that
    \begin{equation*}
        m>m_{per}+\mathbf{r}_{\RR}+\max(m_\Sigma,m_{dif}).
    \end{equation*}
    By \Cref{galoislemmas}(3), we may choose a very special inert prime $\p$ of $F^+$ satisfying \Cref{P1,P2,P3,P4,P5,P6} such that
    \begin{enumerate}[label=(S)]
        \item\label{S} $\exp_\lambda(\mathrm{loc}_\p(s),H^1_{\unr}(F_\p,\bar{\RR}^{(m)}))\ge m-\mathbf{r}_{\RR}.$
    \end{enumerate}

    Now we will obtain a contradiction by constructing a class $c\in H^1(F,\bar{\RR}^{(m)\cplx})$ satisfying
    \begin{enumerate}[label=(C\arabic*)]
        \item\label{C1} $\mathrm{loc}_wc\in H^1_{\unr}(F_w,\bar{\RR}^{(m)\mathbf{c}})$ for all $w\not\in\Sigma\cup\{\p\}.$
        \item\label{C2} $\exp_\lambda(\partial_\p\mathrm{loc}_\p c,H^1_{\sing}(F_\p,\bar{\RR}^{(m)\mathbf{c}}))\ge m-m_{per}.$
        \item\label{C3} $\mathrm{loc}_wc\in H^1_{ns}(F_w,\bar{\RR}^{(m)\cplx})$ for all $w\mid\ell.$
    \end{enumerate}
    We see that this yields a contradiction with $\sum_w\langle \mathrm{loc}_ws,\mathrm{loc}_wc\rangle_w=0$ (\Cref{GlobalDuality}). This is because of the following.
    \begin{itemize}[leftmargin=*]
        \item For $w\not\in\Sigma,$ both $\mathrm{loc}_ws$ and $\mathrm{loc}_wc$ are unramified by \ref{C1}, and thus $\langle \mathrm{loc}_ws,\mathrm{loc}_wc\rangle_w=0$ by \Cref{Tatepairing}.
        \item For $w$ above $\Sigma^+_{min},$ we have $\lambda^{m_\Sigma}\mathrm{loc}_w(s)=0$ by \ref{L2} and \Cref{galoislemmas}(2). Thus
        \begin{equation*}
            \exp_\lambda(\langle \mathrm{loc}_ws,\mathrm{loc}_wc\rangle_w,\ \O_\lambda/\lambda^m)\le m_\Sigma.
        \end{equation*}
        \item For $w\mid\ell,$ by \ref{L2} we have that $\RR_\Q$ is crystalline with Hodge--Tate weights in $[-r,r-1],$ and thus $\mathrm{loc}_w(s)\in H^1_{ns}(F_w,\bar{\RR}^{(m)})$ by \ref{L1} and \Cref{fnsComparison}. By \ref{L1}, \ref{C3} and \Cref{nsVanishing}, we have 
        \begin{equation*}
            \exp_\lambda(\langle s,c\rangle_w,\ \O_\lambda/\lambda^m)\le m_{dif}.
        \end{equation*}
        \item For $w$ the unique place above $\p,$ we first note that $H^1_{\sing}(F_\p,\bar{\RR}^{(m)\cplx})$ is free of rank $1$ over $\O_\lambda/\lambda^m.$ The module $\bar{\RR}^{(m)\cplx}$ is unramified for $\Gamma_{F_\p},$ and so
        \begin{equation*}
            H^1_{\sing}(F_\p,\bar{\RR}^{(m)\cplx})=(\bar{\RR}^{(m)\cplx})/(\mathrm{Frob}_\p-1).
        \end{equation*}
        but by \Cref{GaloisExistence}, the generalized eigenvalues of $\mathrm{Frob}_\p$ in $\bar{\RR}^\cplx$ are $\{p\alpha_1^{\pm1},\ldots,p\alpha_r^{\pm1}\}$ where $\alpha_1^{\pm1},\ldots,\alpha_r^{\pm1}$ are the Satake parameters of $\Pi_\p.$ Thus $(\bar{\RR}^{(m)\cplx})/(\mathrm{Frob}_\p-1)$ is free of rank $1$ over $\O_\lambda/\lambda^m$ by \Cref{P4}. Now \Cref{C2} and \Cref{S} imply
        \begin{equation*}
            \exp_\lambda(\langle \mathrm{loc}_ws,\mathrm{loc}_wc\rangle_w,\ \O_\lambda/\lambda^m)\ge m-m_{per}-\mathbf{r}_{\RR}.
        \end{equation*}
    \end{itemize}
    By our choice of $m,$ this contradicts that $\sum_w\langle \mathrm{loc}_ws,\mathrm{loc}_wc\rangle_w=0.$
    
    Now we construct the class $c\in H^1(F,\bar{\RR}^{(m)\cplx})$ satisfying \Cref{C1,C2,C3}. We choose the additional data to apply the discussion of \Cref{Reciprocity-ALR}. Namely, we consider the following.
    \begin{itemize}[leftmargin=*]
        \item We take $\Sigma_{\mathrm{\lr}}^+=\emptyset$ and $\Sigma^+=\Sigma_{\mathrm{min}}^+.$
        \item We choose data of $W_0,K_0^p$ as in \Cref{Shimura-RSZ}.
        \item We choose a standard indefinite hermitian space $V_1$ over $F/F^+$ of rank $r$ equipped with an isometry $j\colon V_1\otimes_{F^+}\A_{F^+}^{\infty,\p}\rightiso V_1^\circ\otimes_{F^+}\A_{F^+}^{\infty,\p}.$ We then let $V=V_1\obot V_2^\circ,$ and choose an almost self-dual lattice $\Lambda_\p\subseteq V\otimes_{F^+}F^+_\p.$ We denote $K\defeq j^*(K^{\circ,\p})\mathrm{Stab}(\Lambda_\p).$
    \end{itemize}
    We note that \Cref{LR1,LR2,LR3,LR4,LR5} are satisfied by \Cref{L1,L2,L3,L5,L6}, and \Cref{FJ1} is satisfied by the construction of $V.$ Now \Cref{Cohomologymodp}(6) and \Cref{nonproperComparison} give us an isomorphism
    \begin{equation*}
        \Upsilon\colon H^{2r-1}_\et(\mathrm{Sh}(V,K)_{\overline{F}},\O_\lambda(r))/\mathfrak{n}\rightiso(\bar{\RR}^{(m)\mathbf{c}})^{\oplus\mu}
    \end{equation*}
    of $\O_\lambda[\Gamma_F]$-modules, for some $\mu>0.$

    \Cref{reciprocity} affords us a class $Z\in \mathrm{CH}^r(\Sh(V,K))_\m$ such that
    \begin{equation*}
    \begin{split}
        &\exp_\lambda\left(\partial_\p\mathrm{loc}_\p \mathrm{AJ}_\m Z,\ H^1_{\sing}(F_\p,H^{2r-1}_\et(\mathrm{Sh}(V,K)_{\overline{F}},\O_\lambda(r))/\n)\right)\\
        &\qquad=\exp_\lambda(\mathbbm{1}_{H^\circ},\ \O_\lambda[\mathrm{Sh}(V^\circ,K^\circ)]/\mathfrak{n}).
    \end{split}
    \end{equation*}
    In particular, by the definition of $m_{per}$ we have
    \begin{equation*}
        \exp_\lambda(\partial_\p\mathrm{loc}_\p\Upsilon(\mathrm{AJ}_\m Z),\ H^1_{\sing}(F_\p,(\bar{\RR}^{(m)\mathbf{c}})^{\oplus\mu}))\ge m-m_{per}.
    \end{equation*}
    Finally, we take $c$ to be a projection of $\Upsilon(\mathrm{AJ}_\m Z)$ such that $\exp_\lambda(\partial_\p\mathrm{loc}_\p c,\ H^1_{\sing}(F_\p,\bar{\RR}^{(m)\mathbf{c}}))\ge m-m_{per}.$ This satisfies \Cref{C2} by the above, and satisfies \Cref{C1,C3} by \Cref{AJisNS}\footnote{In the non-proper case (e.g. $F^+=\Q$), we need apply this proposition for a toroidal compactification of $\Sh(V,K)$ instead. See \cite[Remark 5.41]{Stroh-Lan} for how the Hecke action extends to such compactifications.} and \Cref{fnsComparison}.
\end{proof}

\begin{proof}[Proof of {\Cref{ThmB}}]
    Under the assumptions of the theorem, \cite[Theorem 1.1, Remarks 1.1 and 1.2]{LXZ} imply that there are nondegenerate Hermitian spaces $V_1,V_2$ of rank $r$ and an irreducible cuspidal automorphic representation $\pi'$ of $U(V_1\oplus V_2)(\mathbb{A}_{F^+})$ which is $H=U(V_1)\times U(V_2)$-distinguished such that $\Pi=\mathrm{BC}(\pi').$ Note that from our assumptions on the archimedean components of $\pi,$ $V_1,V_2$ can be taken to be totally positive definite, and $\pi'_v$ is trivial for archimedean places $v.$ Thus the claim follows from \Cref{ThmA} applied to $\pi'.$
\end{proof}

\begin{appendix}
\section{\texorpdfstring{$K$}{K}-theory}\label{KAppendix}
We briefly collect some definitions about $K_0$ and $G_0$ of (formal) schemes, following mostly \cite{Gillet-Soule}\footnote{Although we use the notation $G_0$ in place of $K_0'.$}. We then prove a result used in \Cref{WeightChapter}.
\subsection{\texorpdfstring{$K_0$}{K0} and \texorpdfstring{$G_0$}{G0}}
We first consider $X,X'$ noetherian schemes. We will let $Y,Z\subseteq X$ resp. $Y',Z'\subseteq X'$ denote closed subschemes.
\begin{definition}\label{KDef}
    We let $G_0(X)$ be the $G$-theory of $X,$ where $G_0(X)$ is the Grothendieck group of finite complexes of coherent sheaves of $\O_X$-modules. We let $K_0^Y(X)$ be the $K$-theory of $X$ with supports, where $K_0^Y(X)$ is the Grothendieck group of finite complexes of coherent locally free $\O_X$-modules which are acyclic outside $Y.$ They become rings under the products
    \begin{equation*}
    \begin{split}
        G_0(X)\times G_0(X)&\to G_0(X)\\
        (\mathcal{F}_1,\mathcal{F}_2)&\mapsto \mathcal{F}_1\otimes^{\L}\mathcal{F}_2
    \end{split}
    \end{equation*}
    and
    \begin{equation*}
    \begin{split}
        K_0^Y(X)\times K_0^Z(X)&\to K_0^{Y\cap Z}(X)\\
        (\mathcal{F}_1,\mathcal{F}_2)&\mapsto \mathcal{F}_1\otimes\mathcal{F}_2
    \end{split}
    \end{equation*}
\end{definition}
\begin{definition}[{\cite[Section 1]{Gillet-Soule}}]
    If $f\colon X'\to X$ is a morphism and if $Y'=f^{-1}(Y)$, we consider the pullback $f^*\colon K_0^Y(X)\to K_0^{Y'}(X').$ If $f$ is proper, we consider the pushforward $f_*\colon G_0(X')\to G_0(X)$ given by
    \begin{equation*}
        f_*([\mathcal{F}])=\sum(-1)^i[R^if_*\mathcal{F}].
    \end{equation*}
    If $f$ is flat, we also consider the pullback map $f^*\colon G_0(X)\to G_0(X')$ given by $f^*([\mathcal{F}])=\mathcal{F}\otimes_{\O_X}\O_{X'}.$
\end{definition}
\begin{proposition}[{\cite[Lemma 1.9]{Gillet-Soule}}]\label{K=G}
    We have a natural map $K_0^Y(X)\to G_0(Y)$ given by
    \begin{equation*}
        [\mathcal{F}]\mapsto\sum_i(-1)^i[H^i(\mathcal{F})],
    \end{equation*}
    which is compatible with the product structures and with flat pullback. If $X$ is regular, then this is an isomorphism.
\end{proposition}

In particular, if $X,X'$ are regular and $f\colon X'\to X$ is proper, we also have a pushforward map $f_*\colon K_0^{Y'}(X')\to K_0^{f(Y')}(X).$

\begin{definition}\label{GrDef}
    We consider the coniveau (or codimension) descending filtration
    \begin{equation*}
        F^iK_0^Y(X)\defeq\bigcup_{\substack{Z\subseteq Y\\\mathrm{codim}_XZ\ge i}}\im(K_0^Z(X)\to K_0^Y(X))
    \end{equation*}
    and the dimension ascending filtration
    \begin{equation*}
        F_iG_0(X)\defeq\bigcup_{\substack{Z\subseteq Y\\\mathrm{dim}(Z)\le i}}\im(G_0(Z)\to G_0(X)).
    \end{equation*}
    We denote by $\mathrm{Gr}^iK_0^Y(X)$ and $\mathrm{Gr}_iG_0(X)$ their corresponding associated graded groups.
\end{definition}

We note that if $f$ is flat, then $f^*$ respects the filtration on $K_0.$ If $f$ is proper, then $f_*$ respects the filtration on $G_0.$ If $X$ is regular of pure dimension $n,$ then \Cref{K=G} is also compatible with the filtrations in the sense that $F^{n-i}K_0^Y(X)\rightiso F_iG_0(Y).$

\begin{remark}
    Following \cite[Appendix B]{Zhang-AFL} one can extend these definitions to locally noetherian formal schemes $X.$
\end{remark}

It is not immediate that the coniveau filtration is compatible with the product structures. As pointed out in \cite[Equation (B.3)]{Zhang-AFL}, the following theorem should also hold for formal schemes $X.$
\begin{theorem}[{\cite[Proposition 5.5]{Gillet-Soule}}]
If $X$ is a regular scheme, we have $F^iK_0^Y(X)_\Q\cdot F^jK_0^Z(X)_\Q\subseteq F^{i+j}K_0^{Y\cap Z}(X)_\Q.$
\end{theorem}

Lastly, we record the following relation between $G_0(X)$ and $\mathrm{CH}^*(X)$ in the case that $X$ is a scheme.
\begin{proposition}[{\cite[Lemma 33, Theorem 34]{Gillet}}]\label{CHtoGrDef}
    Assume $X$ is of finite type over a field $k,$ and let $\ell\neq\mathrm{char}(k)$ be a rational prime. Then there is a surjective map $\O_{(\blank)}\colon\mathrm{CH}^i(X)\to\mathrm{Gr}^iG_0(X)$ and a cycle class map $\mathrm{cl}_{X,\ell}\colon\mathrm{Gr}^iG_0(X)\to H^{2i}_\et(X,\Z_\ell(i))$ such that their composition is the usual cycle class map for $\mathrm{CH}^i(X).$
\end{proposition}

\subsection{A few lemmas}
\begin{lemma}\label{K-incexc}
    Let $X$ be a noetherian scheme. Take a decomposition $X=Y_1\cup\cdots\cup Y_n$ into closed subschemes. Then the map
    \begin{equation*}
        \bigoplus_{j=1}^n\mathrm{Gr}_iG_0(Y_j)\to \mathrm{Gr}_iG_0(X)
    \end{equation*}
    is surjective for every $i.$
\end{lemma}
\begin{proof}
By induction on $n,$ it suffices to prove the case $n=2.$ Let $I_1,I_2$ be the ideal sheaves of $Y_1,Y_2,$ and $i_1\colon Y_1\to X,$ $i_2\colon Y_2\to X$ and $i\colon Y_1\cap Y_2\to X$ the closed immersions. Since $I_1\cap I_2=0,$ we have an exact sequence
    \begin{equation*}
        0\to\O_X\to\O_X/I_1\oplus\O_X/I_2\to\O_X/(I_1+I_2)\to0
    \end{equation*}
    of sheaves on $X.$ Let $\mathcal{F}$ be a coherent $\O_X$-module. Then we get an exact sequence
    \begin{equation*}
    \begin{split}
        &0\to\mathrm{Tor}_1^{\O_X}(\mathcal{F},\O_{Y_1})\oplus \mathrm{Tor}_1^{\O_X}(\mathcal{F},\O_{Y_2})\to\mathrm{Tor}_1^{\O_X}(\mathcal{F},\O_{Y_1\cap Y_2})\to\mathcal{F}\to\\
        &\qquad\to i_{1,*}(\mathcal{F}\rvert_{Y_1})\oplus i_{2,*}(\mathcal{F}\rvert_{Y_2})\to i_*(\mathcal{F}\rvert_{Y_1\cap Y_2})\to0,
    \end{split}
    \end{equation*}
    which expresses $[\mathcal{F}]\in G_0(X)$ as a combination of elements in $G_0(Y_1),G_0(Y_2),G_0(Y_1\cap Y_2).$ Thus $G_0(Y_1)\oplus G_0(Y_2)\to G_0(X)$ is surjective. The claim with the filtration is a consequence of this, as for any $Z\subseteq X$ closed subscheme of dimension $\le i,$ we have the commutative diagram
    \begin{equation*}
    \begin{tikzcd}
        G_0(Y_1\cap Z)\oplus G_0(Y_2\cap Z)\arrow[r,twoheadrightarrow]\arrow[d]&G_0(Z)\arrow[d]\\
        G_0(Y_1)\oplus G_0(Y_2)\arrow[r]&G_0(X)
    \end{tikzcd}
    \end{equation*}
    and thus $F_iG_0(Y_1)\oplus F_iG_0(Y_2)\to F_iG_0(X)$ is also surjective for all $i.$ This implies the claim.
\end{proof}

\begin{lemma}
    Assume that $\mathcal{X}$ is a regular, flat, noetherian scheme over a discrete valuation ring $R.$ Denote by $X^\eta$ the generic fiber of $\mathcal{X}$ and by $X$ its special fiber. Then for every $i$ we have an exact sequence
    \begin{equation*}
        \mathrm{Gr}_{i+1}G_0(X)\to\mathrm{Gr}_{i+1}G_0(\mathcal{X})\to\mathrm{Gr}_iG_0(X^\eta)\to0.
    \end{equation*}
\end{lemma}
\begin{proof}
    It suffices to prove that we have an exact sequences
    \begin{equation*}
        F_{i+1}G_0(X)\to F_{i+1}G_0(\mathcal{X})\to F_iG_0(X^\eta)\to0
    \end{equation*}
    for each $i.$ Note that the image of $F_{i+1}G_0(\mathcal{X})\to G_0(X^\eta)$ lands is $F_iG_0(X^\eta)$: if $\mathcal{Z}\subseteq\mathcal{X}$ has dimension $\le i+1,$ then $\mathcal{Z}\cap X^\eta$ has dimension $\le i$ by flatness, and we have the following commutative diagram.
    \begin{equation*}
        \begin{tikzcd}
            G_0(\mathcal{X})\arrow[r]&G_0(X^\eta)\\
            G_0(\mathcal{Z})\arrow[u]\arrow[r]&G_0(\mathcal{Z}\cap X^\eta)\arrow[u]
        \end{tikzcd}
    \end{equation*}

    We follow the proof of the localization sequence from Borel--Serre \cite[Proposition 7]{Borel-Serre}\footnote{In there, they work with varieties over algebraically closed fields, but the proposition we are interested in does not require this assumption: note that the analogue of \cite[Proposition 1]{Borel-Serre} holds in our setting by the assumption that $\mathcal{X}$ is noetherian.}. Denote $A=F_{i+1}G_0(\mathcal{X})/\mathrm{im}(F_{i+1}G_0(X))$ The above implies we have a map $A\to F_i G_0(X^\eta),$ and we will construct an inverse. If $\mathcal{F}$ is a coherent sheaf on $X^\eta$ supported on a closed subscheme $Z^\eta\subseteq X^\eta$ of dimension $\le i,$ let $\mathcal{Z}$ be any closed subscheme of $\mathcal{X}$ of dimension $\le i+1$ with generic fiber $Z^\eta.$ Note such a $\mathcal{Z}$ exists since we can take the scheme theoretic closure of $Z^\eta.$ Then $\mathcal{F}$ extends to a sheaf $\mathcal{G}$ on $\mathcal{Z}$ with $\mathcal{G}\rvert_{Z^\eta}=\mathcal{F},$ and if $Z$ denotes the special fiber of $\mathcal{Z},$ the proof of \cite[Proposition 7]{Borel-Serre} says the image of $[\mathcal{G}]$ on $G_0(\mathcal{Z})/\mathrm{im}(G_0(Z))$ is independent of the choice of this extension. Thus its image in $A$ is also independent of this choice. Note that this construction is also independent of the choice of $Z^\eta$ and $\mathcal{Z},$ since if $\mathcal{Z}_1,\mathcal{Z}_2$ are two choices of support with dimension $\le i+1,$ we can always take another $\mathcal{Z}$ of dimension $\le i+1$ that contains both $\mathcal{Z}_1$ and $\mathcal{Z}_2.$ Now if $0\to\mathcal{F}'\to\mathcal{F}\to\mathcal{F}''\to0$ where each sheaf is supported on a closed subscheme of $X^\eta$ of dimension $\le i,$ we can choose a single $Z^\eta$ of dimension $\le i$ as their support, and again the proof of \cite[Proposition 7]{Borel-Serre} (taking the ambient scheme to be the scheme theoretic closure $\mathcal{Z}$ of $Z^\eta$) proves that the above construction induces a group homomorphism $F_iG_0(X^\eta)\to A.$ This is evidently an inverse to $A\to F_iG_0(X^\eta).$
\end{proof}

\begin{corollary}\label{KLemma}
    Assume that $\mathcal{X}$ is a regular, flat, noetherian scheme over a discrete valuation ring $R.$ Assume $\mathcal{X}$ has pure dimension $n.$ Denote $X$ the special fiber of $\mathcal{X},$ and say $Y_1,\ldots,Y_n$ are its irreducible components. We assume $Y_1,\ldots,Y_n$ are all of regular of dimension $n-1.$ Denote $X^\eta$ the generic fiber of $\mathcal{X}.$ Then we have an exact sequence
    \begin{equation*}
        \bigoplus_{j=1}^n\mathrm{Gr}^{i-1}K_0(Y_j)\to\mathrm{Gr}^iK_0(\mathcal{X})\to\mathrm{Gr}^iK_0(X^\eta)\to0
    \end{equation*}
\end{corollary}
\begin{proof}
    This follows from combining the previous two lemmas together with \cref{K=G}, since all of $Y_1,\ldots,Y_n,\mathcal{X},X$ are regular.
\end{proof}

\end{appendix}

\begingroup
\bibliographystyle{alpha}
\bibliography{structural/references}
\endgroup
\end{document}